\documentclass[12pt,reqno]{amsart}

\usepackage[T1]{fontenc}
\pdfoutput=1
\usepackage{mathptmx,microtype,graphicx}
\usepackage{amsthm,amsmath,amssymb,amsfonts}
\usepackage{enumitem}
\usepackage[hidelinks]{hyperref}
\usepackage[noadjust]{cite} 
\usepackage{mathtools}
\usepackage[small]{caption}
\usepackage{subcaption} 

\numberwithin{equation}{section}

\theoremstyle{plain}
\newtheorem{theorem}{Theorem}[section]
\newtheorem{lemma}[theorem]{Lemma}
\newtheorem{proposition}[theorem]{Proposition}
\newtheorem{corollary}[theorem]{Corollary}
\newtheorem{claim}[theorem]{Claim}
\theoremstyle{definition}
\newtheorem{definition}[theorem]{Definition}
\newtheorem{remark}[theorem]{Remark}

\newcommand{\eps}{\varepsilon}
\renewcommand{\le}{\leqslant}
\renewcommand{\ge}{\geqslant}

\newcommand{\E}{\mathbb E}
\renewcommand{\P}{\mathbb P}

\newcommand{\Z}{\mathbb Z}

\newcommand{\cA}{\mathcal A}
\newcommand{\cC}{\mathcal C}
\newcommand{\cD}{\mathcal D}
\newcommand{\cE}{\mathcal E}
\newcommand{\cG}{\mathcal G}
\newcommand{\cT}{\mathcal T}

\newcommand{\Gnp}{\cG_{n,p}}

\makeatletter
\def\@tocline#1#2#3#4#5#6#7{\relax
\ifnum #1>\c@tocdepth 
\else
\par \addpenalty\@secpenalty\addvspace{#2}%
\begingroup \hyphenpenalty\@M
\@ifempty{#4}{%
  \@tempdima\csname r@tocindent\number#1\endcsname\relax
}{%
  \@tempdima#4\relax
}%
\parindent\z@ \leftskip#3\relax \advance\leftskip\@tempdima\relax
\rightskip\@pnumwidth plus4em \parfillskip-\@pnumwidth
#5\leavevmode\hskip-\@tempdima
  \ifcase #1
   \or\or \hskip 1em \or \hskip 2em \else \hskip 3em \fi%
  #6\nobreak\relax
\dotfill\hbox to\@pnumwidth{\@tocpagenum{#7}}\par
\nobreak
\endgroup
\fi}
\makeatother

\author[Z. Bartha]{Zsolt Bartha}
\address{HUN-REN Alfr\'ed R\'enyi Institute of Mathematics}
\email{bartha@renyi.hu}

\author[B. Kolesnik]{Brett Kolesnik}
\address{University of Warwick, 
Department of Statistics}
\email{brett.kolesnik@warwick.ac.uk}

\author[G. Kronenberg]{Gal Kronenberg}
\address{University of Oxford, 
Department of Mathematics}
\email{gal.kronenberg@maths.ox.ac.uk}

\author[Y. Peled]{Yuval Peled}
\address{Hebrew University, Einstein Institute of Mathematics}
\email{yuval.peled@mail.huji.ac.il}


\begin{document}

\title[Fuss--Catalan thresholds in graph bootstrap percolation]
{Sharp Fuss--Catalan thresholds in graph bootstrap percolation}

\begin{abstract}
We study graph bootstrap percolation on the Erd\H{o}s--R\'enyi 
random graph $\Gnp$. For all $r \ge 5$, we locate the sharp 
$K_r$-percolation threshold $p_c \sim (\gamma n)^{-1/\lambda}$, 
solving a problem of Balogh, Bollob{\'a}s and Morris. 
The case $r=3$ is the classical graph connectivity threshold, 
and the threshold for $r=4$ was found using strong 
connections with the well-studied $2$-neighbor dynamics 
from statistical physics. 
When $r \ge 5$, such connections break down, 
and the process exhibits much richer behavior. 
The constants $\lambda=\lambda(r)$ 
and $\gamma=\gamma(r)$ in $p_c$ are
determined by a class of 
$\left({r\choose2}-1\right)$-ary tree-like graphs, 
which we call $K_r$-tree witness graphs. 
These graphs are associated with the most efficient 
ways of adding a new edge in the $K_r$-dynamics, 
and they can be counted using the Fuss--Catalan numbers. 
Also, in the subcritical setting, we determine the asymptotic 
number of edges added to $\Gnp$, showing that the 
edge density increases only by a constant factor, 
whose value we identify.
\end{abstract}

\maketitle

\section{Introduction}
\label{S_Intro}

Bollob\'{a}s \cite {Bol68} introduced the concept
of {\it weak saturation} in graph theory, which gives rise to a process called 
{\it $H$-bootstrap percolation} (or {\it $H$-dynamics}) on a graph $G$. 
Given a fixed {\it template} graph $H$ and an initial graph 
$G\subseteq K_n$, at each step of the process a new edge $e$ is added, 
provided there is a copy of $H$ in $K_n$ for which $e$ is the only 
edge that has not yet been added to $G$. 
If no such edge exists, the process terminates, and we denote the 
final graph by  $\langle G \rangle_H$. 
We say that $G$ is {\it weakly $H$-saturated} 
(or that it {\it $H$-percolates})
if every edge 
is eventually added,
that is,  $\langle G \rangle_H=K_n$.

In this work, we find, for every $r\ge 5$, 
the sharp threshold probability
$p_c(n,K_r)$  
at which the Erd{\H o}s--R\'{e}nyi 
random graph $\Gnp$ is likely to $K_r$-percolate. 
This solves, in a strong form, an open problem of 
Balogh, Bollob\'{a}s and Morris  \cite {BBM12}, 
who located the threshold up to polylogarithmic factors 
and asked for the threshold up to constant factors.
This result concerns one of the central problems 
in graph bootstrap percolation. 
Since the seminal work of Bollob{\'a}s \cite {Bol68}, 
the case $H=K_r$ has attracted particular attention, 
and following Balogh, Bollob{\'a}s, and Morris \cite {BBM12}, 
the precise behavior of the $K_r$-dynamics 
on $G=\Gnp$ has remained open.

The classical model of {\it bootstrap percolation} 
was introduced in the statistical  physics literature by  
Chalupa, Leath and Reich \cite{CLR79}, 
and is one of the most well studied of all {\it cellular automata}. 
Such processes, as pioneered by  
Ulam \cite {Ula52} and von Neumann \cite {vonNeu66}, 
are notable in that, although they evolve only 
by local rules, they can exhibit complicated 
global behavior of interest; see, e.g., Gardner's \cite{Gar70}
introduction to Conway's {\it Game of Life}. 
In vertex bootstrap percolation on a graph $G$, 
each vertex is initially {\it infected} with probability $p$, 
and further vertices are infected by some local rule. 
For instance, and 
most commonly, by the {\it $r$-neighbor rule},  
whereby a vertex with at least $r$ 
infected neighbors becomes infected. 
In this context, {\it percolation} is the event that all 
vertices are eventually infected. 
There are many variations; see, e.g., 
the survey by Morris
\cite{Mor17}
and references therein. 

Graph bootstrap percolation can be seen as a 
generalization of the classical $r$-neighbor model, 
where the template 
graph $H$ encodes a more complicated update rule.
Notably, the ``infection'' in graph bootstrap percolation 
spreads amongst  
the \textit{edges} of some graph, 
rather than its vertices. The initialization at  
$\Gnp$ is therefore quite natural, since in bootstrap percolation 
each site is assumed to be initially infected 
independently with probability $p$.

The study of $\Gnp$ focuses on graph properties that occur 
\textit{with high probability}, i.e., with probability tending to $1$ as 
$n\to\infty$. A main theme in random graph theory, initiated by the 
seminal work of Erd\H{o}s and R\'enyi  \cite {ER59}, is the systematic 
study of threshold probabilities for monotone graph properties. 
A property $P$ is said to have a \textit{sharp threshold} at $p_c$ 
if for every fixed $\varepsilon>0$, with high probability, 
$\cG_{n,(1+\varepsilon)p_c}$ has property $P$ but $\cG_{n,(1-\varepsilon)p_c}$ 
does not. For example, a classical result from  \cite {ER59} shows 
that graph connectivity has a sharp threshold at $p_c\sim (\log n)/n$.  
Observe that $\langle G\rangle_{K_3}=K_n$ if and only if $G$ is connected, 
and therefore thresholds 
for $H$-percolation can be viewed as higher forms of the 
celebrated graph connectivity threshold.

\subsection{Main results}

As in \cite{BBM12}, we let 
\[
p_c(n,H)
=\inf\{p>0:\P(\langle \cG_{n,p}\rangle_H=K_n)\ge1/2\}
\]
denote the {\it critical $H$-percolation threshold}.
Roughly speaking, at this point $\Gnp$ is likely to $H$-percolate. 
One of the main results in \cite[Theorem 1]{BBM12} shows that, 
for every $r\ge 3$, 
\begin{equation}\label{E_BBM12}
p_c(n,K_r)=n^{-1/\lambda+o(1)},
\end{equation}
as $n\to\infty$, 
where 
\begin{equation}\label{E_lambda}
\lambda =\frac{\binom{r}2-2}{r-2}\,.
\end{equation}
Intuitively, $-1/\lambda$ is the critical exponent 
since once $p\gg n^{-1/\lambda}$,
 the number of edges added in the first round of the $K_r$-dynamics is 
 much larger than the initial number of edges in $\Gnp$.

The classical result of Erd\H{o}s and R\'enyi \cite {ER59} 
implies that $p_c(n,K_3)\sim(\log n)/n$, where, as usual, 
we write $a_n\sim b_n$ if $a_n/b_n\to1$, 
as $n\to\infty$. When $r=4$, the threshold was identified up to constants
in \cite{BBM12}, and the existence of a sharp threshold at 
$p_c(n,K_4)\sim(3n\log n)^{-1/2}$ follows by \cite{AK18,AK21,Kol22}.

Our first main result establishes a sharp 
threshold for $K_r$-percolation, in 
all of the remaining cases $r\ge 5$. 
Furthermore, we locate the precise
first-order asymptotics of $p_c(n,K_r)$. 
Let $\alpha_d=(d+1)^{d+1}/d^d$
denote the asymptotic growth rate of the $d$th 
Fuss--Catalan numbers (see Section \ref{SS_countingTWGs}), 
that is, the number of $(d+1)$-ary plane trees. 

\begin{theorem}
\label{T_pc}
Fix $r\ge5$. Let $\gamma>0$ be the unique positive root of
\begin{equation}\label{E_gamma}
(r-2)!\gamma^{r-2} = \alpha_{{r\choose2}-2}\,.
\end{equation}
Then, for every $\varepsilon>0$,
\[
\P(\langle \Gnp\rangle_{K_r}=K_n) \to \left\{
\begin{matrix}
0,&\text{if }\gamma np^\lambda =1 -\varepsilon;\\
1,&\text{if }\gamma np^\lambda =1+\varepsilon, 
\end{matrix}
\right.
\]
as $n\to\infty$. In particular,
\[
p_c(n,K_r)\sim (\gamma n)^{-1/\lambda}.
\]
\end{theorem}

Problem 3 in \cite{BBM12} asks for $p_c(n,K_r)$
up to constant multiplicative factors. Theorem \ref{T_pc} gives, moreover,
the precise leading-order asymptotics. 
In addition, Problem 2 in \cite{BBM12} asks for which $H$ is the critical 
$H$-percolation threshold sharp, and, 
by Theorem \ref{T_pc}, the threshold is sharp when $H$ is a clique. 

The distinction between the behavior of $K_r$-percolation 
for $r=4$ and for $r \ge 5$ 
is already evident from the absence of a polylogarithmic factor in $p_c(n,K_r)$. 
There are several interesting reasons for this transition. 
First, although all cliques $K_r$ are, in a certain sense, {\it balanced}, 
only once $r\ge5$ are they {\it strictly balanced} (see Definition \ref{D_SB} below). 
Related to this, on $\Gnp$ the $K_4$-dynamics are essentially dominated by the 
simpler $2$-neighbor bootstrap percolation dynamics. 
Specifically, the most likely way for $\Gnp$ to $K_4$-percolate 
is to start with the two vertices in some edge $e\in E(\Gnp)$ and then 
iteratively add vertices with at least two neighbors amongst those previously added. 
By induction, after each such vertex is added, a $K_4$-percolating graph
is obtained. 
In this sense, $K_4$-percolation on $\Gnp$ (like most forms 
of bootstrap percolation) spreads by ``nucleation,''  
meaning that the percolation spreads locally from the 
boundary of a growing structure (the ``nucleus'') 
until it reaches a critical size (sometimes called a ``critical droplet''), 
at which point the process transitions to a final stage of explosive growth.
On the other hand, for $r\ge5$, the connection to the $(r-2)$-neighbor dynamics breaks down, 
and $K_r$-percolation on $\Gnp$ exhibits qualitatively different behavior. 
In particular, when $r\ge5$, near criticality  
$\Gnp$ has no $K_r$-percolating subgraphs of order 
$1 \ll k\le n^c$, for some $0<c<1$; see Section \ref{S_nuce}. 

We thus take a different approach for finding the sharp threshold for the $K_r$-dynamics. 
Rather than studying $K_r$-percolating subgraphs
of $\Gnp$, we carry out a systematic study of its {\it witness graphs}, that is, 
inclusion-minimal subgraphs of $\Gnp$ that add a given edge. 
The parameter $\lambda$ appears in the exponent of $p_c(n,K_r)$ 
because any witness graph for a given edge with $v$ additional 
vertices contains at least $\lambda v+1$ edges. 
The constant $\gamma$  that appears in $p_c(n,K_r)$ is 
connected to the Fuss–Catalan numbers since the graphs that most efficiently 
add a given edge, which we will call {\it tree witness graphs (TWGs)}, have a certain 
$\left(\binom r2 -1\right)$-ary tree-like structure; see Figure \ref{F_KrTs} below.

Establishing sufficiently precise bounds for the number of witness graphs 
that are more ``costly'' than TWGs turns out to be highly non-trivial, 
and much of the current work is devoted to this task. 
The central challenge, in 
analyzing the structure of 
such graphs, is to 
control the spread of ``zero-cost'' edges that occurs
after a ``costly'' (non-tree-like) step in the $K_r$-dynamics. 
We refer the reader to Section \ref{S_Discussion} for 
a more in-depth 
overview of the challenges and our main ideas in this work.
\smallskip

In addition, our proof of Theorem \ref{T_pc} shows 
that in the subcritical regime almost all edges added to 
$\Gnp$ arise from TWGs. 
This allows us to derive the following sharp estimate 
for the number of edges in 
$\langle\Gnp\rangle_{K_r}$.

\begin{theorem}
\label{T_sub_density}
Fix $r\ge5$. 
Let $p=(\bar\gamma n)^{-1/\lambda}$
for some 
$\bar\gamma > \gamma$, where 
$\gamma$ is as  in \eqref{E_gamma} above. Define 
\[
\bar\alpha:= (r-2)!\bar\gamma^{r-2} > \alpha_{{r\choose2}-2}\,, 
\]
and let $1<\rho < 
1+\left(\binom r2-2\right)^{-1}$
be the smallest root of the equation 
\[
\rho^{\binom r2-1}=\bar\alpha(\rho -1)\,.
\]
Then, 
\[
|E(\langle \Gnp\rangle_{K_r})|\sim \rho\cdot p\binom n2\,,
\]
in probability, as $n\to\infty$.
\end{theorem}

In other words, for subcritical $p$, 
the number of edges in $\Gnp$ essentially 
only increases by a factor of $\rho$, universally bounded by 
$1+\left(\binom r2-2\right)^{-1}$,
after the $K_r$-dynamics have stabilized.
See Figure \ref{F_pc}
for an illustration of our main results (however, see also Section \ref{S_slow}
for a discussion on the limitations of such simulations). 

\begin{figure}[h!]
\centering
\begin{subfigure}[t]{.32\textwidth}
    \includegraphics[trim=0cm 0cm 0cm 0cm, clip, width=\textwidth]{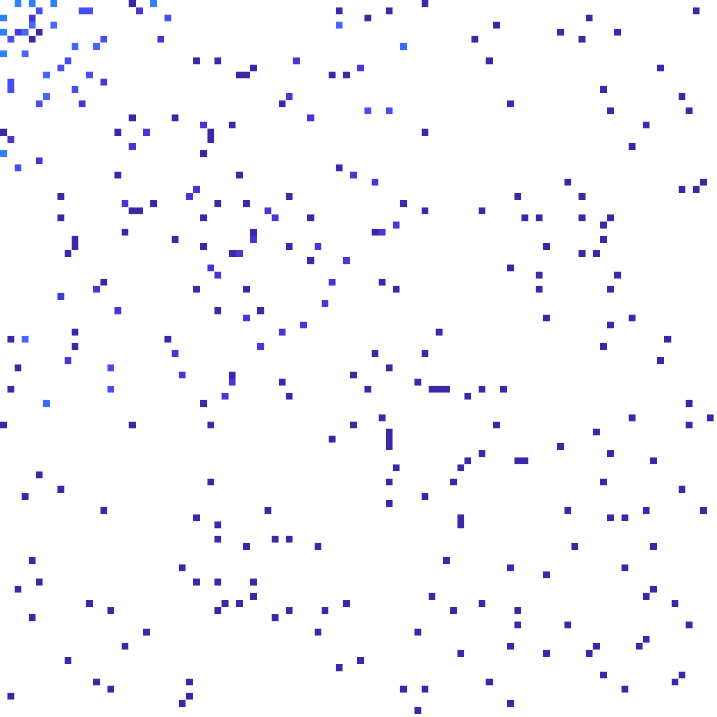}
    \caption{}
    \label{F_pc_a}
\end{subfigure}\hskip0.1cm
\begin{subfigure}[t]{.32\textwidth}
    \includegraphics[trim=0cm 0cm 0cm 0cm, clip, width=\textwidth]{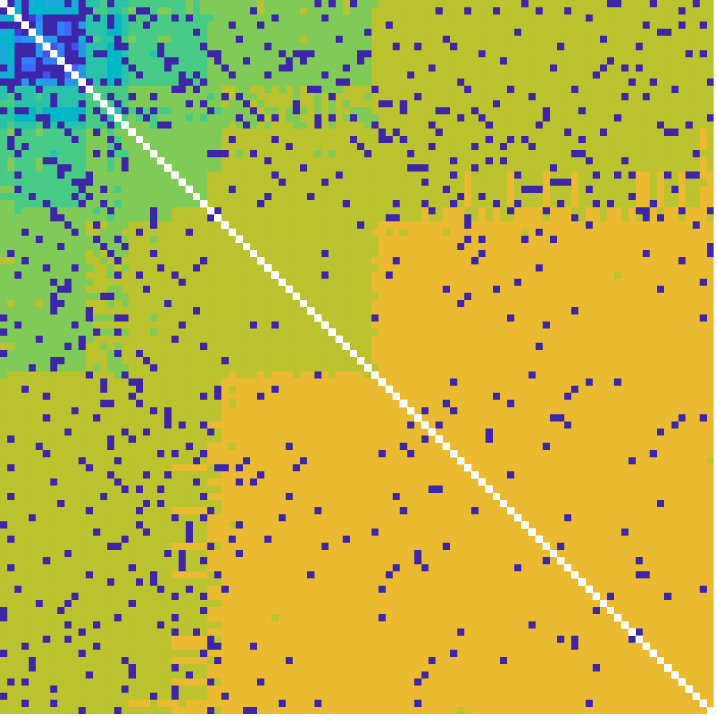}
    \caption{}
    \label{F_pc_b}
\end{subfigure}\hskip0.1cm
\begin{subfigure}[t]{.32\textwidth}
    \includegraphics[trim=0cm 0cm 0cm 0cm, clip, width=\textwidth]{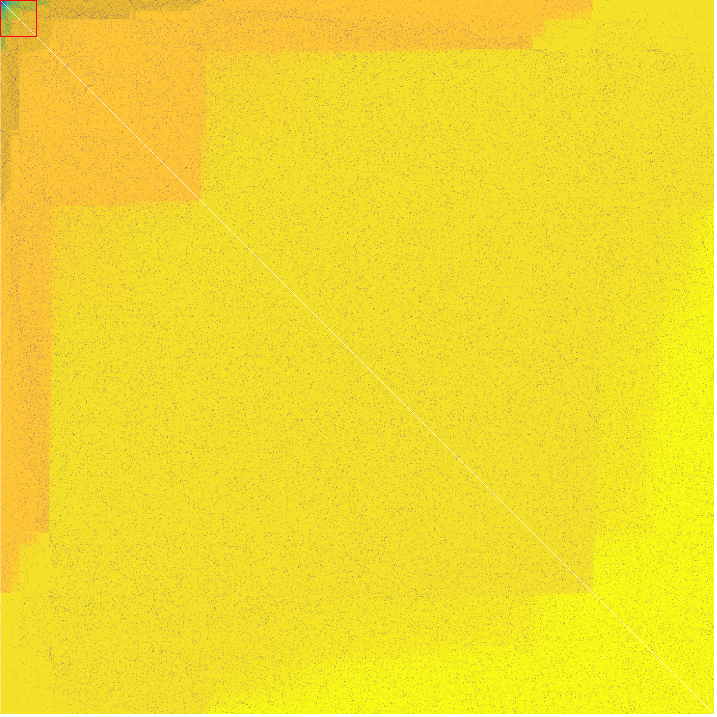}
    \caption{}
    \label{F_pc_c}
\end{subfigure}

\caption{Simulations of the $K_5$-dynamics on random graphs with 
$n=2000$ vertices, before and after criticality. 
Each point $(i,j)$ is colored along a gradient from dark blue to yellow 
according to the time at which the edge $\{v_i, v_j\}$ is added in the $K_5$-dynamics 
(obtaining a symmetric image), where the  vertices are ordered retrospectively, 
biased so that vertices incident to many early-added edges appear earlier. 
Dark blue points represent edges present in the initial graph, 
yellow those added in the final round, and white the edges that are never added.
(\subref{F_pc_a}) In the sub-critical dynamics, the edge density only increases by 
a constant factor (few   lighter blue pixels appear). The image shows $100$ vertices, 
as no percolation occurs in other vertices. The process finished in $4$ rounds. 
(\subref{F_pc_b}) The early stages of the super-critical dynamics (dark blue to green) 
on the first $100$ vertices in the biased order.
(\subref{F_pc_c}) The process ``explodes,'' 
and all remaining edges are added (orange to yellow). 
The center image is the upper left region (highlighted in red) in the right image. 
The process finished in $15$ rounds.}
\label{F_pc}
\end{figure}

\subsection{Related work}

There is a large mathematical 
literature on bootstrap percolation. In particular, 
the $r$-neighbor dynamics on 
Euclidean lattices $\Z^d$
and cubes $[n]^d$
is well studied;
see, e.g., \cite{vanEnt87,Sch92,AL88,Hol03,BBDCM12,HT24,CC99}. 
Moreover, Balister, Bollob\'{a}s, Morris and Smith \cite{BBMS25}
(and see references therein) 
have recently proven the 
{\it universality conjecture} 
concerning more general 
forms of monotone cellular automata.
We also note that 
Janson, \L uczak, Turova and Vallier
\cite{JLTV12} 
(cf.\ \cite{FKR17,AK21,TGL19}) 
have studied the $r$-neighbor dynamics on
$\Gnp$ in great detail.

In graph theory, there has been continued interest in weak saturation
and graph bootstrap percolation 
since the works of 
Bollob\'{a}s \cite {Bol68}
and 
Balogh, Bollob{\'a}s, and Morris \cite {BBM12}. 
Bollob\'{a}s \cite {Bol68}
initiated the study of the minimal number of edges 
in a weakly $K_r$-saturated subgraph of 
$K_n$, identifying this number when 
$3\le r\le 6$. 
Alon \cite{Alo85}, 
Frankl \cite{Fra82}, 
Kalai \cite{Kal84,Kal85} and 
Lov\'{a}sz \cite{Lov77} (cf.\ Yu \cite{Yu93}) generalized the result
to all $r$. 
Later, Balogh, Bollob\'{a}s, Morris and Riordan \cite{BBMR12} 
obtained a hypergraph analogue of this classical result. See also 
\cite{AVZ25, BMTR21, FGJ13, Kal85, KMM21}
(and references therein) for the 
bipartite case, and 
\cite{ST23, Ter2025, Tu1992}
for the asymptotic behavior in the general case. 
Many other related problems have been studied. For example, 
there are several works on the running time 
(number of rounds until 
the $H$-dynamics stabilize); 
see, e.g., \cite {BKPS19,EJKL24,BPRS17,FMS25,GKP17,HL24,NR23}. 
Another line of research, initiated by Kor\'andi and Sudakov  \cite {KS17}, 
studies the minimal number of edges in a subgraph of $G\sim\Gnp$ such that the $H$-dynamics,  
restricted to the ``host graph'' $G$, eventually adds all the edges in $G$; 
see also  \cite{BMTRZ24,DHKSZ24,antonir2025does} for more recent progress in this direction.

As discussed above, Balogh, Bollob{\'a}s, and Morris \cite {BBM12}
initiated the study of threshold probabilities for $H$-percolation, 
primarily focusing on the case that $H=K_r$ is a clique. 
Some other types of template graphs $H$ have been 
studied in \cite{BBM12}
and elsewhere, e.g, in     \cite{BC22,BMTR21,BK24,BKK24};
however, a general understanding
for {\it all} $H$ (see Problem 1 in \cite{BBM12}) remains widely open. 

In closing, let us mention that a sharp Fuss--Catalan-type threshold, 
resembling Theorem \ref{T_pc}, was recently discovered in the context of 
hypergraph bootstrap percolation \cite {LP23}. In that work, the challenge 
in analyzing the counterparts of witness graphs was tackled by exterior algebraic shifting. 
We tried to apply similar methods in this work, via rigidity theory, but we were 
unsuccessful due to some fundamental obstructions; see Section \ref{S_rigid} below. 
On the other hand, it seems likely to us that the strategies presented in the current 
work are robust, and will be capable of giving a ``combinatorial" proof of the results 
in \cite{LP23}, and more importantly, useful in studying various general versions of graph 
and hypergraph bootstrap percolation, where a direct analysis of the 
dynamics is typically challenging.

\subsection{Outline}
In Section \ref{S_Discussion} we give a 
detailed overview of our proof strategy. 
Afterward, in Section \ref{S_UB} we prove the upper bound
in Theorem \ref{T_pc}, and in Sections \ref{S_LB}--\ref{S_sharpLB}
we prove the lower bound. 
In Section \ref{S_LB} we set up terminology 
for the various types of steps that can be taken in the construction 
of general witness graphs, and prove several
fundamental properties of such graphs
that will be used in the proof of the lower bound. 
Section \ref{S_LBuptoC}
establishes a coarse lower bound 
$p_c=\Theta(n^{-1/\lambda})$, 
which already solves Problem 3 in \cite{BBM12}. 
Building on these arguments, and adding several more
new ideas, we find the sharp lower bound in 
Section \ref{S_sharpLB}. 
Finally, in Section \ref{S_subexp}
we prove Theorem \ref{T_sub_density}
on the addition of edges in the
subcritical setting.

\setcounter{tocdepth}{1}
\tableofcontents

\section{Context and proof overview}
\label{S_Discussion}

 The purpose of this section is to explain why and in what ways the 
 $K_r$-dynamics for $r\ge 5$ is different and more complex than when $r=4$,  
 and to outline the main ideas behind our proofs in this setting.

\subsection{Nucleation in the $K_r$-dynamics}
\label{S_nuce}
As observed in \cite[\S4]{BBM12},  
any graph that is $K_4$-percolating can be decomposed into two or three
percolating subgraphs. Therefore, $K_4$-percolating graphs
have percolating subgraphs of all orders. 
A similar 
property, for classical  
bootstrap percolation on Euclidean lattices, 
was first observed by 
Aizenman and Lebowitz \cite{AL88}, and analogues of this property
have since been used in various other contexts. 

The simplest $K_4$-percolating graphs, that are grown recursively, 
by ``nucleation,''
were already recognized in  \cite[Fig.\ 4]{Bol68}. 
As discussed in Section \ref{S_Intro},  to construct such a graph, 
we start with a single edge and then iteratively add vertices to the graph
with two neighbors amongst those previously added. 
The growth of such a graph in $\Gnp$, started from one of its edges, gives rise to 
time-inhomogeneous branching process with an
increasing birthrate (since, as this graph grows, it becomes more likely to find a vertex outside of it  with at least two neighbors inside). 
When $p=(\alpha n\log n)^{-1/2}$,  
survival to time $\alpha\log n$ creates a ``critical droplet'' (the branching process become critical), after which 
the full percolation of $\Gnp$ is all but inevitable. 
As it turns out, for $\alpha<3$, with high probability, the branching process started at {\it some} edge of $\Gnp$ becomes super-critical. 
All other ways of percolating are of smaller order, and the sharp threshold occurs 
at $p_c\sim(3n\log n)^{-1/2}$ \cite{BBM12,AK18,AK21,Kol22}. 
In other words, the most likely way for $\Gnp$ to $K_4$-percolate
is to have an edge $e\in E(\Gnp)$ whose vertices 
form a {\it contagious set} (i.e., a set whose initial infection leads to full infection) for 
the classical $2$-neighbor
bootstrap percolation process on $\Gnp$. 

For $r\ge 5$, one can similarly construct $K_r$-percolating graphs via ``nucleation'' 
by starting with a $K_{r-2}$ and then iteratively adding vertices connected to 
$r-2$ previously added vertices. In fact, 
the series of works \cite{Bol68,Alo85,Fra82,Kal84,Kal85,Lov77} 
have established that every $k$-vertex $K_r$-percolating graph has at least 
$(r-2)k-\binom {r-1}2$ edges, precisely the same number of 
edges as in these “nucleating’’ constructions.

One might expect that the growth of such a graph in $\Gnp$ 
is the cause for $K_r$-percolation. However, a simple union bound implies that, 
when $p$ is near $p_c=n^{-1/\lambda+o(1)}$, the graph $\Gnp$ typically 
contains no percolating subgraphs with $k$ vertices in the range $1\ll k\ll L$,
 where 
\begin{equation}\label{E_L}
L=(np^{r-2})^{-1/(r-3)}= n^{(r-4)/(r^2-r-4)+o(1)}.
\end{equation}
This indicates a qualitative change occurring at $r \ge 5$, as $\lambda = r - 2$ 
when $r = 4$, but $\lambda < r - 2$ for $r \ge 5$. Consequently, the 
Aizenman–Lebowitz-type property fails for $K_r$-percolation with $r \ge 5$; 
that is, there exist $K_r$-percolating graphs that do not contain percolating 
subgraphs of all smaller orders. This helps explain why $K_r$-percolating graphs 
are generally much less understood when $r \ge 5$.
In addition, \eqref{E_L} says that if a “critical droplet’’ for $K_r$-percolation in 
$\Gnp$ does form, it must appear rather abruptly (see Section~\ref{S_ncs}). 
Our proof does not identify such a droplet directly; instead, it takes a 
different approach based on the analysis of \textit{witness graphs}.

\subsection{The coarse bounds for $p_c$}
The concept of a witness graph plays a key role  
in \cite{BBM12}, 
where the coarse bounds \eqref{E_BBM12} are proved.
A witness graph  $W$ for an edge $e$ is an inclusion-minimal 
subgraph for which $e\in E(\langle W\rangle_{K_r})$; see Definiton \ref{D_WG} below. 
The existence of such a witness graph $W$ as a subgraph of $G$ bears witness 
to the fact that the $K_r$-dynamics starting from $G$ eventually add $e$ to the graph. 
As observed in \cite{BBM12}, witness graphs satisfy an Aizenman--Lebowitz-type property. 
Specifically (see Lemma \ref{L_AL} below), in each step of the 
$K_r$-dynamics, the maximal number of edges in a witness graph, 
over all witness graphs for 
edges added so far, 
can increase by at most a factor of ${r\choose2}$, 
the number of edges in $K_r$. 

The bounds \eqref{E_BBM12} are obtained as follows in \cite{BBM12}.
To obtain the upper bound in \eqref{E_BBM12}, the authors first show that, 
when $p$ is sufficiently supercritical, 
most edges in $K_n$ are added by a {\it $K_r$-ladder graph} 
(see Figure \ref{F_KrTs} below)
of logarithmic height, and then full percolation is achieved by sprinkling. 
Informally, such a graph is obtained by stringing together a series of 
copies of $K_r$, with its target edge $e$ in a copy at one of the ends, 
and then removing all 
shared edges and $e$. The lower bound in
\eqref{E_BBM12}, on the other hand, is obtained by 
first showing that
all witness graphs with $k$ vertices have at least as many edges 
$\lambda (k-2)+1$ 
(see also \cite[Lemma 12]{BK24} for an alternative proof of this fact, 
that extends to all graphs $H$) 
as there are in $K_r$-ladder graphs of the same size, 
and then applying the Aizenman--Lebowitz property for witness graphs.

\subsection{Sharp upper bound on $p_c$}
\label{S_introUB}
A limitation of the approach in \cite{BBM12} is that, apart from 
determining the minimal number of edges, it does not examine in 
detail the combinatorial structure of witness graphs. 
The first step in finding a sharp upper bound on $p_c$ is 
to identify the class of all edge-minimal witness graphs. 
As it turns out, these are 
$({r\choose2}-1)$-ary tree-like graphs that we call {\it $K_r$-tree
witness graphs (TWGs)}, and this explains the 
``Fuss--Catalan-related''
constant $\gamma$ in Theorem \ref{T_pc}. 
Informally, a TWG for an edge $e$
is obtained by first stringing together copies of 
$K_r$ in a tree-like manner,
so that $e$ is in one (root) copy, and every edge is in at most 
two. Then, the TWG is obtained by removing $e$ and all such 
shared edges from the graph. Note that these removed edges, 
and $e$ in particular, will be added back by the $K_r$-dynamics. 
See Figure \ref{F_KrTs} for examples.

\begin{figure}[h]
\centering
\includegraphics[scale=0.85]{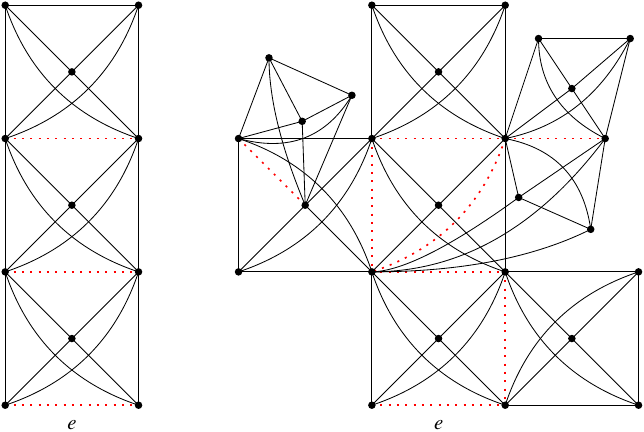}
\caption{A $K_5$-ladder (at left) of height 3 and a more complicated
$K_5$-tree (at right) of order 8. The non-edges of these $K_5$-trees that get 
added by the $K_5$-dynamics are indicated by red dotted lines. 
Their target edges
$e$ are labelled. 
}
\label{F_KrTs}
\end{figure}

To establish the upper bound in 
Theorem \ref{T_pc}
we show (see Section \ref{S_UB} below) that, in the supercritical regime, 
with high probability, all edges 
are eventually added to $\Gnp$ by the $K_r$-dynamics using a 
TWG of logarithmic ``order'' (number of copies of $K_r$ used in its construction). 
This is obtained by studying the expectation and variance of the number of 
TWGs for a given edge $e$ in $\Gnp$. The central idea is that every subgraph 
$S$ of a TWG $T$ for some $e$ has strictly smaller edge density than $T$, 
and moreover, the closer the density of $S$ is to that of $T$, the more closely 
its structure resembles a TWG for $e$.
The technicalities, however, are somewhat involved.

\subsection{Sharp lower bound on $p_c$}
\label{S_introLB}
The sharp lower bound in Theorem \ref{T_pc}
is more challenging, 
and it is here that much of the innovation in this work occurs. 
We obtain it by excluding the appearance of 
 {\it all} types of witness graphs for a given edge $e$. 
 As mentioned above, TWGs are the only edge-minimal witness graphs, 
 where a $k$-vertex TWG has $\lambda (k-2)+1$ edges. 
 Overall, our strategy is to show that a witness graph
with $k$ vertices and 
$\lambda (k-2)+1+\chi$ many edges is, in some useful sense, 
$O(\chi)$-close to being 
a TWG.
We call $\chi$ its {\it excess} number of edges. 

The {\it red edge algorithm (REA)}  introduced in \cite{BBM12} is 
a deterministic procedure that constructs a witness graph $W$ for a 
target edge $e$. At each step, the edges added so far form a collection 
of edge-disjoint components, 
and the algorithm introduces a new copy of $K_r$ that merges 
all components sharing edges with it.
In the final step, the copy of $K_r$ that is added includes the target edge $e$ 
and it merges all existing components into a single one.

There are three types of steps in the REA.
First, \textit{tree steps}, in which components are merged in a tree-like manner. 
For example, a witness graph is a TWG precisely when all its REA steps 
are tree steps.
Second, \textit{costly steps}, which increase the excess $\chi$ of the witness graph, 
typically by merging components in a non–tree-like way.
Third, \textit{zero-cost steps}, which take place within a single component 
and do not increase the excess.
While costly steps themselves are harmless
(since their impact is 
accounted for by the excess $\chi$), they may initiate cascades of zero-cost steps, 
potentially giving rise to a highly complicated structure within the witness graph. 
Bounding the spread of zero-cost steps in terms of the excess is far from trivial, 
and to address this problem, we introduce two key new ideas.

\subsubsection{The tree decomposition}

In Section \ref{S_LBuptoC}, 
we introduce the {\it tree decomposition} of a witness graph, 
and describe how it is constructed over time as the REA unfolds.
In each step, we decompose all existing components into a ``bad part'' 
and various other ``tree parts.''
These tree parts are similar to TWGs, but more general, 
as they may have multiple 
``target" edges that have yet to be reused by the REA. 
We do this carefully enough so that, on the one hand, 
any zero-cost spread of edges occurs only within the bad part, 
and so that, on the other hand, 
the number of inclusion-maximal tree parts that appear during the 
REA construction is bounded by $O(\chi)$.

We show that this decomposition 
is, on its own, enough to prove that $p_c=\Theta(n^{-1/\lambda})$, 
which already answers Problem 3 in \cite{BBM12}. 
We note that the assumption that $r\ge 5$ is heavily used here. 
Indeed, when $r=4$, a TWG itself creates a large zero-cost spread of edges 
(in fact, it turns into a clique), 
so it is impossible to distinguish between the 
bad part and the tree parts of a component. 
Indeed, when $r=4$, the lower bound
$\Omega(n^{-1/\lambda})$ fails, 
recalling that $p_c(n,K_4)\sim(3n\log n)^{-1/2}$.

\subsubsection{The $(r-2)_*$-bootstrap percolation process}
\label{S_*proc}

In Section~\ref{S_sharpLB}, to establish the sharp lower bound in 
Theorem~\ref{T_pc}, we refine our control 
on the spread of zero-cost edges 
within the various tree parts of the tree decomposition.
Using this refinement, we deduce that the tree parts are all close to being TWGs, 
and the sharp lower bound follows.

It is not hard to see that, when considered on its own, a tree part $T$ 
does not permit the propagation of zero-cost edges. However, if some of its 
vertices are involved in costly steps of the REA, this may cause additional 
edges to spread throughout $T$. 
Controlling this spread is the main objective 
in 
Section~\ref{S_sharpLB}. 

To this end, we make a comparison with a certain variation 
of the $(r-2)$-neighbor dynamics on the vertices of  $T$, which we call 
the {\it $(r-2)_*$-bootstrap percolation process}. 
The ``seeds'' (i.e., initially infected vertices) in this process are the vertices
of $T$ that are involved in costly steps. 
This process evolves like the usual $(r-2)$-dynamics, except 
for one special rule. 
Being a tree part, $T$ consists of copies of $K_r$ arranged so that each 
new copy intersects the union of the preceding ones in exactly one edge.
In the modified dynamics, we allow a vertex $v$ which belongs to such a 
shared edge $f$ to be infected at a lower threshold of $r-3$ if it has $r-4$ infected 
neighbors in one copy of $K_r$ containing $f$, and another infected neighbor in a 
different such copy, assuming that none of these neighbors is the other endpoint of $f$.

The close connection between the edge $K_r$-dynamics and the vertex $(r-2)_*$-dynamics 
follows from the observation that the propagation of new edges in the $K_r$-dynamics 
can “pass through’’ a shared edge $f$, from one part of $T$ to another only if there exists 
a vertex $v \in f$ satisfying the above condition.
Otherwise, the ``hinge'' (i.e.,  shared edge) $f$ in such a $T$ acts as 
a sort of ``bottleneck'' for the spread of new edges throughout $T$. 
See Section \ref{S_vBP} and Figure \ref{F_star_step} below for more details. 

Finally, even though the tree parts in our decomposition are not trees in the usual sense, 
they possess a workable tree-like structure. This allows us to analyze the 
$(r-2)_*$-bootstrap process on a tree part by extending the ideas 
in the work of Riedl \cite{Rie12} on the $r$-neighbor dynamics in trees.

\section{Sharp upper bound}
\label{S_UB}

In this section, we prove the sharp 
upper bound 
in Theorem \ref{T_pc}.

\begin{proposition}[Sharp upper bound]
\label{P_pcUB}
Fix $r\ge5$ and 
let $\eps>0$. Put  
\[
p=\left(\frac{1+\eps}{\gamma n}\right)^{1/\lambda}.
\]
Then, with high probability, 
$\langle\cG_{n,p}\rangle_{K_r}=K_n$. 
\end{proposition}

Throughout this section, we will use $c$ to denote a 
constant depending only on $r$. 
The specific value of $c$ may change, 
but it will always represent a constant 
depending only on $r$. We may also use $c'$, etc.,\ sometimes
to emphasize that its value is changing from one line to another.

\subsection{Growing TWGs}
\label{S_growTWGs}

Let us now formalize the notion of a TWG, 
as discussed briefly in Section \ref{S_introUB} above. 
In fact, we will give two recursive definitions 
of such structures, depending on whether it is more 
convenient to, in a certain sense (to be made precise below), 
grow the TWG from its {\it root}
or by {\it grafting a leaf} somewhere along one of its {\it branches}.

\begin{definition}[Growing a TWG from the root]
\label{D_TWGroot}
A graph $T$ is called a 
{\em tree witness graph (TWG)} for  the 
edge $e$ if either 
\begin{enumerate}
\item $T=e$; or else 
\item for some copy 
$H^{(e)}$ of $K_r$ 
with $e\in E(H^{(e)})$ there are TWGs 
$T_f$ for each edge $f\in E(H^{(e)}\setminus e)$
such that $T_f$ and $T_{f'}$ are vertex disjoint
outside of $H^{(e)}$ 
for all $f\neq f'\in E(H^{(e)}\setminus e)$, 
and $T=\bigcup_{f\in E(H^{(e)}\setminus e)}T_f$.
\end{enumerate}
In case (2) we call $H^{(e)}$ the {\it root} of $T$
and the $T_f$ its {\it primal branches}. 
\end{definition}

Next, we define the order of a TWG. 
Informally, this is the number of copies of $K_r$ 
in its associated hyper-tree.

\begin{definition}[Order of TWGs]
The {\it order} $\vartheta(T)$ of a TWG 
$T$ for $e$ is defined inductively. 
If $T$ consists of a single edge then $\vartheta(T)=0$, otherwise
\[
\vartheta(T)
=1+\sum_{f\in E(H^{(e)}\setminus e)}\vartheta(T_f),
\]
where $H^{(e)}$ is the root of $T$.
If $\vartheta(T)=k$, then we call $T$ a {\it $k$-TWG}. 
\end{definition}

\begin{definition}[Branches and leaves of TWGs]
\label{D_Bs}
Let $T$ be a TWG with root $H^{(e)}$. 
We call $B\subset T$ a {\it branch}
of $T$ if $B$ is a TWG. 
Inductively, $B$ is a branch of $T$ if $B=T$ or if $B$ is
a branch of one of its primal branches $T_f$. 
A branch $L$ of order $\vartheta(L)=1$ is 
called a {\it leaf} of $T$. 
\end{definition}

In particular, note that if $B$ is a branch of a TWG $T$,
then $B$ is a TWG for some edge $e'\notin E(T)$.

It will sometimes be more convenient to 
construct larger TWGs by removing an edge 
from a smaller TWG
and replacing it with a leaf. 
Hence, we also give the following equivalent definition of TWGs.

\begin{definition}[Grafting a leaf to a TWG]
\label{D_TWGleaf}
A graph $T$ is a $0$-TWG for an edge $e$ if $T=e$. 
For $k\ge1$, a graph $T$ is a $k$-TWG for $e$ 
if there exists 
\begin{itemize}
\item a $(k-1)$-TWG $T'$ for $e$; 
\item an edge $e'\in E(T')$; 
and 
\item a copy $H'$ of $K_r$ 
containing $e'$, and no other vertices
in $T'$ beside the endpoints of $e'$, 
\end{itemize}
such that $T$ is obtained from $T'$ by 
\begin{itemize}
\item removing the edge $e'$; and then 
\item {\it grafting} (taking the union with) the leaf $L$,  
obtained from $H'$ by removing $e'$, in its place. 
\end{itemize}
\end{definition}

Finally, we note that the following facts
can be proved by a simple induction, using either of
the above equivalent definitions.

\begin{lemma}\label{clm:kTWG_edges}
Suppose that $T$ is a $k$-TWG 
for an edge $e$. Then 
\begin{itemize}
\item $v(T)=(r-2)k+2$;  
\item $e(T)=\lambda(r-2)k+1$; and 
\item $e\in E(\langle T\rangle_{K_r})$. 
\end{itemize} 
\end{lemma}

In other words, if a graph $G$ contains a TWG $T$
for an edge $e$, then $T$ is indeed {\it witness} to the fact that 
$e\in E(\langle G\rangle_{K_r})$ is eventually
added by the $K_r$-dynamics. 
For the other two facts, 
simply note that $T$ contains the edge
$e$ and its two vertices, and that, each time a TWG
is grown by adding a leaf, $r-2$ vertices 
and $\lambda(r-2)={r\choose2}-2$ edges
are added.

\subsection{Counting TWGs}\label{SS_countingTWGs}

We recall that the {\it Fuss--Catalan number} 
\begin{equation}
\label{E_FC}
\mathrm{FC}_d(k)
=\frac{1}{dk+1} \binom{(d+1)k}{k}
\end{equation}
is the number of $(d+1)$-ary trees 
with $k$ internal nodes and $dk+1$ leaves. 
When $d=1$, $\mathrm{FC}_1(k)$ is the 
well-known {\it Catalan number}  
$C(k)$.

By Stirling's approximation,  
the Fuss--Catalan numbers  
have asymptotic growth rate
\[
\alpha_d
=\frac{(d+1)^{d+1}}{d^d}.
\] 
More specifically, for any fixed $d$, 
it follows by \eqref{E_FC}  that 
\begin{equation}
\label{E_FCasy}
\mathrm{FC}_d(k) 
\sim 
\beta_d\frac{\alpha_d^k}{k^{3/2}},
\end{equation}
as $k\to\infty$, 
where
\[
\beta_d=\sqrt{ \frac{d+1}{2\pi d^3}}. 
\]

Finally, we note that, similarly to the 
Catalan numbers, the 
Fuss--Catalan numbers can be defined 
recursively by setting 
$\mathrm{FC}_d(0)=1$ and, for $k\ge1$, 
\begin{equation}
\label{E_FCrec}
\mathrm{FC}_d(k)
=\sum_{k_1+\cdots+k_{d+1}=k-1}
\prod_{i=1}^{d+1}\mathrm{FC}_d(k_i).
\end{equation}

\begin{definition}
Let $t(k)$ denote the number of  (labelled)
$k$-TWGs
for a given edge $e$ on a given  vertex set
of size $(r-2)k+2$. 
\end{definition}

\begin{lemma}[Number of TWGs]
\label{lem:count_TWG}
For all integers $k\ge0$, we have that 
\begin{equation}\label{E_tk}
t(k)=\frac{((r-2)k)!}{(r-2)!^k} \mathrm{FC}_{\binom{r}{2} -2}(k). 
\end{equation}
\end{lemma}

Note that, by \eqref{E_gamma}, \eqref{E_FCasy}
and \eqref{E_tk}, it follows that 
\begin{equation}\label{E_tkasy} 
t(k)\sim 
\beta_{{r\choose2}-2}\frac{\gamma^{(r-2)k}}{k^{3/2}}((r-2)k)!,
\end{equation}
as $k\to\infty$.

\begin{proof}
The proof is by induction on $k$, combining 
Definition \ref{D_TWGroot}
with the recurrence \eqref{E_FCrec}. 

The base case $k=0$ is trivial, since $t(0)= 1$. 
On the other hand, 
for $k>0$, 
to construct a $k$-TWG $T$ for $e$, 
we need to choose 
\begin{enumerate}
\item the other $r-2$ vertices 
in the root $H^{(e)}$ of $T$; 
\item the orders $k_f=\vartheta(T_f)$  
for each of its ${r\choose2}-1$ many primal branches $T_f$, 
in such a way that $\sum_f k_f=k-1$; 
\item a partition $(\Pi_f)_f$ of the remaining $(r-2)(k-1)$ vertices 
outside of $H^{(e)}$, 
in such a way that each $|\Pi_f|=(r-2)k_f$; and 
\item a $k_f$-TWG for each primal branch $T_f$.
\end{enumerate}
As such, it follows that 
\[
t(k)
= {(r-2)k\choose r-2}
\sum_{\sum_f k_f=k-1} 
((r-2)(k-1))!
\prod_{f} \frac{t(k_f)}{((r-2)k_f)!}. 
\]
Consequently, the sequence 
\[
\frac{(r-2)!^k}{((r-2)k)!}t(k)
\] 
satisfies \eqref{E_FCrec}
with $d={r\choose2}-1$, 
and the result follows. 
\end{proof}

\subsection{Balanced TWGs}
The following concept is needed for technical 
reasons that will appear later in the proof of 
Proposition \ref{P_pcUB}.

\begin{definition}[Balanced TWGs]
Let $T$ be a TWG for an edge $e$,  
with root $H^{(e)}$, and primal branches $T_f$, 
as in Definition \ref{D_TWGroot}. 
We say that $T$ is {\it balanced}
if the orders $\vartheta(T_f)$ are all equal. 
\end{definition}

Note that this definition is {\it not} recursive. 
That is, crucially, the primal branches $T_f$ 
themselves need not be balanced.

\begin{definition}
Let $b(k)$ denote the number of (labelled) 
balanced $k$-TWGs
for a given edge $e$ on a given  vertex set
of size $(r-2)k+2$. 
\end{definition}

Note that, if 
$T$ is balanced then $k-1$ is 
divisible by $\binom r2 -1$, 
in which case we put
\begin{equation}\label{E_q}
q=\frac{k-1}{\binom r2 -1}, 
\end{equation}
so that all of its primal branches 
are of order $\vartheta(T_f)=q$.

\begin{lemma}[Number of balanced TWGs]
\label{L_count_bTWG}
For all integers $k\ge0$
such that $k-1$ is divisible by $\binom r2 -1$, 
we have that 
\begin{equation}\label{E_bk}
b(k)
= 
\frac{((r-2)k)!}{(r-2)!}
\left(\frac{t(q)}{((r-2)q)!}\right)^{\binom r2 -1}.
\end{equation}
\end{lemma}

\begin{proof}
To construct a balanced $k$-TWG $T$ 
for a given edge $e$ on a given  vertex set
of size $(r-2)k+2$, we need to choose 
\begin{enumerate}
\item the $r-2$ other vertices in the 
root $H^{(e)}$; 
\item a partition of the $(r-2)(k-1)$ 
remaining vertices into 
$\binom r2 -1$ equal sets of size $(r-2)q$; and
\item $\binom r2 -1$ TWGs of order $q$
for the primal branches. 
\end{enumerate}
Hence
\[
b(k)
= 
\frac{((r-2)k)!}{(r-2)!((r-2)q)!^{\binom r2 -1}}
t(q)^{\binom r2 -1},
\]
giving the result. 
\end{proof}

The key point is 
that the difference between $b(k)$ and $t(k)$ 
is not seen at the exponential scale.
Indeed, 
by \eqref{E_tkasy} and \eqref{E_bk} it follows that 
\begin{equation}\label{E_bkasy}
b(k)
\sim
\phi
\frac{t(k)}{k^{(3/2)\left(\binom r2 -2\right)}},
\end{equation}
where 
\[
\phi=
\frac{\left(\binom r2 -1\right)^{3/2} \beta_{{r\choose2}-2}^{\binom r2 -2}}
{(r-2)!\gamma^{r-2} }
\]
is a constant depending only on $r$. 

Finally, the last result in this section
implies that in the supercritical setting, 
we have, in expectation, many balanced $k$-TWGs
for any given edge $e$ for $k=\beta\log n$, 
provided that $\beta$ is sufficiently large.

\begin{lemma}\label{clm:mu}
Fix $e\in E(K_n)$. Let $\eps>0$ and put   
\[
p=\left(\frac{1+\eps}{\gamma n}\right)^{1/\lambda}.
\]
Suppose that $k=k(n)$ is a sequence such that 
$k-1$ is divisible by $\binom r2 -1$ for all $n$, 
and that 
$1\ll k\ll \sqrt{n}$, as $n\to\infty$.
Then, as $n\to\infty$, the number $B_k^{(e)}$
of balanced $k$-TWGs for $e$ in $\cG_{n,p}$
satisfies 
\[
\E(B_k^{(e)})\sim 
\phi'   \frac{(1+\eps)^{(r-2)k}}{k^{(3/2)
\left(\binom r2 -1\right)}}p,
\]
where
\[
\phi'=\beta_{{r\choose2}-2} \phi.
\]
\end{lemma}

\begin{proof}
By linearity of expectation, 
\[
\E(B_k^{(e)}) = \binom {n-2}{(r-2)k}  
b(k)  p^{\lambda(r-2)k+1}.
\]
Since $1\ll k\ll \sqrt{n}$, we have that 
\[
\binom{n-2}{(r-2)k}\sim \frac{n^{(r-2)k}}{((r-2)k)!}. 
\] 
Therefore, by
\eqref{E_tkasy} and \eqref{E_bkasy},
we find that 
\[
\E(B_k^{(e)})\sim 
\beta_{{r\choose2}-2} \phi 
\frac{(\gamma np^\lambda)^{(r-2)k}}
{k^{(3/2)\left(\binom r2 -1\right)}}p,
\]
and the result follows.
\end{proof}

\subsection{Partial TWGs}
\label{S_SofT}

The next key step towards the proof of 
Proposition \ref{P_pcUB}
is studying subgraphs of TWGs.

\begin{definition}[Partial TWGs]
We call a subgraph $S$ a {\it partial TWG} 
for an edge $e$
if there exists a TWG $T$ for $e$ such that 
$S\subsetneq T$ is a proper edge-induced subgraph of $T$.
\end{definition}

We will need to quantify, at least approximately, 
how far a partial TWG $S$ is from being a complete TWG. 
In the case that $S$ is the intersection 
of two potential TWGs $T_1$ and $T_2$, 
this information will help
bound the correlation
between the events that $T_1$ and $T_2$
are present in $\cG_{n,p}$.  

For a graph $G$ and some $v\in V(G)$, 
we let $E_v(G)\subset E(G)$
denote the edges of $G$ that are incident to $v$.

\begin{definition}[Sets $F$, $P$ and $D$]
Let $T$ be a TWG for an edge $e$. Suppose that
$S\subsetneq T$ is a partial TWG for $e$.
We partition the set of vertices in $T$ 
that are not in $e$ into the
sets $F$, $P$, and $D$ of vertices $v$ 
for which $E_v(T)$ is {\it fully} contained in, 
{\it partially} contained in, and {\it disjoint} from $E(S)$. 
\end{definition}
In other words, for each such $v\notin e$, we put 
\begin{itemize}
\item $v\in F$ if $E_v(T)\subset E(S)$; 
\item $v\in D$ if $E_v(T)\cap E(S)=\emptyset$; and 
\item $v\in P$ otherwise. 
\end{itemize}
Note that  that $F$, $P$
and $D$ are pairwise disjoint since $S$ is edge induced and thereby has
no isolated vertices.

To be precise, we can write 
$F=F(S,T)$,
$P=P(S,T)$ and  $D=D(S,T)$
as all three sets depend on $S$ and $T$. 
However, the set $F\cup P=V(S)\setminus e$
depends only on $S$, and we denote 
\[
\sigma(S)=|V(S)\setminus e|
\]
the number of vertices in $S$ that are not in $e$.

Finally, let us 
introduce the three key parameters
associated with $S\subsetneq T$
which will help us measure how far 
$S$ is from being a TWG.

\begin{definition}[Key partial TWG parameters]
\label{D_keypar}
We define the 
\begin{itemize}
\item {\it edge efficiency} of $S$ as 
$\cE(S)=\lambda\sigma(S)-e(S)$;
\item {\it within tree deficiency} of $S$ 
as $\cD(S,T)=|P|$; and
\item {\it tree extendability} $\cT(S)$ of $S$
to be the minimum of $\sigma(T_*)-\sigma(S)$ over all
TWGs $T_*$ for $e$ 
satisfying $S\subsetneq T_*$.
\end{itemize}
\end{definition}

As indicated by the notation, 
$\cD$ depends on $S$ and $T$, but 
$\cE$ and 
$\cT$
are, in fact, 
independent of $T$.

\subsection{Structure of partial TWGs}
The goal of this subsection is to establish the following lemma 
that connects between the three key partial TWG parameters.
\begin{lemma}[Partial TWG structural lemma]
\label{L_strTWB}
Let $T$ be a TWG and $S\subsetneq T$ a partial TWG. Then, 
\begin{enumerate}
\item $\cE(S)\ge 0$, with equality 
if and only $S$ is obtained from $T$ 
by removing one of its branches; 
\item $\cD(S,T)\le c\cE(S)+2$; and 
\item $\cT(S)\le c\cE(S)$, 
\end{enumerate}
for some constant $c$ depending only on $r$. 
\end{lemma}

Recall that a branch $B\subset T$ is a TWG for some $e'\not\in E(T)$. 
When {\it removing} $B$ from $T$, we do {\it not}, 
however, mean to remove the 
two vertices in $e'$ from what remains. 

The proofs of these results,  
in Sections \ref{S_strTWB1}, \ref{S_strTWB2} 
and \ref{S_strTWB3} below, 
are by a somewhat technical case analysis. 
The reader might wish to first skip
ahead to 
Sections \ref{S_EBs} and \ref{S_pcUB}
to see how it all fits together, before
returning to the proof of this lemma. 

Parts (1) and (2) are closely related, 
and the proof of (2) is essentially a 
continuation of the arguments for (1). 
Part (3), on the other hand, 
requires a different type of case analysis.

\subsubsection{Strictly-balanced graphs}
The proof of Lemma \ref{L_strTWB} makes critical use of the 
fact that the clique $K_r$ is strictly balanced for $r\ge5$.
To highlight this fact, we present Lemma \ref{L_SB} below in the 
broader setting of strictly balanced graphs; Lemma \ref{L_cEL}, 
concerning the clique $K_r$ for $r\ge5$, 
will follow as a special case useful in our arguments.

\begin{definition}[Strictly balanced]
\label{D_SB}
A graph $H$ is {\it strictly balanced} if 
\begin{equation}\label{E_SB}
\frac{e(F)-1}{v(F)-2} < \lambda(H),
\end{equation}
for every subgraph $F\subset H$ 
with $3\le v(F)<v(H)$, 
where
\begin{equation}\label{E_lamH}
\lambda(H)=\frac{e(H)-2}{v(H)-2}
\end{equation}
is its {\it $\lambda$-density} 
(extending \ \eqref{E_lambda} above). 
\end{definition}

\begin{definition}
\label{D_xi}
For a strictly-balanced graph $H$, 
we let $\xi(H)>0$ denote the minimal
possible value of 
\[
\lambda(H)(v(F)-2)-(e(F)-1)
\]
over all $F\subset H$ with $3\le v(F)<v(H)$. 
\end{definition}

\begin{lemma}
\label{L_SB}
Suppose that $H$ is a strictly balanced graph
and that $K$ is a proper subgraph 
of the graph $H\setminus e$,  
obtained from $H$ by 
removing some $e \in E(H)$. 
Let $x$ denote the number of vertices in 
$K$ that are not in $e$. 
Then we have that $0\le x\le v(H)-2$
and 
\[
e(K) \le \lambda(H)x, 
\]  
with equality if and only if 
either 
\begin{enumerate}
\item $x=0$ and so $e(K)=0$; or else
\item $x=v(H)-2$ and $e(K)=e(H)-2$.
\end{enumerate}
In all other cases, in fact, 
\[
e(K) \le \lambda(H) x-\xi(H),
\] 
where $\xi(H)>0$ is as in 
Definition \ref{D_xi}. 
\end{lemma}

\begin{proof}
First, we note that if $K$ contains $x=0$ 
vertices outside of $e$, 
then clearly $e(K)=0$ 
since $K\subset H\setminus e$, 
and 
so $\lambda(H) x-e(K)=0$ holds trivially. 

Next, if $x=v(H)-2$, then
\[
\lambda(H)x-e(K)
=e(H)-2-e(K)
\ge 0
\] 
by \eqref{E_lamH} and 
since $K\subsetneq H\setminus e$. 
Moreover, equality holds 
if and only if $e(K)=e(H)-2$. 

Finally, suppose that $0<x<v(H)-2$.
In this case, consider 
$F=K\cup e$ obtained from $K$ by adding $e$. 
Observe that 
\[
3\le v(F)=x+2 < v(H). 
\]
Then, since $H$ is strictly balanced, 
it follows 
by \eqref{E_SB}
that 
\[
e(K) = e(F)-1 
\le \lambda (H)(v(F)-2) -\xi(H)
= \lambda(H) x-\xi(H),
\]
as claimed.
\end{proof}

\subsubsection{Proof of Lemma \ref{L_strTWB}}

Let us return to the setting of Lemma \ref{L_strTWB}, 
where $S\subsetneq T$
is a proper 
edge-induced 
subgraph of a $k$-TWG $T$, for some $k\ge1$. 
By Definition \ref{D_TWGleaf}, $T$ can be obtained  
by removing an edge $e'$ from some $(k-1)$-TWG $T'$
and then grafting a leaf $L$ in its place. 
Therefore, since $L\cup e'$ is a copy of $K_r$, 
the next useful result follows 
immediately by Lemma \ref{L_SB}.

\begin{lemma}[Efficiency within a leaf]
\label{L_cEL}
Suppose that $L$ is not contained in $S$. 
Let $x$ denote the number of vertices in 
$S\cap L$ that are not in $e'$.
Then 
\[
e(S\cap L) \le \lambda x, 
\] 
with equality if and only if 
\begin{enumerate}
\item $x=0$; or else
\item $x=r-2$ and $e(S\cap L)=\binom r2 -2$,
\end{enumerate}
and in all other cases,
\[
e(S\cap L) \le \lambda x-\xi,
\] 
for some $\xi>0$ depending only on $r$. 
\end{lemma}

In our proofs, 
it will often be useful to consider one of the two 
following partial TWGs of $T'$.
\begin{itemize}
\item Let $S'$ be the edge-induced subgraph of $T'$, 
induced by the edges of $S$ that are contained in $T'$.
\item Let $\bar S'$ be obtained 
from $S'$ by adding the edge $e'$
and its vertices (if not already included in $S'$). 
\end{itemize}

Note that because $S\subset T$
and $e'\notin E(T)$, it follows
that $e'$ is not an edge in $S'$. Therefore, $S'\subsetneq T'$, 
which will allow us
to apply various induction hypotheses directly
in some cases.  On the other hand, $\bar S'$ may be equal to $T'$,
in which case other arguments will be required.

\subsubsection{Efficiency: Proof of Lemma \ref{L_strTWB}(1)}
\label{S_strTWB1}

First, we observe that if $S$ is obtained from $T$ 
by removing a branch $B$ that is a TWG for some 
$f\notin E(T)$ then 
$T_f=S\cup f$ is a $k'$-TWG for $e$ for some $k'\ge 0$. 
In this case, 
$\sigma(S)=v(T_f)-2=k'(r-2)$ 
and $e(S) = e(T_f)-1=({r\choose 2}-2)k'$, 
and so $\cE(S)=0$. 

Next, for the inequality 
and the other direction of the equality case, 
we proceed by induction on $k$, 
the order of $T$.

The base case $k=0$ holds trivially, 
as then $T$ is a single
edge and $S=\emptyset$, 
which can be obtained from $T$ by
removing $T$ (a branch of $T$) from $T$.  

For the induction step, suppose that $k\ge1$. 
Then, as discussed above, 
$T$ can be obtained by removing an edge $e'$ 
from some $(k-1)$-TWG $T'$
and then grafting a leaf $L$ in its place.

{\bf Empty case.} 
Suppose that $S$ contains no edges in $L$
and no vertices in $L$, except possibly those in $e'$. 
Then $S=S'$. Moreover, we have that $S\subsetneq T'$, since 
$e'\in E(T')$ but  
$e'\not\in E(T)$. 
Therefore, the inductive 
hypothesis applies directly to show that 
$\cE(S)=\cE(S')\ge0$. 
For the case of equality, 
note that if $S$ is obtained from $T'$ by removing
a branch $B'$ from $T'$, then $e'\in E(B')$. In this case,  
$S$ can be obtained from $T$ by removing the branch $B$
obtained from $B'$ by removing $e'$ and grafting $L$
in its place.

{\bf Complete case.} 
Next, suppose that 
$L\subset S$. 
Then $\bar S'\subsetneq T'$
since $S\subsetneq T$. 
Hence by induction, 
we have that $\cE(\bar S')\ge0$,
and it follows that $\cE(S)\ge0$ since
$S$ has precisely $r-2$ more vertices 
and ${r\choose2}-2=\lambda(r-2)$
more edges than $\bar S'$, 
so that $\cE(S)=\cE(\bar S')\ge0$. 
For the case of equality, 
recall that $e'\in E(\bar S')$. 
Therefore, if $\bar S'$ is obtained from $T'$ by 
removing one of its branches $B$, 
then $e'\notin E(B)$, 
in which case 
$B$ is also a branch of $T$, and 
$S$ can be obtained by removing $B$ from $T$.

{\bf Nearly-complete case.}
Next, we suppose that 
$e(S\cap L)=\binom r2-2$, that is, 
$S$ contains all except one of the 
edges $e''$ in $L$. 
Note that $S$ necessarily contains all $r-2$ of the vertices
in $L$ that are not in $e'$. 
Therefore, 
by induction,
\[
\cE(S) 
\ge \cE(S') 
+ \lambda (r-2)-e(S\cap L)
=\cE(S')\ge 0.
\]
Moreover, if equality holds 
then necessarily $\cE(S')=0$
and $S'$ contains 
all vertices in $e'$
that are not in $e$, as then 
$\sigma(S)=\sigma(S')+(r-2)$. 
In this case, $\bar S'$ is a subgraph of $T'$
on the same vertex set as $S'$. 
It follows that $\cE(\bar S')<\cE(S')=0$, 
which in turn implies that $\bar S'=T'$.
Hence $S$ contains all edges in $T$, 
except the branch $B=e''$ of $T$.

{\bf Intermediate cases.} 
Finally, suppose that 
$0<e(S\cap L)<\binom r2 -2$
and that $S$ contains $x$ many of the vertices in 
$L$ that are not in $e'$.
In this case, 
\[
\cE(S)
\ge \cE(S') + \lambda x - e(S\cap L)
>\cE(S')\ge 0,
\]
where the strict inequality follows by 
Lemma \ref{L_cEL}.

\subsubsection{Deficiency: Proof of Lemma \ref{L_strTWB}(2)}
\label{S_strTWB2}

Once again, the proof is by induction on $k$, in the 
same setting as in the proof of part (1) of the Lemma. 
The base case $k=0$ is trivial, as then $S = \emptyset $,
and so $P=\emptyset$ as well, in which case $\cE=0$
and $\cD=0$. 

Recall that in our 
proof of Lemma \ref{L_strTWB}(1) above, 
we showed that 
$\cE(S) $ is greater than or equal to 
either 
$\cE(S') $ or $\cE(\bar S')$, 
depending on the case. 

In the current proof, we will show that if 
there is no increment in $\cE$, 
then there is no increment in $\cD$.
In addition, we observe that if 
$\cE$ does increase, the its increment is 
bounded away from $0$ by Lemma \ref{L_cEL}. 
This will complete the proof, 
noting that increments of $\cD$ are all at most $r$.

{\bf Empty case.} 
If $S$ contains no edges in $L$
and no vertices in $L$, except possibly those in $e'$, 
then we have that $S=S'\subsetneq T'$, and so 
$\cE(S)=\cE(S')$ and $\cD(S,T)=\cD(S',T')$. 
Therefore, this case follows 
by the inductive hypothesis applied 
to $S'$ and $T'$.

{\bf Complete case.}
Recall that, if $L\subset S$, then 
$\bar S'\subsetneq T'$ and 
$\cE(S)=\cE(\bar S')$.
We claim that 
$\cD(S,T)=\cD(\bar S',T')$.
Indeed, it is clear that the $r-2$ 
vertices of $L$ that are not in $e'$
belong to $F(S,T)$. 
Furthermore, for each $v\in e'$,
\[
E_v(T)=(E_v(T')\setminus e')\cup E_v(L). 
\] 
Therefore, for each such $v$, we have that 
$v\in P(S,T)$ if and only if $v\in P(\bar S',T')$,  
and so the result follows by the inductive 
hypothesis applied to $\bar S'$ and $T'$.

{\bf Nearly-complete case.}
Recall that when 
$S$ contains all except one edge $e''$ of $L$,  
we showed that
$\cE(S)\ge\cE(S')$ with equality if and only if 
$S'$ contains all of the vertices
in $e'$ that are not in $e$. Suppose that equality holds, 
and consider the following two cases. 
First, if $\cE(S')=\cE(S)=0$ then $S$ contains all the edges of $T$ 
except $e''$ whence $\cD(S,T)\le 2$ 
(if either of the vertices
in $e''$ are in $e$ then they are not 
accounted for in $\cD(S,T)$ so there is inequality here). 
Second, if $\cE(S')>0$ then $\bar S'$ is also a proper subgraph of 
$T'$ and $\sigma(S')=\sigma(\bar S')$ because $S'$ contains 
all of the vertices in $e'$ that are not in $e$. 
Therefore, $\cE(S)=\cE(S')=\cE(\bar S')+1$, 
and the claim follows by induction on $\bar S'$.
On the other hand, recall that if 
$\cE(S)>\cE(S')$ 
then some $v\in e'$ (for which $v\notin e$) is not in $S'$, 
in which case the increment
$\cE(S)-\cE(S')\ge \lambda$.

{\bf Intermediate cases.}
Recall that, when 
$0<e(S\cap L)<\binom r2 -2$, 
we noted that Lemma \ref{L_cEL} implies
$\cE(S)>\cE(S')$.
However, this lemma in fact gives  
$\cE(S)-\cE(S')\ge\xi$, 
for some $\xi>0$ depending only on $r$.

\subsubsection{Extendability: Proof of Lemma \ref{L_strTWB}(3)}
\label{S_strTWB3}

Finally, we turn to the third part of the lemma, 
regarding vertex-minimal tree extensions of $S$. 

As in the previous parts, 
the proof is by induction on $k$. 
However, the case analysis will 
be different. 

The base case $k=0$ is trivial, 
so we proceed with the 
inductive step.

{\bf Case 1.} 
If $e(S\cap L)=0$, 
then the claim follows directly by induction.

{\bf Case 2.} 
If $\bar S'=T'$ then $\cE(S')=0$, and so 
\[
\cE(S) = \lambda x - e(S\cap L),
\]
where, as before, $x$ is the number of vertices in 
$S\cap L$ that are not in $e'$.

Taking $T_*=T$, we find that $\cT(S)\le r-2-x$, 
since
$S\subsetneq T$ and $\sigma(T)=\sigma(S)+(r-2-x)$. 

There are two sub-cases to consider:

{\bf Case 2a.} 
If $e(S \cap L)=\binom r2 -2$ then $x=r-2$ and $\cT(S)=\cE(S)=0$.

{\bf Case 2b.} 
Otherwise, $0<e(S\cap L)<\binom r2 - 2$, 
since $\bar S'=T'$ and $S\subsetneq T$. 
Hence,
by Lemma \ref{L_cEL}, we have that $\cE(S)\ge\xi$, 
and so
$\cT(S) \le c\cE(S)$.

{\bf Case 3.} 
Suppose that 
$e(S\cap L)>0$ and $\bar S' \subsetneq T'$. 
Note that 
\[
e(S)=e(\bar S')+e(S\cap L)-1\,,
\]
where the $-1$ accounts for the removal of $e'$.

Denote by $\bar T_*$ the TWG for $e$ 
satisfying $\bar S'\subsetneq \bar T_*$ and 
\[
\sigma(\bar T_*)=\sigma(\bar S')+\cT(\bar S').
\]
By induction, 
$\cT(\bar S')\le c\cE(\bar S')$.

There are three sub-cases to consider:

{\bf Case 3a.} 
If both of the vertices in $e'$ are in $S\cup e$, 
then it follows that 
$\sigma(S)=\sigma(\bar S')+x$. 
Therefore,
\[
\cE(S)=\cE(\bar S')+\lambda x-e(S\cap L)+1.
\]
By Lemma \ref{L_cEL}, we find that 
$\cE(S) \ge \cE(\bar S')$, 
where equality holds if and only if $L\subset S$, 
and otherwise in fact $\cE(S) \ge \cE(\bar S')+1$.

We construct $T_*$ from $\bar T_*$ 
by removing $e'$ and grafting $L$ in its place. 
As such, $v(T_*)=v(\bar T_*)+r-2$, and so 
\[
\cT(S)
\le \cT(\bar S')+(r-2-x).
\]
Consequently, 
if $x=r-2$ then 
$\cT(S) \le c\cE(S) $, 
and otherwise, we have that 
$\cT(S) \le \cT(\bar S')+(r-3)$ and 
$\cE(\bar S')\le \cE(S)-1$, 
and so  
$\cT(S)\le c\cE(S)$.

{\bf Case 3b.}  
If neither of the vertices in $e'$ are 
in $S\cup e$, 
then $x\ge2$, and equality holds 
if and only if $S\cap L$ is an edge. 
Furthermore, 
\[
\sigma(S)-\sigma(\bar S')
=x-2,
\] 
where the $-2$ 
accounts for the vertices in $e'$. 
Therefore,
\[
\cE(S)-\cE(\bar S')
=\lambda(x-2)-e(S\cap L)+1.
\]

We construct $T_*$ as follows:

{\bf Case 3bi.}  
If $x>2$, 
then note that $F=S\cap L$ has 
$3\le v(F)=x<r$ and $e(F)=e(S\cap L)$. 
Hence (see Definition \ref{D_xi}) $\cE(S)-\cE(\bar S')\ge\xi$. 
Therefore, in this case, 
we can obtain $T_*$ from $\bar T_*$ 
simply by removing $e'$ and grafting $L$ in its place. 
Then $\cT(S)\le\cT(\bar S')+r-3$, 
and the claim follows.

{\bf Case 3bii.}  
Suppose that $x=2$, 
so that $S\cap L$ is a single edge $f$. 
Then $\cE(S)=\cE(\bar S')$.
In this case, we obtain $T_*$ from $\bar T_*$
by a more delicate construction.  
The main observation is that
$e'$ is an isolated edge in $\bar S'$, 
since neither of the vertices in $e'$ are 
in $S\cup e$. Likewise, $f$ is an isolated edge in $S$. 
Hence, to obtain $T_*$ from $\bar T_*$, 
we replace the edge $e'$ in $\bar T_*$ with the edge $f$. 
We observe that: 
\begin{itemize}
\item $S\subsetneq T_*$
(since neither vertex in $e'$ is in an edge of $S$); 
\item $T_*$ is a TWG for $e$ 
(since neither vertex in $e'$
is in $e$); and 
\item $v(T_*)=v(\bar T_*)$. 
\end{itemize}
By induction, we find that  
\[
\cT(S)
\le \cT(\bar S')
\le c\cE(\bar S')
=c\cE(S),
\]
as required.

{\bf Case 3c.} 
Finally, suppose that one vertex 
$u$ in $e'$ is in $S\cup e$ 
but that the other vertex $v$ in $e'$ is not. 
This case is somewhat similar to the previous Case 3b. 
For completeness, let us give the details. 

First, note that $x\ge1$,  
\[
\sigma(S)-\sigma(\bar S')
=x-1
\]
and 
\[
\cE(S)-\cE(\bar S')
=\lambda(x-1)-e(S\cap L)+1.
\]

We construct $T_*$ as follows:

{\bf Case 3ci.}  
If $x>1$, we obtain $T_*$ from $\bar T_*$ 
by removing $e'$
and grafting $L$ in its place. 
By similar reasoning as in Case 3bi, 
we find that
$\cT(S)-\cT(\bar S') \le r-3$ and $\cE(S)-\cE(\bar S')\ge\xi$, 
and the claim follows.

{\bf Case 3cii.}  
Otherwise, if $x=1$ then $S\cap L$ 
is a single edge $e''$
between $u$ (the vertex in $e'$ that is in $S\cup e$) 
and some vertex $v''$
in $L$ that is not in $e'$. 
In this case, we construct $T_*$ from $\bar T_*$ 
by replacing $v$ with $v''$. 
Note that 
\begin{itemize}
\item $S\subsetneq T_*$
(since $v$ is not in any edge of $S$); 
\item $T_*$ is a TWG for $e$ 
(since $v$ is not in $e$); and 
\item $v(T_*)=v(\bar T_*)$, 
\end{itemize}
and by induction 
\[
\cT(S)
\le \cT(\bar S')
\le c\cE(\bar S')
=c\cE(S).
\]

This concludes the proof
of all parts in Lemma \ref{L_strTWB}.

\subsection{Enumeration of partial TWGs}
\label{S_EBs}

Recall (see Lemma \ref{clm:kTWG_edges}) 
that $k$-TWGs have $s+2$ vertices
and $\lambda s+1$ edges, 
where $s=(r-2)k$ is its number of vertices
that are not in $e$. 
By Lemma \ref{lem:count_TWG} 
(and \eqref{E_tkasy}) 
there are 
\[
s^{O(1)}\gamma^{s}s!
\]
many $k$-TWGs. 

The following lemmas
establish similar results 
for the number of partial TWGs $S$ 
and the number of 
ways to complete such an 
$S$ to a TWG $T$, 
up to an error depending on its 
edge efficiency 
parameter $\cE(S)$.

\begin{lemma}[Counting partial TWGs]
\label{lem:countXs}
The number of partial TWGs $S$, for 
a given edge $e$, with  
\begin{itemize}
\item $\sigma$ many vertices not in $e$; and 
\item edge efficiency $\cE$ 
\end{itemize}
is at most 
\[
\sigma^{c\cE}\gamma^{\sigma}\sigma!,
\]
for some constant $c$ depending only on $r$. 
\end{lemma}

\begin{proof}
By Lemma \ref{L_strTWB}(3), 
every such partial TWG $S$ 
is contained in some TWG 
$T_*=T_*(S)$
with $\sigma(T_*)\le \sigma+c\cE$. 
In particular, every such $S$
is contained in a $k$-TWG with $k=\lceil (\sigma+c\cE)/(r-2) \rceil$. 
Therefore, we can construct such an $S$ by first choosing 
such a $k$-TWG $T$ and then selecting 
$\lambda\sigma-\cE$ of its edges. 
By Lemma \ref{lem:count_TWG}, 
the number of possible choices is at most
\[
\gamma^{(r-2)k}[(r-2)k]!
{\lambda(r-2)k+1\choose\lambda\sigma-\cE}
\le \sigma^{c\cE}\gamma^{\sigma}\sigma!,
\]
as claimed. 
\end{proof}

\begin{lemma}[Counting extensions of partial TWGs]
\label{lem:comXs}
Suppose that $S$ is a 
partial TWG for a given edge $e$, 
with 
\begin{itemize}
\item $\sigma=\sigma(S)$ many vertices not in $e$; and 
\item edge efficiency $\cE=\cE(S)$. 
\end{itemize}
Suppose also that $V$ 
is a vertex set of size $(r-2)k+2$
that contains 
\begin{itemize}
\item all vertices in $S$; 
\item the edge $e$; and 
\item an additional 
$\sigma'=(r-2)k-\sigma$
other vertices.
\end{itemize}
Then the number of $k$-TWGs $T$ for $e$ 
on $V$  
such that $S\subsetneq T$ is at most
\[
k^{c(\cE+1)}\gamma^{\sigma'}\sigma'!,
\]
for some constants $c$ 
depending only on $r$. 
\end{lemma}

\begin{proof}
For a fixed $S$, there is a bijective correspondence 
between a $k$-TWG $T$ for $e$ containing $S$ 
and the complement $\tilde S$ induced
by all the edges in $T$ that are not in $S$.  
Since $F(\tilde S,T)=D(S,T)$ and 
$P(\tilde S,T)=P(S,T)$, 
it follows that 
\[
\sigma(\tilde S)
=(r-2)k-\sigma(S)+\cD(S,T).
\]
Hence, by Lemma \ref{L_strTWB}(2), 
\[
\sigma(\tilde S)+\sigma(S)
\le (r-2)k+c\cE(S)+2.
\]
Therefore, 
since $e(S)+e(\tilde S)=\lambda(r-2)k+1$,
we find that 
\begin{align*}
\cE(S)+\cE(\tilde S)
&= \lambda (\sigma(S) +\sigma(\tilde S))-(e(S)+e(\tilde S)) \\
&\le \lambda((r-2)k+c\cE(S)+2) - (\lambda(r-2)k+1) \\
&\le c(\cE(S)+1). 
\end{align*}
Altogether, 
$\sigma(\tilde S)\le \sigma' + c(\cE(S)+1)$ and 
$\cE(\tilde S)\le c(\cE(S)+1)$.
Therefore, by applying 
Lemma \ref{lem:countXs},
we find that there are at most
\[
k^{c(\cE+1)}\gamma^{\sigma'}\sigma'!
\]
such $\tilde S$, 
which completes the proof. 
\end{proof}

\subsection{Sharp upper bound}
\label{S_pcUB}

In this section, we prove 
the sharp upper bound 
Proposition \ref{P_pcUB}.

To begin, let us return to the setting of 
Lemma \ref{clm:mu}, wherein 
we fix some $e\in E(K_n)$ and let $B_k^{(e)}$ be 
the number of balanced $k$-TWGs for $e$ contained
in $\cG_{n,p}$. 
We now, more specifically, suppose that $k=\beta\log n$, 
for some $\beta>0$
to be determined below (but will depend only on $r$). 

Since $1\ll \log n\ll \sqrt{n}$, 
Lemma \ref{clm:mu} applies, and so  
\begin{equation}\label{E_mu}
\mu=\E(B_k^{(e)})\sim 
c_\beta   
(\log n)^{-c}
(1+\eps)^{(r-2)k}
p,
\end{equation}
as $n\to\infty$, 
for some constant $c_\beta>0$ depending only on $r$ and $\beta$
(and, as always, $c$ is some constant depending only on $r$). 

Let us write $T_1\sim T_2$ if 
$T_1$ and $T_2$ are two distinct 
balanced $k$-TWGs for $e$ such that $e(T_1\cap T_2)>0$. 
We put 
\[
\Delta 
= \sum_{T_1\sim T_2} \P(T_1,T_2 \subset \cG_{n,p})
=\sum_{T_1\sim T_2} p^{2\lambda(r-2)k+2-e(T_1\cap T_2)}.
\]

We will use
Janson's inequality, which we recall 
asserts that 
\[
\P(B_k^{(e)}=0)\le 
\exp\left(-
\frac{\mu^2}{\mu+\Delta}
\right).
\]

In order to apply this inequality,  
we will first bound $\Delta / \mu^2$ from above. 
We will use the fact that
\[
\frac{\mu^2}{\mu+\Delta}
=
\left(\frac 1\mu + \frac {\Delta}{\mu^2}\right)^{-1}
\]
and show that $1/\mu$ and $\Delta/\mu^2$
are both $\ll n^{-c}$, for some $c>0$.

Observe that
\begin{align*}
\Delta 
&= \sum_S\sum_{T_1\cap T_2=S}
p^{2\lambda (r-2)k+2-e(S)} \\ 
&\le \sum_S |\{T:S\subset T\}|^2 p^{2\lambda (r-2)k+2-e(S)} ,
\end{align*}
where the summation is over all non-empty partial TWGs $S$ 
of balanced $k$-TWG for $e$ (with labels in $[n]$) and, 
for each such $S$, $|\{T:S\subset T\}|$ is the set of all such $T$
that contain $S$. 

Then, by Lemmas \ref{lem:countXs} and \ref{lem:comXs}, 
we find that
\begin{align*}
\Delta 
&\le  
\sum_{\sigma,\cE}
{n\choose\sigma} \sigma^{c\cE}\gamma^{\sigma}\sigma! 
\left({n\choose\sigma'} k^{c(\cE+1)}\gamma^{\sigma'}\sigma'!\right)^2
p^{2\lambda (\sigma+\sigma')+2-(\lambda \sigma-\cE)}\\
&\le 
\sum_{\sigma,\cE} (1+\eps)^{\sigma+2\sigma'}
p^{2+\cE}
k^{c(\cE+1)}, 
\end{align*}
where the summation is over all $\sigma$ and $\cE$ 
for which there exists a partial TWG $S$ of a balanced 
$k$-TWG for $e$ with these specific 
parameters $\sigma(S)=\sigma$
and $\cE(S)=\cE$,  
and where, for convenience, we let $\sigma'=(r-2)k-\sigma$.
Note that, by \eqref{E_mu}, 
\[
\mu\ge k^{-c}
(1+\eps)^{\sigma+\sigma'}
p. 
\]
Therefore, 
\begin{equation}\label{eq:Delmu2}
\frac{\Delta}{\mu^2}
\le 
k^c\sum_{\sigma,\cE}(1+\eps)^{-\sigma}(pk^c)^\cE.
\end{equation}

We claim that the summation in 
\eqref{eq:Delmu2} 
is $o(n^{-c})$, for some $c$ depending only on $r$. 
To prove this, we will take two cases, 
with respect to whether the partial TGW
$S$ is edge inefficient $\cE>0$ or edge
efficient $\cE=0$. In the latter case, 
the assumption that $T$ is balanced
will finally come into play.

{\bf Case 1.} 
There are $O(k^2)$ summands with $\cE >0$. 
Note that, since $\cE(S)=\lambda\sigma(S)-e(S)$,  
it follows that $\cE \ge 1/(r-2)$. 
Also note that, trivially, $(1+\eps)^{-\sigma}\le 1$. 
As a consequence, 
recalling that $k=O(\log n)$, 
we deduce that the total contribution of 
these summands to \eqref{eq:Delmu2} 
in this case
is at most $n^{-(1+o(1))c}$, 
where $c=1/((r-2)\lambda)$, 
which only depends on $r$, as required.

{\bf Case 2.} 
If $\cE(S)=0$ for some partial TWG $S$ 
of a TWG $T$ 
then, by Lemma \ref{L_strTWB}(1), 
$S$ is obtained from $T$ by removing a single branch $B$. 
Since we may assume that $S\neq\emptyset$, it follows that 
$B\neq S$. In other words, $B$ is a branch of one of the 
${r\choose2}-1$ primal branches of $T$. 
Therefore, $S$ contains all vertices in the root $H^{(e)}$ of $T$
and ${r\choose2}-2$ of the primal branches of $T$
in full. 
Since $T$ is balanced, each primal branch contains
$(r-2)q$ vertices outside the root, were $q$ is as in 
\eqref{E_q}. Hence, 
\[
\sigma(S) 
\ge 
(r-2)+(r-2)q\left({r\choose2}-2\right)
\ge ck. 
\]
As such, the $O(k)$ summands in \eqref{eq:Delmu2} 
in which $\cE=0$ also have that $\sigma\ge ck$. 
Consequently, their total contribution 
is at most $k^c(1+\eps)^{-ck}$. 
Recall that $k=\beta\log n$. Therefore, by taking 
$\beta$ large,  
the total contribution is 
at most $n^{-c}$, for some $c>0$ depending only on $r$. 

With the claim above now proved, we now have that 
\[
\Delta/\mu^2
\ll (k/n)^c, 
\]
for some $c>0$ depending only on $r$. 
Also note that by \eqref{E_mu} and taking $\beta>0$ 
large, we may assume also that $\mu\gg n^{c}$. 

Therefore, by Janson's inequality, 
\[
\P(B_k^{(e)}=0) 
\le 
\exp(n^{-c})\ll n^2.
\]
Then, taking a union bound over all 
$e\in E(K_n)$, we find that, with high probability, 
$\cG_{n,p}$ contains a balanced $k$-TWG 
for all such $e$. 
In particular, $\langle \cG_{n,p} \rangle_{K_r}=K_n$, 
as claimed.

\section{Towards the lower bound}
\label{S_LB}

In this section, we prepare for the
proof of the upper bound 
in Theorem \ref{T_pc}.
Throughout, we let $H=K_r$, for some fixed $r\ge5$. 
However, for convenience often 
write $H$ to denote a copy of $K_r$. 

\subsection{A sharp lower bounds for TWGs}
\label{S_TWGthreshold}
Recall that in Section \ref{S_UB}, we showed that if 
\[
p=\left(\frac{1+\eps}{\gamma n}\right)^{1/\lambda},
\]
for some $\eps>0$, then, with high probability, all 
edges can be 
added 
to $\langle\cG_{n,p}\rangle_{K_r}$ 
by a TWG of logarithmic size. 
The first observation we make towards a matching lower bound, is that 
TWGs are much less likely
on the other side of this threshold.

\begin{lemma}[TWG threshold]
\label{L_lbTWG}
Fix some $e\in E(K_n)$. 
Let $\eps>0$ and put 
\[
p=\left(\frac{1-\eps}{\gamma n}\right)^{1/\lambda}.
\]
Then, with high probability, 
there is no TWG $T$ for $e$ in $\cG_{n,p}$. 
\end{lemma}

\begin{proof}
By Lemma \ref{lem:count_TWG},
there are $O(\gamma^\sigma \sigma!)$ 
(labelled) TWGs for a given edge $e$
of size $\sigma=(r-2)k$. 
Such TWGs have 
$e(T)=\lambda \sigma+1$ many edges. 
Therefore, the probability that a given $e$ has
a TWG in $\Gnp$ is bounded by 
\begin{align*}
p\sum_{s=0}^\infty 
p^{\lambda s}{n\choose s}
O(\gamma^s s!)
&\le O(p)\sum_{s=0}^\infty (\gamma np^\lambda)^s\\
&=O(p)\sum_{s=0}^\infty (1-\eps)^s
=O(p/\eps)\to0,
\end{align*}
and the result follows. 
\end{proof}

This result 
suggests  
that a sharp threshold for  $K_r$-percolation  occurs at
$p_c\sim(\gamma n)^{-1/\lambda}$. 
However, TWGs are only the simplest possible 
type of witness graph (formally defined in the next section).
To prove the sharp lower bound in 
Theorem \ref{T_pc}
we will require a reasonably precise 
understanding of the structure of general 
witness graphs, and specifically, how far from 
tree-like they can be as a function of their edge density.

\subsection{Witness graph algorithm (WGA)}
\label{S_WGA}

For each edge $e\in E(\langle G \rangle_H)$
eventually added to $G$ by the $H$-dynamics, 
we assign a  
{\it witness graph (WG)} 
$W^{(e)}\subset G$
that bears witness to the addition of $e$, 
in the sense that 
$e\in E(\langle W^{(e)} \rangle_H)$. 
The graph $W^{(e)}$ is defined by the recursive procedure
introduced in  \cite{BBM12}, 
which we call the {\it witness graph algorithm (WGA)}.

\begin{definition}[Witness graphs]
\label{D_WG}
Recall that $E_t$ denotes the set of edges added to a graph 
$G$ in the $t$th round of the $H$-dynamics. 
For each $e\in E(\langle W^{(e)} \rangle_H)=\bigcup_{t\ge0}E_t$, 
we define $W^{(e)}$ as follows. 
\begin{enumerate}
\item If $e\in E_0$, we simply put $W^{(e)}=e$. 
\item Otherwise, if $e\in E_t$, for some $t\ge1$, 
we select a copy $H^{(e)}$ of $H$, for which 
$E(H^{(e)}\setminus e)\subset \bigcup_{s<t}E_s$, 
and put 
\begin{equation}\label{E_WGdef}
W^{(e)}=\bigcup_{f\in E(H^{(e)}\setminus e)} W^{(f)}. 
\end{equation}
\end{enumerate}
\end{definition}

Note that, in the above algorithm, 
there is potentially 
some
flexibility in choosing the copy $H^{(e)}$ of $H$. 
However, any arbitrary choice will do, 
and will have no effect on our arguments.

\subsubsection{Aizenman--Lebowitz property}

As shown in \cite{BBM12}, 
WGs have the following  useful 
property, reminiscent of an analogous 
property for $r$-neighbor
bootstrap percolation on $\Z^d$, 
first observed in that context by 
Aizenman and Lebowitz \cite{AL88}.

\begin{definition}[Size of WG]
For convenience, 
we let $\sigma (W)=v(W)-2$ denote the {\it size} of a WG. 
In other words, $\sigma(W)$ is the number of vertices
in $W=W^{(e)}$ that are not in $e$. 
\end{definition}

Put 
\[
m(t)=\max_{e\in \bigcup_{s\le t}E_s} \sigma(W^{(e)}). 
\]
Note that, if 
$E_0,E_1\neq\emptyset$, 
then $m(0)=0$ and $m(1)=v(H)-2$. 
On the other hand, for $t>1$,
by the recursive structure of WGs, 
\[
m(t)
\le v(H)-2+(e(H)-1)m(t-1)
\le e(H)m(t-1),
\]
since $v(H)-2=m(1)\le m(t-1)$. 
In other words, in each round of the $H$-dynamics, 
$m(t)$ expands by at most a factor of $e(H)$. 
As such, we have the following result; 
see Lemma 13 in \cite{BBM12}
and Lemma 8 in \cite{BK24}.

\begin{lemma}[A--L for WGs]
\label{L_AL}
Suppose that, for some edge $e$, its WG is of
size $\sigma(W^{(e)})\ge v(H)-2$. 
Then, for all $ \sigma'\in [v(H)-2, \sigma(W^{(e)})]$, 
there is some 
edge $f$ with WG of size $\sigma(W^{(f)})\in [\sigma',  e(H) \sigma']$. 
\end{lemma}

In order to identify the 
sub-critical region for 
$K_r$-percolation, we will 
use Lemma \ref{L_AL} 
as follows. If a given edge
$e$ is eventually added by the $K_r$-dynamics, 
then it has a witness graph $W$. 
Either $W$ is of size at most $\beta\log n$, 
or else {\it some} edge $f$
has a witness graph $W'$
of size between $\beta\log n$
and ${r\choose 2}\beta\log n$. Using these facts, 
our goal is to show that
$\P(e\in E(\Gnp))\to0$.

\subsection{Red edge algorithm (REA)}

We will use the following 
{\it red edge algorithm (REA),}
as introduced in \cite{BBM12}, 
to study non-tree WGs. 

We assume that $e\notin E(G)$, so that $W^{(e)}\neq e$. 

This algorithm describes the formation 
of $W=W^{(e)}$ for a given 
$e\in \langle G\rangle_H$. 
Essentially, the REA is obtained from the WGA 
(which recall describes the formation of {\it all} WGs)
by suppressing all steps in the WGA that do not
contribute to the formation of $W$.
For technical convenience,  
we also ``slow down'' the dynamics, so that 
in each step a single copy of $H$ is completed. 

More formally, 
$W$ can be described in terms of a series of 
copies of $H$, 
\[
H^{(e_1)},H^{(e_2)},\ldots,H^{(e_m)},
\]
such that 
(1) $e_m=e$ and (2) for all $i<j$, if $e_i\in E_t$ 
then $e_j\notin E_s$ for all $s<t$. 
In other words the $e_i$ are weakly time ordered, with respect
to the time of their addition to $G$. 
The edges $e_i$ are distinct and not in $G$.

For ease of notation, let us write $H_i=H^{(e_i)}$. 
In each step $j\le m$ of the REA, we add all edges in 
$H_j$ that are not in 
$\bigcup_{i<j}H_i$. 
One such edge $e_j$ is colored red. 
All other (if any) edges are in $G$ and colored black. 
Note that the color
of any other edges  in $H_j$ have already been 
assigned by the REA in previous steps $i<j$. 
After all $m$ steps of the REA have been completed, 
$W$ is obtained as the graph induced 
by all the black edges. 

It will sometimes be helpful to write 
$H_i^*$ to emphasize that the REA 
defines an edge-colored version of $H_i$.
In this way, $W$ is simply obtained by removing all the
red edges from $\bigcup_{i=1}^m H_i^*$.

\subsubsection{REA components}
\label{S_component}

The REA gives rise to a natural component 
structure, which evolves in time with the algorithm. 

Recall that after the $k$th step of the REA, 
we have added $k$ copies 
$H_1\ldots,H_k$ of $H$. 
Consider an auxiliary graph $\cA_k$, with a vertex
$a_i$ for each 
$H_i$, with $i\le k$. Two such vertices $a_i$ and $a_j$, 
are neighbors
if their corresponding $H_i$ and $H_j$ 
share at least one (red or black) edge. 

Naturally, each 
connected component $\cC$ of $\cA_j$ corresponds
to an edge-colored graph   
\[
C=\bigcup _{a_i\in\cC}  H_i^*,
\]
which we refer to as an  {\it component} existing
after the $j$th step of the REA. 
Note that all such components are (red and black) edge disjoint 
(but can share vertices).

\subsubsection{REA step types}

We saw that a TWG $T$ of size $\sigma$ 
has $\lambda \sigma+1$ edges. 
It turns out that when $H$ is a clique  
(see \cite{BBM12,BK24}) this is the minimal 
possible number of edges in a WG.
As such we make the following definition.

\begin{definition}[Excess edges]
We call 
\[
\chi(W)=e(W)-(\lambda \sigma(W)+1)\ge 0.
\]
the {\it excess number of (black) edges} in $W$. 
\end{definition}

In order to study the combinatorics of WGs with
a given excess $\chi$,
it will be helpful to differentiate between various types of steps in the REA.

\begin{definition}[Tree step (TS)]
The $i$th step of the REA, when $H_i$ is added, 
is called a {\it tree step (TS)} 
if it starts a new component or 
if all of the components 
being merged with $H_i$ contain a single edge
in $H_i$ and are vertex disjoint 
outside of $H_i$.
\end{definition}

\begin{definition}[Internal red step (IntR)]
The $i$th step of the REA, when $H_i$ is added, 
is called an {\it internal red step (IntR)} if  (i) $H_i$ merges 
with a single existing component $C$, (ii) $V(H_i)\subset V(C)$, 
and (iii) only a red edge is added to $C$.
\end{definition}

In other words, an IntR step can be the first step that occurs when running the 
$K_r$-dynamics internally on some component $C$.

\begin{definition}[Costly steps]
The $i$th step of the REA is called a {\it costly step} 
if it is neither a TS nor an IntR. 
The associated {\it cost} $\kappa_i$ to the step 
is $r=v(H)$ plus the number of vertices 
(if any) 
outside $H_i$
that are in at least 
two of the components being merged
with $H_i$. 
If the $i$th step is not costly (i.e., TS or IntR)
then we 
put $\kappa_i=0$. 
The {\it total cost} of $W$ is 
\[
\kappa(W)=\sum_{i=1}^m \kappa_i. 
\]
\end{definition}

The term {\it cost} is justified by the next result 
which follows directly by the 
proof of Lemma 12 in \cite{BK24}.

\begin{lemma}[$\chi$ is $\Omega(\kappa)$]
\label{L_chikappa}
Let $H=K_r$, with $r\ge5$. 
Then its excess number of (black) edges in $W$ 
satisfies 
\[
\chi(W)\ge \xi \kappa(W),
\] 
where $\xi=\xi(H)>0$ is some constant, 
depending only on $H$. 
\end{lemma}

In other words, each costly step 
is associated with an irreversible contribution
to the excess $\chi$. 

We note that the above bound 
follows by the inductive proof of 
Lemma 12 in \cite{BK24}, 
which shows 
that this property 
holds for all components at any given time of the REA. 
(Naturally, we can extend the definitions of $\sigma$, $\chi$ 
and $\kappa$ to components $C$.)
Since 
there is a single component
by the end of the process, the result follows. 

Although there is a ``price'' to be paid for 
each costly step, they can also lead to the 
spread of (zero-cost) IntR steps. Indeed, 
the main challenge in analyzing general WGs 
is the delicate interplay between costly
and IntR steps.

\subsection{Tree Components}
Our proof strategy 
for the lower bound hinges on the following ``stability'' property: 
a WG with a small excess is ``close", 
in some combinatorial sense, to a TWG. 
A basic version of this property is proved in Section \ref{S_LBuptoC},
and a more refined version in Section \ref{S_sharpLB}.
For this purpose, we give the following key definitions:

\begin{definition}[Tree components]
\label{D_TCs}
A component $C$ which exists
after some step of the REA, for a witness graph $W$,  
that was formed with only 
TSs is called a {\it tree component (TC)}. 
\end{definition}

The union of black and
red edges in a TC
is of the following form.

\begin{definition}[$K_r$-tree]
\label{D_Kr_tree}
We call a graph $T$ a {\it $K_r$-tree}
(or a {\it clique tree})
of {\it order} $\vartheta$ if it is 
the union of $\vartheta$
many copies $H_1,\ldots,H_\vartheta$ of $K_r$, 
and that, 
for every $1<i\le \vartheta$, $H_i$ shares exactly 
one edge with $H_1 \cup\cdots \cup H_{i-1}$.
We call an edge $e\in E(G)$ an {\it internal edge} 
of $T$ if it is contained in at least two of the copies of $K_r$. 
\end{definition}

\begin{remark}
In Section \ref{S_UB}, we used
the letter $T$ to refer to TWGs in the $K_r$ dynamics. 
From now on, we use $T$
to refer to
$K_r$-trees. 
These structures are clearly related. Specifically, 
the union of red and black edges
in a TWG is a $K_r$-tree. In general, every TC is a 
$K_r$-tree whose edges are colored red or black. 
Conversely,  every $K_r$-tree can be attained as a 
TC via some edge coloring. 
\end{remark}

The following claim regarding the structure of 
$K_r$-trees will be useful throughout the 
remainder of this article.

\begin{claim}\label{clm:Kr_tree_structure}
Suppose that a $K_r$-tree $T$ is 
composed of $\vartheta$ copies $H_1,...,H_\vartheta$ of $K_r$
(as in Definition \ref{D_Kr_tree}). 
Then, we have the following. 
\begin{enumerate}
\item Every clique in $T$ is contained in some $H_i$.
\item If vertices $x\neq y$ have at least three 
common neighbors in $T$ then they are neighbors.
\end{enumerate}
\end{claim}
\begin{proof}
We prove this by induction on the order $\vartheta$ of $T$. 
The base case $\vartheta=1$ is trivial. We assume that the 
claim holds for all $K_r$-trees of order $\vartheta-1$ and 
that $T$ is a $K_r$-tree of order $\vartheta$. 
Let  $T'=\cup_{i<\vartheta}H_i$. 
We use the following observation: 
if a vertex belongs to 
$V_\vartheta=V(H_\vartheta)\setminus V(T')$ 
then all its neighbors are in $H_\vartheta$.

To prove the first claim, 
suppose that some vertices $v_1,...,v_j$ form a clique in $T$. 
If they are all in $T'$ then the claim follows by induction. 
Otherwise, one of these vertices is in $V_\vartheta$, 
whence all of them are in $H_\vartheta$, as claimed.

For the second claim, consider 
$x,y$ and their common neighbors $v_1,...,v_j$, for some $j\ge 3$.
If they all belong to $T'$ then the claim follows by induction. 
Therefore, at least one them is in $V_\vartheta$. 
If a common neighbor, say $v_1$, is in $V_\vartheta$, then 
$x,y \in V(H_\vartheta)$ and are therefore neighbors. 
Otherwise, we may assume that $x\in V_\vartheta$, 
whence $v_1,...,v_j$ belong to $V(H_\vartheta)$. 
At most $2$ of these common neighbors lie in the intersection 
of $H_\vartheta$ and $T'$, and since $j\ge 3$ 
at least one of them is in $V_\vartheta$, 
in which case, as we have already shown, 
it follows that $x,y$ are neighbors.
\end{proof}

\begin{lemma}\label{lem:Kr_stable}
If $T$ is a $K_r$-tree then $\langle T\rangle_{K_r}=T$, 
i.e.,
$T$ is {\it $K_r$-stable}.
\end{lemma}

\begin{proof}
If $H$ is a copy of $K_r$ and $H\setminus e \subset E(T)$ 
then the endpoints of $e$ have  $r-2\ge 3$ common neighbors 
and  $e\in E(T)$ by Claim \ref{clm:Kr_tree_structure}(2). 
Therefore, $T$ is $K_r$-stable,  as claimed.
\end{proof}

\subsection{Target edges}

Note that, in the REA, all of the red edges $e_i$ are added 
along the way towards eventually adding $e=e_m$. 
As such, we make the following definition. 

\begin{definition}[Target edges]
After the $k$th step of the REA for $W$, we call a red edge $e_i$, 
for some $i\le k$, a {\it target
edge} if it has not yet been {\it reused}, meaning that 
$e_i\notin E(H_j)$, for all $i<j\le k$. 
We call $e=e_m$ the {\it final target edge} of $W$. 
\end{definition} 

Note that, after the $k$th step, 
the most recent red edge $e_k$
is a target edge, but there can be more. 
However, by induction 
(recalling \eqref{E_WGdef} above), after the last $m$th step
of the REA, the final target edge $e=e_m$ is the only
remaining target edge. 
This observation, although quite simple, will play a crucial role
in our arguments. For instance, we use it now to show that TWGs 
are the only WGs with no excess $\chi=0$. This is a first step towards 
our goal of showing that a WG with a small excess is similar to a TWG. 

\begin{lemma}\label{lem:TC_with_one_target}
If $C$ is a TC that exists after some step of the REA with a unique target 
edge $e$ then $C$ is a TWG for $e$. In particular, if every step in the 
REA of a witness graph $W$ is a TS, then $W$ is a TWG.
\end{lemma}
\begin{proof}
We proceed by induction of the number $k$ of TSs used to form $C$. 
If $k=1$ the claim is trivial. Otherwise, $C$ is formed by a copy $H^{(e)}$ 
of $K_r$ that has a single edge in common with each of the components 
$C_1,...,C_h$ that are being merged. Observe that $C_1,...,C_h$ are TCs, 
since $C$ is a TC. Additionally, the edge $e_i \in E(C_i)\cap E(H)$, 
for each $1\le i\le h$, is the unique target edge of $C_i$, 
as otherwise $C$ would have 
another target edge besides $e$. 
Therefore, each $C_i$ is a TWG for $e_i$.
By 
Definition \ref{D_TWGroot}, we find that $C$ is a TWG for $e$, as required. 
\end{proof}

\begin{corollary}\label{cor:TWG_only_minimizers}
For every witness graph $W$ 
we have that 
$\chi(W)=0$ if and only if $W$ is a TWG.
\end{corollary}
\begin{proof}
The fact that $\chi(W)=0$ if $W$ is a TWG is clear. 
In the other direction, recall that if $\chi=0$ then all steps are 
either TS or IntR steps. 
However, in this case, if there was ever 
an IntR step during the REA,  then the first such step occurs in a TC, 
in contradiction to Lemma \ref{lem:Kr_stable}. Therefore, $W$ is a 
TC whose REA uses only tree steps, an so it is a TWG by  
Lemma \ref{lem:TC_with_one_target}.
\end{proof}

\section{Coarse lower bound}
\label{S_LBuptoC}

Since the full proof of the lower bound
in Theorem \ref{T_pc} is somewhat involved, 
we will first give an argument that 
proves that $p_c=\Omega(n^{-1/\lambda})$. 
These arguments already contain 
some of the main ideas, 
and, together the upper bound in Section \ref{S_UB}, 
answers Problem 3 in \cite{BBM12}.
(We caution the reader 
that the proof of the sharp lower bound in 
Section \ref{S_sharpLB} is not self-contained,  
but rather builds on the arguments in this section.)

\begin{proposition}[Coarse lower bound]
\label{P_pcLBcoarse}
For $r\ge5$, we have that 
\[
p_c(n,K_r)= \Omega(n^{-1/\lambda}). 
\]
\end{proposition}

We note that Proposition \ref{P_pcLBcoarse} fails if $r=4$, where  
$p_c$ is of order $(n\log n)^{-1/\lambda}$ (as discussed in Section \ref{S_Intro}). 
We point out that the assumption that $r\ge 5$ is indeed used in 
Lemma \ref{lem:Kr_stable} above and Claim \ref{C_KrDyn_on_union} below, 
where both claims playing a key role in controlling the spread of IntR steps.

\subsection{Tree decomposition}
\label{S_CTD} 

Consider the witness graph $W$ for 
some edge $e$, as given by the REA. 
Recall (see Section \ref{S_component}) 
that after the $j$th step of the REA, 
we have an edge-disjoint union of components. 
Each such component $C$ (of red and black edges) 
can be written as 
\[
C=\bigcup _{a_i\in\cC} H_i^*,
\]
where $\cC$ is the connected component
associated with $C$ in the 
auxiliary graph $\cA_{j}$. 
Recall that the $H_i^*$ are the red and black edge-colored
copies of $K_r$ added in the REA. 

By the following recursive procedure, 
we will decompose each such $C$ into the following subgraphs: a 
{\it bad part} $B$ (possibly empty) and some number 
of {\it tree parts} $T_1,\dots,T_k$ (possibly $k=0$) . 
We will call this the {\it tree decomposition}
of $C$. In the sequel we will show, by induction on the number of 
steps in the REA, that the decomposition of every 
component $C$ satisfies the following properties:

\begin{enumerate}
\item The underlying graph (red and black edges) 
of every tree part is a $K_r$-tree. 
\item If $B=\emptyset$ then $k=1$ and $C=T_1$ is a TC.
\item If $B\neq\emptyset$ then 
\begin{enumerate}
\item the intersection graph $T_i \cap B$ of every tree part 
$T_i$ and $B$ consists of a single edge that is not a target edge of $C$ 
(and is called the {\it base} of $T_i$); and
\item every two tree parts intersect in at most one vertex, 
and if they do, that vertex lies in $B$. In particular, 
the tree parts are pairwise edge disjoint.
\end{enumerate}
\end{enumerate}

Since these properties are proved inductively, we can (and will) 
use them to define the recursive procedure that constructs 
the tree decomposition.

For the base case of the recursive construction, suppose that 
$C$ is a single copy of $K_r$ with one red edge and 
$\binom r2-1$ black edges. 
In this case, we simply put 
$B=\emptyset$ and 
$T_1=C$ . 

Otherwise, let us suppose that $C$
is formed by adding $H_j^*$
and merging it with some
previously defined components $C_1,\ldots,C_h$. 
Each such $C_i$ has a bad part $B_i$
and tree parts $T^{(i)}_1,\ldots,T^{(i)}_{k_i}$. 
The bad part $B$ and tree parts $T_1,\dots,T_k$
of $C$ are defined recursively, with respect to the following cases:

{\bf Tree step.} If $C$ is formed in a TS, we claim that there is at 
most one tree part in each $C_i$ that contains an edge in $H_j^*$. 
Indeed, if $B_i$ is empty then the claim is trivial by property (2). 
Otherwise, by property (3b) the tree parts of $C_i$ are pairwise edge 
disjoint and a TS can intersect with $C_i$ in at most one edge. 
Let $T$ be the $K_r$-tree that is formed by $H_j^*$ and the tree parts 
that intersect it ($T$ is indeed a $K_r$-tree since $C_1,...,C_h$ 
are vertex-disjoint outside of $H_j^*$, since we are considering a TS), 
and let $T_1,...,T_s$ denote all the other tree parts in $C_1,...,C_h$. 
Consider two subcases:
\begin{itemize}
\item {\bf Standard tree step.} If there is at most one non-empty bad part among  
$B_1,...,B_h$, then set $B:=\bigcup_{i=1}^h B_i$ and the tree parts of $C$ 
to be $T_1,...,T_s$ and the new tree part $T$.
\item {\bf Bad tree step.} Otherwise, there are at least two non-empty bad parts. 
In this case, we set $B:=(\bigcup_{i=1}^h B_i)\cup T$ 
and the tree parts of $C$ to be $T_1,...,T_s$.
\end{itemize}

We stress that in the case of a bad tree step, the $K_r$-tree 
$T$ is {\it not} considered to be a tree part of $C$, but is rather used 
only as an {\it auxiliary tree} in the construction
of its decomposition.

{\bf IntR step.} If $C$ is formed in an IntR step, 
then $h=1$ and all the edges in $H_j^*$ except the new red edge 
$e_j$
are already contained in $C_1$. 
In this case, we simply 
obtain $B$ from $B_1$ by adding $e_j$, 
and the tree parts of $C$ are those of $C_1$, unchanged.

{\bf Costly step.} Finally, suppose that $C$ is formed in a costly step. 
Let $V_j$ be the vertex set containing $V(H_j)$
and all vertices outside of $V(H_j)$ that are in at least 
two of the $C_i$ being merged with $H_j^*$. 
Recall that $\kappa_j=|V_j|$ is the cost of this step. 
We call $V_j$ the set of vertices {\it involved}
with this costly step. 

In this case, we let $B$ be the union of 
$H_j^*$ with $\bigcup_{i=1}^h B_i$ and 
all tree parts $T$ in some $C_i$ such that 
$V(T)\setminus V(B_i)$ contains a  
vertex in $V_j$. The tree parts of $C$ are all the tree parts $T$ 
of the various $C_i$ for which $V(T) \setminus V(B_i)$ is disjoint of $V_j$.

\subsection{Properties of the tree decomposition}
\label{S_PropsTCdecomp}
In this section, we prove the structural properties of the tree decomposition
discussed in the previous section. 

\begin{lemma}\label{lem:tree_decomp_struct}
After every step of the REA for a witness graph $W$, the tree decomposition 
of any given existing component $C$ 
satisfies the properties (1)--(3) listed in Section \ref{S_CTD} above.
\end{lemma}

\begin{proof}
The proof is by induction on the number of steps in the REA that are used 
to construct a component $C$.

Recall that, in the base case, $C$ is a single copy of $K_r$, 
$B$ is empty and $T_1=C$. 
Therefore properties (1)--(3) follow directly.

\textbf{Standard tree step.} Suppose that $C$ is formed in a standard TS. 
If all the $B_i$ are empty, then by induction, every $C_i$ is a TC, in which case 
$C$ is a TC and properties (1)--(3) follow directly. 

Otherwise, suppose that $B_1\ne\emptyset$ 
and the other $B_i$ are empty, 
whence $B=B_1$. In addition, $C_2,..,C_h$ are TCs 
using property (2) by induction, 
and are therefore contained in the new tree part $T$. 
Note that the intersection of $H_j$ and $C_1$ is a single edge $e$, 
which is either in some 
tree part, $T'$ say, of $C_1$,
or else in $B_1$ but in no tree part of $C_1$. 
In consequence, the other tree parts $T_1,...,T_s$ 
are all the tree parts of $C_1$ 
besides $T'$ (if it exists). In particular, since $T_1,...,T_s$ 
have not been effected by this 
step and $B=B_1$ has not changed, 
properties (1)--(3) hold for them by induction.

It remains to verify property (3) for $T$. Since $B=B_1\subseteq C_1$ and  
the current step is a TS, the intersection of 
$T$ and $B$ is equal to the intersection 
of $T'$ and $B$, if $T'$ exists, and otherwise it is equal to the edge $e$. 
In former case, property (3a) holds for $T$ by induction applied to $T'$. 
In the latter case, property (3a) holds, as the edge $e$
has been reused to form $T$. 
Next, for property (3b),
note that, 
similarly, the intersection of $T$ and $T_\ell$, for any $1\le \ell\le s$, 
is equal to the intersection of $T'$ and $T_\ell$, if $T'$ exists, 
and otherwise it is 
equal to the intersection of $e$ and $T_\ell$. 
In the former case, property (3b) holds by induction. In the latter case, 
by induction, 
$T_\ell$ intersects $B$ in a single edge. This edge cannot be $e$, 
by our assumption that $e$ is in no tree part of $C_1$. 
Therefore, $e$ and $T_\ell$ intersect in at most one vertex, 
which belongs to $B$ if it exists.

\textbf{Bad tree step.} Suppose that $C$ is formed in a bad TS. 
Recall that, in this case, the auxiliary $K_r$-tree $T$ is not a tree part of $C$, 
but is rather contained in its bad part $B$. As such, property (1) follows 
directly by induction since no new tree parts are created, 
and property (2) holds vacuously since $B\neq\emptyset$. 

Also recall that, in this case, $B$ is the union of 
$B_1,...,B_h$ and $T$. Let $1\le \ell \le s$, and suppose that $T_\ell$ 
was a tree part of $C_i$, for some $1\le i \le  h$. Let $X$ denote the 
intersection of $T$ and $C_i$. 
As we saw in the previous case of a standard tree step, 
$X$ is either a tree part $T'$ of $C_i$, or a single edge in $B_i$. 
The vertex intersection of $T_\ell$ with $B$ is comprised of its vertex 
intersection with $B_i\cup X$, since the current step is a tree step, 
and so no other vertex of $B$ belongs to $C_i$.
Furthermore, we claim that the intersection of $T_\ell$ 
with $B_i \cup X$ is equal to its intersection with $B_i$. 
Indeed, this is clear if $X$ is an edge of $B_i$. 
Otherwise, 
if $X$ is the tree part $T'$, 
we derive the claim using property (3b) by induction, 
which asserts that the intersection of $T_\ell$ and $T'$ is contained in $B_i$. 
Using this claim, property (3a) for $T_\ell$ follows by induction. 

For property (3b), let $1\le l\ne l'\le s$ and $1\le i,i'\le h$ be given, 
such that  $T_\ell$ is a tree part of $C_i$ and $T_{\ell'}$ a tree part of $C_{i'}$. 
If $i=i'$, then (3b) follows directly by induction. 
Otherwise, since the current step is a TS, $T_\ell$ and $T_{\ell'}$ 
can intersect in at most one vertex from $H_j$ which 
(if it exists) must belong to $B$, since $H_j\subset B$.

\textbf{IntR step.} 
Next, suppose that $C$ 
is formed in an IntR step.
Recall that, in this case,
$h=1$ and $H_j\setminus e_j\subset C_1$. 
We obtain 
$B$ from $B_1$ by  
adding the new red edge $e_j$
to $B_1$, and all (if any) tree parts
$T$ of $C_1$ are unchanged. 
Therefore, as in the previous case, 
properties (1) and (3b) hold by induction, 
and property
(2) holds vacuously. Hence, it remains only to verify property (3a). 
To this end, we will use the following claim to deduce that, 
for all tree parts $T$, 
the intersection $T\cap B$
is the same as $T\cap B_1$.  

\begin{claim} \label{C_KrDyn_on_union}
Fix $r\ge 5$. Suppose that $G,G'$ are graphs that satisfy 
$|V(G)\cap V(G')|\le 2$. 
Then, the vertex set of every copy $H$ of $K_r$ that is completed by the 
$K_r$-dynamics on $G\cup G'$ is either contained in 
$V(G)$ or in $V(G')$. 
\end{claim}

\begin{proof}
   Denote $A=V(G)\setminus V(G')$ and $B=V(G')\setminus V(G)$.
Towards a contradiction, let $H$ be the first copy of $K_r$, 
adding some edge $e$
to $G\cup G'$, that violates the claim. 
Observe that $H\setminus e$ has no edges between $A$ and $B$, 
since no such edge was previously added by the $K_r$-dynamics. 
By our assumption on $H$, 
both $|V(H)\cap A|$ and $|V(H)\cap B|$ are positive.
But then, since $r\ge 5$, 
one of these  sets has at least two vertices, 
and so  
$H$ has at least two edges between $A$ and $B$, 
which is a contradiction.
\end{proof}

Let $T$ be a tree part of $C$.
Note that, to show that 
$T\cap B=T\cap B_1$, we need to show that, when the new edge 
$e_j$ is added to $B_1$, the tree part $T$ does not acquire 
a new vertex in $B$. 
By Lemma \ref{lem:Kr_stable}, 
$T$ is $K_r$-stable, and so at least one vertex in 
$H_j$ is outside of $T$. By induction, 
by property (3), 
the graphs $T$ and $(C_1\setminus T)\cup B_1$ intersect in a single edge, 
which is the edge in $T\cap B_1$. 
Therefore, combining Claim \ref{C_KrDyn_on_union} and 
$V(H_j) \not\subset V(T)$, we deduce that  $V(H_j)$ is disjoint from 
$V(T)\setminus V(B_1)$. In particular, it follows that
$T\cap B=T\cap B_1$, as required.

\textbf{Costly step.} 
Finally, suppose that $C$ is formed in a costly step. 
Property (1) holds by induction, and (2) holds vacuously. 
Let $T$ be a tree part of some $C_i$ such that 
$(V(T)\setminus V(B_i))\cap V_j=\emptyset$. 
Recall that such a $T$ remains a tree part
of $C$ after the costly step. 
We claim that $V(T)\cap V(B)=V(T)\cap V(B_i)$. 
Indeed, suppose, towards a contradiction, that there exists a vertex 
$v$ in $V(T)\cap V(B)$ that is not in $V(B_i)$. Then, 
in particular, $v\notin V_j$. 
More specifically, 
this means that 
$v$ is not in $V(H_j)$ and does not belong to any other component 
$C_{i'}$ with $i'\ne i$. 
Therefore, $v$ must lie in a tree part $T'$ of $C_i$ that was added to $B$.
However, by induction, the intersection $T\cap T'$ is contained in $B_i$, 
which establishes the claim. Consequently, property (3a) follows by induction. 

For property (3b), let $T,T'$ be tree parts of $C$ that belonged to 
$C_i,C_{i'}$ respectively prior to the step. If $i=i'$, then 
property (3a) holds directly by induction. Otherwise, 
$V(T)\cap V(T')\subset B_i\cap B_{i'}$. Indeed, every vertex of $T$ not in 
$B_i$ is not in $V_j$, hence it does not appear in $C_{i'}$ (and similarly for $T'$). 
Moreover, by induction, $T$ intersects $B_i$ in an edge, 
and likewise $T'$ intersects $B_{i'}$ in an edge. 
These edges cannot be identical, since 
$T$ and $T'$ lie in different components. 
Therefore, they intersect in at most one vertex, 
which must belong to $B$, as required.
\end{proof}

\subsection{Complexity bounds}

Our goal is to use the tree decomposition to bound the complexity 
of a witness graph $W$ in terms of its cost. 
We start with the following claim.

\begin{lemma}\label{lem:few_bad_TS}
Let $C$ be a  component  existing after some step of the 
REA for a witness graph $W$. Let $\beta(C)$ denote 
the number of bad tree steps that occur during in the formation of $C$, 
and let $\kappa(C)$ be the cost of $C$. Then, $\beta(C)=O(\kappa(C))$. 
\end{lemma}

\begin{proof}
Let us consider the $m'$ steps of the REA that are involved 
in the formation of $C$. Let $b_j$, for $1\le j\le m'$,  
denote that number of components (in the REA restricted to $C$) 
that have a non-empty bad part after the $j$th step. 
Clearly, if a bad tree step occurs at time $j$ then $b_j\le b_{j-1}-1$. 
On the other hand, $b_j \le b_{j-1}+1$ and 
equality can occur only if step $j$ is a costly step. 
Therefore, the number of bad tree steps is bounded 
by the number of costly steps,  which is $O(\kappa(C))$. 
\end{proof}

Next, to describe the complexity of 
components $C$, we make the following definitions. 

\begin{definition}[Maximal tree parts]
\label{D_maxTPs}
We say that a tree part is {\it maximal}
if it is inclusion maximal during in REA. 
We let $\omega=\omega(W)$ denote the 
total number of maximal tree parts. 
\end{definition}

Let us emphasize here that an axillary tree 
$T$ added to $B$ in a bad tree step is not itself 
a maximal tree part of $W$. 
Recall that such a $T$
is formed by merging $H_j^*$ with various tree parts $T_i$ 
(such that at least two of these belong to a component $C_i$ 
with a $B_i\neq\emptyset$). In this case, 
it is the various $T_i$ that are maximal tree parts, rather than $T$. 

\begin{definition}[Tree number]
The tree number $\tau(C)$
of a component $C$ is defined recursively. 
If $C$ is a TC then $\tau(C)=0$. Otherwise, 
if $C$ is formed by merging some $C_1,\ldots,C_h$
with $H_j^*$, we let $\tau(C)=\sum_{i=1}^h\tau(C_i)+\tau_j$, where 
 $\tau_j$ is the number of tree parts
in the various $C_i$  
that are added to the bad part $B$ of $C$ after the merge.
When such a tree part $T$ is
added to $B$, we say that 
it has become {\it compromised}.
\end{definition}

We note that there are two ways that 
a tree part can become compromised: 
(a) in a costly step it can interact with a vertex involved with the costly step 
(i.e., contain such a vertex outside of its bad part), or 
(b) in a bad tree step, it can be merged together with the copy of $K_r$
added in this step to form the axillary tree used in such a step.

The following lemma bounds the tree number of 
a component in terms of its cost.

\begin{lemma}[Complexity bound] 
\label{L_maxTC}
Let $C$ be a component that exists after some 
step of the REA for a witness graph $W$. Then,
\[
\tau(C)=O(\kappa(C)).
\]
In particular, $\tau=O(\kappa)$, 
where 
$\tau=\tau(W)$ is the tree number of $W$
and $\kappa=\kappa(W)$ is its total cost. 
\end{lemma}

\begin{proof}
The lemma follows directly if $C$ is a TWG, 
as then both sides are equal to $0$. 
Otherwise, we prove that $\tau(C)= O(\kappa(C)+\beta(C))$, 
which suffices by Lemma \ref{lem:few_bad_TS}.
 
Suppose that $C$ is formed by merging some $C_1,\ldots,C_h$
with $H_j^*$. Then, by the definition of the cost of $C$, 
we have that 
\[
\kappa(C)=\sum_{i=1}^{h}\kappa(C_i)+\kappa_j,
\]
where $\kappa_j=|V_j|$ is the cost of the step 
(where $\kappa_j>0$ if a costly step, and $\kappa_j=0$ otherwise).

To complete the proof, 
we show that if $\tau_j>0$,  then it is bounded by some 
constant factor multiplied by $\kappa_j$ plus the increment in $\beta(C)$. 
Observe that $\tau_j>0$ occurs only in two types of steps:

\textbf{Bad tree step.} In such a step, 
$\tau_j\le h =O(1)$ and the number 
$\beta(C)$ of bad tree steps increases by $1$.

\textbf{Costly steps.} 
In a costly step, 
for every $C_i$, each tree part of $C_i$ that 
becomes compromised has a vertex in $V_j$ that is not in the bad part 
$B_i$ of $C_i$. 
In addition, each such compromised tree part contributes 
a different vertex to $V_j$ by property (3b) of 
Lemma \ref{lem:tree_decomp_struct}. Therefore, 
the contribution of $C_i$ to $\tau_j$ is at most $|V_j|$. Since the 
number of $C_i$ satisfies $h =O(1)$, 
it follows that   $\tau_j=O(\kappa_j)$.
\end{proof}

Finally, 
we deduce the following results, 
which asserts that the combinatorial complexity of 
a general WG (as measured 
in terms of the number and total size of its maximal tree parts) 
is bounded in terms
of its cost $\kappa$.

\begin{corollary}
\label{C_omega}
If 
$W$ is a TWG then 
$\omega(W)=1$ and 
$\kappa(W)=0$. 
Otherwise, more generally, 
its number of maximal 
tree parts satisfies $\omega(W)
\le \tau(W)+1=O(\kappa)$. 
\end{corollary}

\begin{proof}
If $W$ is not a TWG then $\tau(W)$ is the number of maximal tree parts
that ever become compromised (and added to a bad
part) during the REA. Clearly, the final target edge of 
$W$ can be a target edge of at most one maximal tree part. 

We claim that every other maximal tree part $T$ must 
be compromised during the REA.  
Towards a contradiction, suppose that some such
$T$ is never compromised.  
Suppose that after the $i$th step of the REA the tree part $T$ 
is fully formed. At this time, it has at least one target edge,
none of which are the final target edge of $W$.

Consider the $j$th step, for some $j>i$, 
that one of these target edges is reused. 
Such a time exists, 
since eventually all such edges will be reused. 
Note that, in the previous 
$(j-1)$th step,
$T$ is the tree part of some component, with a bad part $B$. 
By property (3a) in 
Lemma \ref{lem:tree_decomp_struct}, 
each target edge of $T$ has at least one vertex outside of $B$. 

To complete the proof, 
we take cases with respect to the step type of the $j$th step:

{\bf Standard tree step.} 
If the $j$th step is a standard tree step, then this would 
further expand $T$, and contradict its maximality. 

{\bf Bad tree step.}  
If the $j$th step is a bad tree step, then $T$
would be part of an auxiliary $K_r$-tree
that is added to some bad part, and so 
then $T$ would, in fact, become compromised. 

{\bf IntR step.}  
In the proof of Lemma \ref{lem:tree_decomp_struct},
we showed that, in an IntR step, all of the vertices in the copy of $K_r$ 
that is completed in such a step belong to the bad part of the component
that is being augmented. 
Therefore the $j$th step cannot be an IntR step
since, as already mentioned, 
each target edge of $T$ has at least one vertex
outside of $B$. 

{\bf Costly step.} 
If the $j$th step is a costly step then, since each target edge of $T$
has at least one vertex
outside of $B$, it follows that at least one of the vertices
in the target edge of $T$ that is reused in this step 
belongs to $V_j$, and so $T$ does, in 
fact, become compromised. 
\end{proof}

\begin{corollary}
\label{C_TP_sizes}
Let $W$ be a witness graph of order $\sigma$ and cost $\kappa$, 
and let $T_1,...,T_\omega$ be its maximal tree parts. Then,
\begin{equation}
\label{eq:sum_tp_size}
\sum_{s=1}^{\omega}|V(T_i)| \le \sigma + O(\kappa)\,.
\end{equation}
\end{corollary}

\begin{proof}
For $1\le s\le \omega$, we say that a vertex $v\in V(T_s)$ 
is a base vertex of $T_s$ if there exists a time in the REA in which $T_s$ 
is a tree part of some component and $v$ belongs to the bad part of this component. 
By Lemma \ref{lem:tree_decomp_struct} and the fact that the bad part of 
a component can only grow during the REA we find that every maximal tree part $T_s$ 
has either $0$ or $2$ base vertices. Therefore, by Corollary \ref{C_omega}, 
the base vertices contribute $O(\kappa)$ to the left side of \eqref{eq:sum_tp_size}. 

Next, suppose that the same vertex $v$ appears in two maximal tree parts 
$T\neq T'$, but is not a base vertex in either of them. 
We claim that there is a step in the REA in which distinct components 
$C\supset T$ and $C'\supset T'$ are merged, and one of the following holds: 
either this step is a bad tree step and $v\in V(H_j)$, 
or this step is a costly step and $v\in V_j$. 

Before proving the claim, we observe that it implies the corollary. 
Indeed, each bad tree step $j$ can contribute at most 
$|V(H_j)|=O(1)$ to $\sum_{s}|V(T_s)|-\sigma$, 
and by Lemma \ref{lem:few_bad_TS} there are $O(\kappa)$ such steps. 
In addition, each costly step $j$ can contribute at most 
$O(\kappa_j)$, since each vertex in $V_j$ can be in $O(1)$ 
components being merged in this step. 
Therefore, the total contribution of the costly steps is bounded by $O(\kappa)$.

We proceed with proving the claim. Let $t$ (resp.\ $t'$)  be the first step in the 
REA in which $v$ belongs to a tree part $ T_1$  
(resp.\ $T_1'$) that is contained in $T$ (resp.\ $T'$). 
Note that $t\neq t'$, since they are the times of standard tree steps, 
in which distinct tree parts
are created. 
Assume that $t<t'$, and let $C_1, C_1'$ 
denote the components containing $T_1,T'_1$, resp., at time $t'$. 
Note that $T_1'$ is a tree part of $C_1'$ 
that will grow into $T'$ in the future of the REA. 

We claim that $C_1\ne C_1'$. 
Suppose, towards a contradiction, that $C_1=C_1'$. 
Recall that $v$ is not a base vertex of $T'$, and thus it is also not a base vertex of $T_1'$. 
In particular, $v$ is not in the bad part of $C_1'=C_1$. 
That is, by time $t'$, $T_1$ has grown into a tree part $T_2$ of $C_1$, 
where $T_1\subseteq T_2\subseteq T$, that has not yet been compromised. 
In consequence, we have that $T_2$ and $T_1'$ 
are distinct tree parts of the same component, 
and they share the vertex $v$ outside of its bad part, 
contradicting property (3b) of Lemma \ref{lem:tree_decomp_struct}.

Therefore, there exists a step $j$ in the REA in which 
components $C,C'$ containing $C_1,C_1'$, resp., are merged. 
If this is a costly step, then $v\in V_j$ since it belongs to at least two of the merged components. 
Otherwise, this has to be a tree step.  In such a case, $C,C'$ are disjoint outside of $H_j$, 
whence $v\in V(H_j)$. We conclude the proof by showing this step cannot be a standard tree step. 
Indeed, if it were so, the step would merge subgraphs of $T,T'$ into the same tree part, 
in contradiction to $T,T'$ being two distinct maximal tree parts. 
\end{proof}

\subsection{Proof of the coarse lower bound}

We start this section by applying 
Corollaries \ref{C_omega} and \ref{C_TP_sizes} 
to obtain a bound on the number of witness graphs of a given size and cost.
by Corollary \ref{C_omega}, a general $W$
with total cost $\kappa>0$ has only 
$\omega=O(\kappa)$ many maximal tree parts. 
By Lemma \ref{L_chikappa}, 
such a $W$ has
an excess 
$\chi=\Omega(\kappa)$ number of edges.
Hence,  we want to show that 
the number (compared with TWGs of the same size) 
of such witness graphs is offset by 
the extra cost $p^\chi$ to be paid
for any such $W$.  

Recall that $\kappa=0$ if and only if 
$W$ is a TWG. 
By Lemma \ref{lem:count_TWG}, 
the number of 
(labelled) TWGs for a given 
edge $e$ of size $\sigma=(r-2)k$
is at most $O(\gamma^{\sigma}\sigma!)$. 
For general WGs, on the other hand, 
we will prove the following upper bound. 

\begin{lemma}
\label{L_ubWs}
The number of (labeled) witness graphs
$W$ for a given edge $e$ with 
size $\sigma$ and cost $\kappa>0$
is at most 
\[
A^\sigma\cdot \sigma!\cdot \sigma^{O(\kappa)}\,
\]
for some constant $A>0$. 
\end{lemma}
This bound is far from optimal, 
but suffices for our current purposes. 
Indeed, in our application (in the proof of Proposition \ref{P_pcLBcoarse}
below) 
the size of $W$ is $\sigma=O(\log n)$ and 
its cost  
$\kappa=O(\log\log n)$
is relatively small. As we will see, in this range, 
the above bound is already
more than enough   
to prove Proposition \ref{P_pcLBcoarse}.

\begin{proof}
Observe that every (black) edge in a 
witness graph belongs to one of its maximal tree components, 
or to $E(H_j)$ for some bad tree step or costly step $j$. 
Therefore, to construct the witness graph, we choose 
(i) the number $\omega$ and sizes $t_1,...,t_\omega$ of its maximal tree parts, 
(ii) unlabeled tree components $T_1,...,T_\omega$ of the chosen sizes, and 
(iii) the number and red/black edge-colorings of bad tree steps and costly steps.
     
Then, we  consider a vertex-disjoint union $X$ of these 
unlabeled tree components and a (colored, unlabeled) copy of 
$K_r$ for each bad tree step and costly step. We conclude the 
construction by (iv) choosing an onto function from the unlabeled 
vertices of $X$ to the labels $\{1,...,\sigma+2\}$ by which we 
identify the vertices from $X$ and create the witness graph.

We turn to bound the number of witness graphs that can be constructed. 
(i) By Corollaries \ref{C_omega} and \ref{C_TP_sizes}, we have that 
$\omega = O(\kappa)$ and $t_1+\cdots +t_\omega=\sigma + O(\kappa)$, 
hence there are $\sigma^{O(\kappa)}$ choices for the number and sizes of the TCs.
(ii) There are at most $A^x$ many TCs with $x$ unlabeled vertices, for some $A>0$. 
Indeed, this follows 
directly by the tree-like structure of TCs. Therefore, there are $A^{\sigma+O(\kappa)}$ 
choices for the tree components.
(iii) By Lemma \ref{lem:few_bad_TS} and the fact that the number of 
costly steps is bounded by $\kappa$, we have that there are $2^{O(\kappa)}$ 
choices for the number and  edge-colorings of the copies of $K_r$ 
corresponding to bad tree steps and costly steps. 
(iv) Finally, by the above considerations we also have that the total 
number of vertices in the disjoint union $X$ is $\sigma + O(\kappa)$. 
Hence, the number of ways to label the vertices of $X$ 
using all the labels from the set $\{1,...,\sigma+2\}$ is 
$\sigma !\cdot \sigma^{O(\kappa)}$,  
and this concludes the proof.
\end{proof}

We will use the following technical result
to deal with very costly WGs. 

\begin{lemma}[No very costly WGs] 
\label{L_costlyWG}
Fix some $e\in E(K_n)$. 
Let  $a,b>0$ and put 
$p=(a/n)^{1/\lambda}$. 
Then, for some $c>0$, with high probability, 
there is no witness graph for $e$ in $\cG_{n,p}$
of size $\sigma \le b\log n$ and excess $\chi\ge c\log \log n$.
\end{lemma}

\begin{proof}
This follows by a simple union bound. 

In fact, we will show that, for some 
sufficiently large $c>0$, with high probability, 
there are no subgraphs of 
$\cG_{n,p}$ with $\sigma+2$ 
vertices and at least 
$\lambda \sigma+1+c\log \log n$
edges, within this range of $\sigma\le b\log n$. 
Indeed, the expected number of 
such subgraphs is
at most
\[
{n\choose \sigma+2}{\sigma^2/2\choose \sigma+1+c\log \log n}
p^{\lambda \sigma+1+c\log \log n}. 
\]
It can be seen 
(we omit the details) 
that this at most 
\[
n^{-(1+o(1))(c/\lambda-\lambda b)\log\log n},
\]
and the result follows, taking $c$ large and summing over
the $O(\log n)$ possibilities for $\sigma$. 
\end{proof}

Finally, we are ready to 
prove the coarse lower bound 
$p_c =\Omega(n^{-1/\lambda})$. 

\begin{proof}[Proof  of Proposition \ref{P_pcLBcoarse}]
Let 
$p=(\eps/n)^{1/\lambda}$, 
for some small $\eps>0$, to be determined below. 
Fix some edge $e\in E(K_n)$. We will show that 
$e\notin E(\langle\Gnp\rangle_{K_r})$ with high probability. 

First, we note that by Lemma \ref{L_lbTWG}, 
and taking $\eps>0$ sufficiently small, 
we may assume that if $e$ has a WG in 
$\Gnp$ then its cost $\kappa>0$. 
Indeed, the only WGs with $\kappa=0$ are 
TWGs by Corollary \ref{cor:TWG_only_minimizers}, 
and by Lemma \ref{L_lbTWG} there are 
with high probability no TWGs for $e$
in $\Gnp$, provided that 
$\eps<1/\gamma$.

By the Aizenman--Lebowitz property, Lemma \ref{L_AL}, 
it suffices to show that, for some $\beta>0$, with high probability 
there are no WGs 
\begin{itemize}
\item $W$ for $e$ with $\kappa(W)\ge1$
and $\sigma(W)/\log n\le \beta$; or 
\item $W'$, for any edge $f$, with 
$\sigma(W')/\log n\in [\beta,{r\choose2}\beta]$. 
\end{itemize}
Furthermore, note that, 
by Lemmas \ref{L_chikappa} and \ref{L_costlyWG}, 
we may assume that the cost $\kappa$ 
of any such potential $W$ or $W'$ is 
$O(\log\log n)$.

To rule out the first event, we apply 
Lemmas 
\ref{L_chikappa} and 
\ref{L_ubWs}. 
The expected number of such $W$
is at most 
\begin{align*}
\sum_{\sigma=0}^{\beta\log n}~
\sum_{\kappa=1}^{O(\log\log n)}
\binom n{\sigma} A^\sigma \sigma!\sigma^{O(\kappa)}p^{\lambda \sigma+1+\xi\kappa} \\ \le 
(\log n)^{O(\log\log n)}\cdot p\cdot 
\sum_{\sigma=0}^{\beta\log n} (\eps A)^\sigma\,,
\end{align*}
which tends to $0$ as $n\to\infty$ provided $\eps<1/A$.

For the second event, the expected number of such $W'$ is at most 
\begin{align*}
\binom n2\sum_{\sigma=\beta\log n}^{\beta\binom r2\log n}~
\sum_{\kappa=0}^{O(\log\log n)}
\binom n{\sigma} A^\sigma \sigma!\sigma^{O(\kappa)}p^{\lambda \sigma+1+\xi\kappa} \\ \le 
(\log n)^{O(\log\log n)}\cdot n^2p \cdot 
\sum_{\sigma\ge \beta\log n} (\eps A)^\sigma\,,
\end{align*}
Note that here we include $\kappa =0$ since $W'$ is a WG for any edge $f$, 
hence we also need to consider the possibility it is a TWG. In this case, 
we use the fact that the number of TWGs of size $\sigma$ is at most $A^\sigma\sigma!$. 
Clearly, for large enough $\beta>0$, 
this expectation tends to $0$ as $n\to\infty$, provided $\eps<1/A$.
\end{proof}

\section{Sharp lower bound}
\label{S_sharpLB}

Finally, we now turn to the lower bound 
in Theorem \ref{T_pc}.

\begin{proposition}[Sharp lower bound]
\label{P_pcLBsharp}
Fix $r\ge5$ and 
let $\eps>0$. Put  
\[
p=\left(\frac{1-\eps}{\gamma n}\right)^{1/\lambda}
\]
and fix $e\in E(K_n)$. 
Then, with high probability, 
$e\notin E(\langle\cG_{n,p}\rangle_{K_r})$. 
\end{proposition}

To prove this result, we build on the arguments of Section \ref{S_LBuptoC}. 
A closer inspection of those arguments reveals why they fall short 
of establishing a sharp lower bound: The main issue is that, in bounding 
the number of WGs of a given size and cost, we relied on the estimate 
$A^x$ for the number of TCs with $x$ unlabeled vertices (see Lemma \ref{L_ubWs}). 
However, $A>0$ is a constant {\it strictly} larger than 
$\gamma$ (the exponential growth rate of TWGs) since there are 
exponentially more TCs than TWGs. 

To refine this approach, we must control the 
TCs that actually arise in the tree decomposition of a WG of {\it small} cost, 
and show that they are {\it nearly} TWGs. In this section, we achieve this by 
a detailed analysis of the spread of IntR steps within compromised tree components 
(i.e., those that have interacted with costly steps). 
Ultimately, we will prove that the number of target edges 
(a natural proxy for complexity) that a tree part of a WG can have is bounded, 
up to a constant factor, by the total cost of the WG.

More specifically, in this section, we will analyze the spread of 
new edges that occurs after a tree part $T$ of a 
WG has interacted with other parts of the WG in costly (non-tree-like) ways. 
To be more concrete, suppose that $T$ is a $K_r$-tree
and $G$ is some other graph. We want to bound
the number of edges in 
$\langle G\cup T\rangle_{K_r}$
with both endpoints in $T$ 
in terms of the size of the vertex set $S=V(G)\cap V(T)$. 
We think of $T$ as a tree part and $S$ as the vertices of $T$ 
that are involved in costly steps in the REA. 
To bound this set of edges, we will make a comparison with a 
certain vertex bootstrap percolation process, which we call the $(r-2)_*$-BP process, 
defined in Section \ref{S_vBP} below. 
In this context, $S$ is the set of {\it seed vertices} that are initially infected. 
The {\it comparison lemma} (Lemma \ref{L_CompLem}) in 
Section \ref{S_CompLem} shows that this process provides a useful comparison: 
all endpoints in edges that are added to $T$ by the $K_r$-dynamics 
are eventually infected in $(r-2)_*$-BP. This initial comparison, 
however, is only one piece of the puzzle. In Section \ref{S_ExpLem}, 
we will prove an {\it expansion lemma} (Lemma \ref{S_ExpLem}) 
that limits the spread of
infection under the $(r-2)_*$-BP dynamics
in terms of the size of $S$. 
Finally, 
in Sections \ref{S_SprLem}--\ref{S_WGrev}, 
we combine the comparison and expansion lemmas to deduce a  
{\it spread lemma} (Lemma \ref{L_SprLem}), which is the key to 
proving a sharp lower bound on $p_c$. Using this result, 
we build upon the arguments
in Section \ref{S_LBuptoC}, 
obtaining very close control on the complexity of tree parts
in a witness graph with a given total cost.

\subsection{$(r-2)_*$-BP on TCs}
\label{S_vBP}

Recall that the graph formed  by the edges (both red and black) in a 
TC is a $K_r$-tree (see Definition \ref{D_Kr_tree}.) We have found it difficult to 
directly analyze the spread of IntR steps throughout a TC after it interacts  
with costly steps. 
Instead, we control this spread indirectly, 
using  the following modified version of the
$(r-2)$-neighbor bootstrap percolation  process on a $K_r$-tree. 

In what follows, it will be 
helpful to think of $T$ as
some tree part in a WG, 
and the set $S\subset V(T)$
as the locations where it has 
interacted with costly steps. 
As we will see, the following vertex  dynamics are chosen 
in such a way that if an edge is added to $T$ by the $K_r$-dynamics, 
then both of its endpoints are eventually infected by the vertex dynamics. 
In our applications of this process, the infected vertices will represent 
locations where we lose control over the spread of edges by the $K_r$-dynamics, 
and so we will usually assume (the worst case scenario) that they induce a clique.

\begin{definition}[$(r-2)_*$-BP]
Let $T$ be a $K_r$-tree
and $S\subset V(T)$ a set of {\it seeds}, 
which we consider to be initially {\it infected}. 
Then, in each step of the process, either 
\begin{itemize}
\item {\it (usual step)} some vertex $v$ with at least $r-2$ 
infected neighbors becomes infected; or else
\item {\it (special step)} if no such $v$ exists, but 
for some internal edge $e$, there are two cliques $H_i\neq H_j$ 
containing $e$ such that $H_i$ has $r-4$ infected vertices 
and $H_j$ has $1$ infected vertex,  
and all $r-3$ of these vertices are not in $e$, 
then  we infect an arbitrarily chosen vertex $u\in e$; and 
otherwise, 
\item {\it (terminal step)} if neither type of step is possible, 
then the process terminates. 
\end{itemize}
We call this process {\it $(r-2)_*$-BP} 
on $T$, 
and let $\langle S;T\rangle_*$ denote the set of 
eventually infected vertices in $T$ started with 
the seed set $S$. 
\end{definition}

\begin{remark}
The reason why only one vertex $u\in e$
becomes infected by the special rule is  
for technical reasons that will become 
apparent below. Note that, after such a special 
step, the other vertex in $e$ will become
infected by the usual rule. Hence 
$(r-2)_*$-BP prioritizes usual steps, 
and uses special steps only when necessary. 
\end{remark}

\begin{figure}[h]
\centering
\includegraphics[scale=1.00]{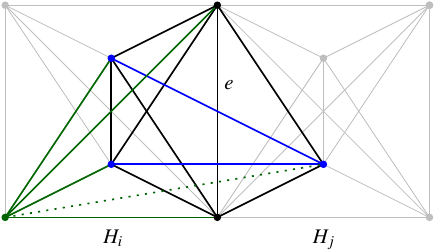}
\caption{Illustration of the special step
when $r=6$. Two copies $H_i,H_j$ of $K_6$
share an edge $e$. 
There are $r-3=3$ infected
vertices (in blue), with $r-4=2$ in $H_i$ and 1
in $H_j$. Assume that all other edges (in blue) between these vertices, 
that are not already in $T$, are eventually added by the $K_r$-dynamics. 
Then, these vertices together with those in $e$ form a clique of size $r-1=5$ 
(the relevant edges of $T$ are in black). The $K_r$-dynamics then add 
all other edges between $H_i$ and the infected vertex in $H_j$; 
the addition of one such edge (dotted) is indicated with green edges. 
Likewise, once one vertex in $e$ is infected by the special step, all other vertices
in $H_i$ (starting with the other vertex in $e$) get   infected by usual steps. 
In this way, the special step is taken so as to maintain a useful 
comparison with the $K_r$-dynamics. 
}
\label{F_star_step}
\end{figure}

As discussed in Section \ref{S_*proc},
the special step above is related to the presence of 
``bottlenecks'' as new edges
attempt to spread throughout a
tree part $T$ of a WG after it has interacted with costly steps.  
More specifically, in the special step, if the $r-3$
infected vertices in $H_i\cup H_j$ form a clique (as discussed, 
in bounding the spread of edges, we will usually assume 
that the infected vertices induce a clique, since we have little 
control over the spread of edges between such vertices) then these vertices, 
together with the endpoints of the shared edge $e$, form a clique of size $r-1$. 
If this is the case, then the $K_r$-dynamics will add further (in fact, all) edges 
between $H_i$ and the infected vertex in $H_j$. This is why we infect a 
vertex of $e$ in the special step: since edges will begin to ``flow'' from 
one side of $e$ to the other (i.e., between $H_i$
and the infected vertex in $H_j$) we must (in order to maintain a useful comparison) 
allow the vertex dynamics to spread ``through'' this edge. 
Once one of the vertices in $e$
is infected in the special step, all other vertices in $H_i$ (including the other vertex in $e$) 
will be infected by usual steps, in line with the fact that the $K_r$-dynamics will add 
all edges between $H_i$ and the infected vertex in $H_j$;  
see Figure \ref{F_star_step}. 
Furthermore, note that the configuration in the special step is {\it minimal}, 
in the sense that if there are fewer than $r-4$ infected vertices in $H_i$ 
or none in $H_j$, then no more edges can be added between 
$H_i$ and $H_j$ by the $K_r$-dynamics, in which case $e$ remains 
a ``bottleneck'' between $H_i$ and $H_j$.


\subsection{Comparison lemma}
\label{S_CompLem}

In order to study the interaction of costly 
steps with TCs, 
we will consider the following situation:  
Suppose we take the union of some graph
$G$ with a $K_r$-tree $T$ and run the $K_r$-dynamics. 
Clearly, we cannot control the effect of the $K_r$-dynamics 
inside the graph $G$, since we did not assume anything on its structure. 
However, we will show that an edge $e$ whose endpoints are in 
$V(T)$ can be infected only if both its endpoints are in the set of 
eventually infected vertices $\langle S;T\rangle_*$ in the $(r-2)_*$-BP
process on $T$,  if we take the set $S=V(G)\cap V(T)$ to be the seeds. 

In our application of the following lemma, $G$ will correspond to the 
graph that is comprised of the costly steps during the REA, 
which interact with a TC $T$. 
Roughly speaking, we will use it to show that IntR steps
can only spread as far as $(r-2)_*$-BP.

\begin{lemma}[Comparison lemma]
\label{L_CompLem}
Let $r\ge 5$, $T$ be a $K_r$-tree, $G$ a graph and $S=V(G)\cap V(T)$.
Then,
\[
\langle G\cup T\rangle_{K_r} \subseteq Q\cup T\,,
\]
where $Q$ is a clique on the vertex set $V(G)\cup \langle S;T\rangle_*$.
\end{lemma}

We note that it is easy to see that  \(
\langle G\cup T\rangle_{K_r} = Q\cup T\,
\)  if the graph $G$ is $K_r$-percolating.

\begin{proof}
We prove the lemma using Claim \ref{clm:Kr_tree_structure} 
on the structure of $K_r$-trees.
Denote $G_*=\langle G\cup T\rangle_{K_r}$ and $I_*=\langle S;T\rangle_*$. 
Suppose, 
towards a contradiction, that 
$\tilde e=\tilde u\tilde v\in E(G_*) \setminus E(Q\cup T)$ is the first 
such additional edge,  
with at least one vertex $\tilde v\in V(T)\setminus I_*$, that 
is added to $G_*$ by the $K_r$-dynamics. 
Let $\tilde H$ be the copy of $K_r$ that $\tilde e$ completes, 
such that $\tilde H\setminus \tilde e\subset G_*$. Denote 
$V(\tilde H)=\{\tilde u,\tilde v,x_1,...,x_{r-2}\}$. 
Note that since $\tilde e$ is the first additional edge added to $G_*$ 
with an endpoint in 
$V(T)\setminus I_*$, and in addition, the vertex 
$\tilde v\notin I_*$, then the edges in $\tilde H$ containing $\tilde v$ are all in $T$. 
In particular, $V(\tilde H)\subset V(T)$, since $\tilde H$ is a clique.   

First, note that if also $\tilde u\notin I_*$ 
then the $2(r-2)$ edges between $x_1,...,x_{r-2}$ and 
$\tilde u,\tilde v$ belong to $T$. Therefore, $\tilde u$ and $\tilde v$ have 
$r-2\ge 3$ common neighbors in $T$, 
and so, by Claim \ref{clm:Kr_tree_structure}(2), 
they are neighbors in $T$, in contradiction to $\tilde e\notin T$.

Hence, we assume that $\tilde u \in V_*$. Suppose that, 
up to relabeling the $x_i$ if necessary, that  
\[
\{x_1,...,x_{r-2}\}\setminus I_*=\{x_1,...,x_k\}, 
\] 
for some $0\le k\le r-2$. 
We note that all of the $r-2$ edges between $\tilde v$ and $x_1,...,x_{r-2}$
and all of the $k(r-1-k)$ edges between $x_1,...,x_k$ and $x_{k+1},...,x_{r-2},\tilde u$ 
are edges of $T$, since 
each such edge has 
at least one endpoint in $V(T)\setminus I_*$.

We consider the following cases:

{\bf Case 1.} If $k\ge 3$ then $\tilde u$ and $\tilde v$ 
have at least three common neighbors in $T$, 
and by Claim \ref{clm:Kr_tree_structure}(2) they are neighbors in $T$, 
in contradiction to $\tilde e\notin T$.

{\bf Case 2.} If $k = 0$ then $\tilde v$ has $r-2$ neighbors in 
$T$ that belong to $I_*$. This means that $\tilde v$ 
should have been infected in the $(r-2)_*$-BP process, 
and it contradicts the assumption that $\tilde v\notin I_*$. 

{\bf Case 3.} If $k = 1$ then $x_1$ has $r-2$ neighbors in 
$T$ that belong to $I_*$: $x_2,...,x_{r-2}$ and $\tilde u$. 
Similarly to Case 2, this contradicts 
$x_1\notin I_*$.

{\bf Case 4.} Finally, suppose that $k = 2$. 

First, we claim that $\tilde v,x_1,...,x_{r-2}$ induce a clique in $T$. 
Note that $\tilde v,x_1,x_2$ form a clique in $T$, and that 
each of these vertices are neighbors with all of 
$x_3,...,x_{r-2}$. Moreover, 
by Claim \ref{clm:Kr_tree_structure}(2), 
we see that $x_ix_j\in E(T)$, for all $3\le i<j\le r-2$, since 
each such 
$x_i,x_j$ have the three common neighbors 
$\tilde v,x_1,x_2$
in $T$. Therefore, 
by Claim \ref{clm:Kr_tree_structure}(1), 
we find that 
$\tilde v,x_1,...,x_{r-2}$ belong to some copy $H_i$ of $K_r$ in $T$.

Second, observe that $\tilde u,x_1,x_2$ also form a clique in $T$. 
Hence, by Claim \ref{clm:Kr_tree_structure}(1),  
they belong to some copy $H_j$ of $K_r$ in $T$. 
Note that $H_i\ne H_j$, since $\tilde e\notin T$, and 
that 
$H_i,H_j$
intersect in the edge $x_1x_2$. 

Finally, observe that $x_3,...,x_{r-2}\in I_*$ are $r-4$ infected vertices in $H_i$, 
and $\tilde u\in I_*$ is $1$ infected vertex in $H_j$. 
Therefore, either $x_1$ or $x_2$ would have been infected in $(r-2)_*$-BP 
in a special step (or in a usual step if one of them has more neighbors in $I_*$), 
in contradiction to $x_1,x_2\notin I_*$.
\end{proof}

Lemma \ref{L_CompLem} asserts that all the edges that are added to 
$G\cup T$ by the $K_r$-dynamics are in 
the clique $Q$ of the vertices in $V(G)\cup\langle S;T\rangle_*$. 
However, in our application of this result, we actually need a slightly stronger statement,  
namely, that all the copies $H$ of $K_r$ that are completed during this run of the 
$K_r$-dynamics are fully contained (not only the edges 
that these $H$ add to $G\cup T$) in $Q$. 

\begin{corollary}\label{cor:comp_cor}
Let $r\ge 5$, $T$ be a $K_r$-tree, $G$ a graph and $S=V(G)\cap V(T)$. 
Then, every copy $\tilde H$ of $K_r$ that is completed by the 
$K_r$-dynamics started from $G\cup T$ satisfies
\(
V(\tilde H) \subset V(G)\cup\langle S;T \rangle_*\,.
\)
\end{corollary}

In other words, each vertex
of such a $\tilde H$
is either in $G$ or 
infected by the $(r-2)_*$-BP 
process on $T$
started from $S$. 

\begin{proof}
Denote $G_*=\langle G\cup T\rangle_{K_r}$ and $I_*=\langle S;T\rangle_*$. 
Suppose, in contradiction, that $\tilde H$ is a copy of $K_r$ that is 
completed in the run of the $K_r$-dynamics on $G\cup T$ and 
$V(\tilde H)\cap V(T)\setminus I_*\neq\emptyset$. 
Let $\tilde v$ be a vertex in $V(\tilde H)\cap V(T)\setminus I_*$ and 
$\tilde e \in G_*\setminus (G\cup T)$ 
the edge that is added to complete $\tilde H$. 

By Lemma \ref{L_CompLem},  the endpoints of every edge in 
$G_*$ that is not in $T$ are in $V(G)\cup I_*$. 
Therefore, all the edges of $\tilde H$ that contain $\tilde v$ belong to $T$, 
and, in particular, $V(\tilde H)\subset V(T)$. 
Consequently, the endpoints of $\tilde e$ are in $I_*$, 
whence $\tilde v\notin\tilde e$. 

We consider two cases: 

First, if $1\le |V(\tilde H)\setminus I_*|\le 2$ then 
$\tilde v$ has at least $r-2$ neighbors in $T$ from $I_*$.  
This means that it should have been infected in the $(r-2)_*$-BP process, 
which would contradict the assumption that $\tilde v \notin I_*$. 

Otherwise, suppose $|V(\tilde H)\setminus I_*|\ge 3$. 
The edges between $V(\tilde H)\setminus I_*$ 
and the endpoints of $\tilde e$ belong to $T$. 
Therefore, the endpoints of $\tilde e$ have 
at least $3$ common neighbors in $T$. 
But then  $\tilde e\in T$, by Claim \ref{clm:Kr_tree_structure}(2), 
in contradiction to the assumption that $\tilde e$ 
is added by the $K_r$-dynamics on $T\cup G$. 
\end{proof}

\subsection{Expansion lemma}
\label{S_ExpLem}

As discussed above, 
we plan to use the comparison lemma 
(Lemma \ref{L_CompLem})
to bound the spread of IntR steps
after a tree part of a witness graph has
interacted with costly steps. 
In doing so, we will need to  
gain some control on 
the size of $\langle S;T\rangle_*$, 
as a function of $|S|$.

\begin{lemma}[Expansion lemma]
\label{L_ExpLem}
Let $T$ be a $K_r$-tree for some $r\ge 5$ and $S\subset V(T)$. 
Then 
\[
|\langle S;G\rangle_*|
=O(|S|). 
\]
\end{lemma}

In other words the 
$(r-2)_*$-BP process can expand at most linearly. 
In fact, in our application of this lemma,
sub-exponential expansion would 
suffice, but we do not know a simpler 
proof of such a weaker result.

Before proceeding with the proof, which is somewhat subtle, 
let us give some intuition for why such a result 
should be true. 
The edge density of a  $K_r$-tree $T$
is approximately 
\[
\lambda'=
\frac{r+1}{2}
\approx 
\frac{1+\vartheta({r\choose2}-1)}{2+\vartheta(r-2)}
\] 
if its order $\vartheta$ is large. 
If the infection spreads 
via $(r-2)_*$-BP
mostly locally inside its copies of $K_r$, 
then the edge density of 
the graph $T_*\subseteq T$ induced by 
$\langle S;T\rangle_*$ will be lower (since $K_r$ is strictly balanced). 
If, however, the infection spreads more globally,  
and the edge density of $T_*$ is closer to 
$\lambda'$, 
then many full copies of $H$ will become infected. 
Intuitively, in this case, we expect that many times during the 
$(r-2)_*$-BP process, vertices will become infected by  
more than the minimal number $r-2$ of infected neighbors (note that, 
if a copy of $K_r$ becomes fully infected, then its  
last vertex to become infected has at least $r-1$ infected neighbors
at the time of its infection). 

As in the work of Riedl \cite{Rie12}, a key step in the following proof
is to keep track of the number of infected neighbors
that vertices have when they become infected. 
The added difficultly, in the current proof, is the presence
of special steps, in which vertices can become infected with 
less than the usual threshold number of infection. 
A careful analysis is required to show that 
all such special steps are compensated for 
later on during the process.

\begin{proof}
Let $S$ be a set of seeds for the $(r-2)_*$-BP process
on a $K_r$-tree $T$. 
Let $H_1,\ldots,H_\vartheta$  be the copies of $K_r$
used in the construction of $T$ (as in Definition \ref{D_Kr_tree}). 
We also let 
\begin{itemize}
\item $I_*=\langle S;T\rangle_*$ be the set
of all eventually infected vertices; 
\item $I=I_*\setminus S$ be the set of all 
infected non-seed vertices; and 
\item $E$ be the set of all edges with either both endpoints in $I$
or one in $S$ and the other in $I$. 
\end{itemize}

Note that each vertex in $I$ is incident to at least $r-3$
edges in $E$. 

The proof follows three main parts.

{\bf Part 1.} First, we relate $|E|$ 
and $| I |$
by 
carefully assigning edges to vertices at the time
of their infection. 

In each step of the infection process, 
when a vertex $v$ is infected, 
there are at least $r-3$
edges in $E$ between $v$ 
and previously infected vertices. 
We call a step a {\it $k$-step} if there are exactly $k$ such edges. 
Furthermore, we let 
$t_k$ denote the total number of $k$-steps
during the process.  

\begin{claim}\label{clm:partIbnd} We have that 
\[
|E|
\ge \frac{r+1}2|I| 
+\frac 18\sum_{k\ge r-1}kt_k.
\]
\end{claim}

Since the inequality may not be intuitively clear, let us give some 
explanation before the proof. 
Recall that $(r+1)/2$ is the (asymptotic) edge density of a $K_r$-tree $T$.
Consider the subgraph $G_*$ induced by all eventually infected vertices
in $I_*$. 
Intuitively, in order for the infection to spread, 
the seed set $S$
will need to cover some of the less dense 
regions of $T_*$. 
As a result, the ratio of the number of edges in 
$E$ by the number of eventually infected non-seeds in $I$ will pick up a 
``boost'' (the sum over $k \ge r-1$ in the claim), 
corresponding to times when vertices 
have strictly more than $r-2$ (at least $r-1$) 
infected neighbors (associated with denser regions of $T_*$).

\begin{proof}
Clearly,
\[
|I|=\sum_{k\ge r-3} t_k
\]
and 
\[
|E|=\sum_{k\ge r-3} k t_k.
\]
Therefore, 
\begin{align}
\nonumber |E|-\frac{r+1}2|I| &= \sum_{k\ge r-3}\left(k-\frac{r+1}2\right)t_k\\
&\ge \frac{r-7}2\cdot t_{r-3}+\sum_{k\ge r-1}\left(k-\frac{r+1}2\right)t_k\,.\label{eq:t_sum_bound}
\end{align}
In this inequality we have used $r\ge 5$ to deduce that the 
(discarded) 
summand 
$(r-5)/2\cdot t_{r-2}$, corresponding to $k=r-2$, is non-negative.

If $r\ge 7$ then the claim 
follows since the first term in \eqref{eq:t_sum_bound} 
is non-negative and, in addition, 
$k-(r+1)/2\ge k/8$ for every $k\ge r-1$ and $r\ge 7$.

For $r=5,6$ we also need to bound $t_{r-3}$ from above. 
Note that $(r-3)$-steps are special steps, in which
one of the two uninfected vertices in some internal edge 
$e$ 
becomes infected. This occurs due to the fact that there are exactly
two $H_i\neq H_j$, both containing $e$, where $H_i$ has exactly 
$r-4$ infected vertices and $H_j$ has exactly one infected vertex 
(not in $e$, whose vertices are uninfected). 
Furthermore, recall that special steps are only taken
when usual steps are not possible. 
Therefore, after this vertex in $e$ becomes infected in a special step, 
all of the other vertices in $H_i$ will become infected
by usual steps. The step that infects the last vertex to become infected in  
$H_i$ will be a $k$-step, for some $k\ge r-1$. Altogether, we find that  
\[
t_{r-3}\le \sum_{k\ge r-1} t_k. 
\]
Consider the case $r=6$. Combining this inequality 
with \eqref{eq:t_sum_bound} yields
\[
|E|-\frac 72|I| \ge \sum_{k\ge 5}(k-4)t_k\,,
\]
and the claim follows since $k-4\ge k/8$ for every $k\ge 5$.

In the case $r=5$ we need a slightly better 
upper bound for $t_2$. 
In this case, since $r-4=1$, $H_i$ and $H_j$ 
play the same role and will both become 
entirely infected after the special step. 
Therefore, the two steps that 
infect the last vertex in each of them 
are a $k$-step for some $k\ge 4$. 
We deduce that 
\[
2t_2\le \sum_{k\ge4} t_k.     
\]
Combining this inequality 
with \eqref{eq:t_sum_bound} gives
(in the case $r=5$) that 
\[
|E|-3|I| \ge \sum_{k\ge 4}(k-7/2)t_k\,,
\]
and the claim follows since $k-7/2\ge k/8$ for every $k\ge 4$.
\end{proof}

{\bf Part 2.} Next, we obtain some further inequalities
by considering an exploration process of $T$, one $H_i$ at a time.

To begin, we reveal one its edges. Then, in each 
subsequent step, we reveal an unexplored $H_i$ 
which has exactly one edge in the subgraph of $T$ 
revealed so far. 
Since $T$ is a $K_r$-tree, all of $T$ will eventually
be revealed by this procedure. 
Also note that 
exactly $r-2$ vertices and $\binom r2-1$ edges in $H_i$ are new, 
as the other $2$ vertices are in a previously revealed edge. 

Let $n_j$, for $0\le j\le r-3$, be the number of steps in the exploration process
when exactly $j$ of the new vertices in the revealed $H_i$ are in $I_*$. 
Quite crucially, we define $n_{r-2}$ slightly differently, as
the number of steps in which all $r-2$ new vertices 
(and hence all $r$ vertices in $H_i$) are infected, but that not all 
vertices in $H_i$ are seeds. 

The lower bound
\begin{equation}\label{E_ILB}
\sum_{j=1}^{r-2}jn_j
\le
|I_*|
\end{equation}
is clear. To obtain an upper bound, we note that 
the only other infected vertices that are 
not counted by this sum are either in the initial 
edge or in $S$. Hence
\begin{equation}\label{E_IUB}
|I_*|
\le 
\sum_{j=1}^{r-2}j n_j +(2+|S|).
\end{equation}

Next, we claim that 
\begin{equation}\label{E_EUB}
|E|
\le 
1+\sum_{j=1}^{r-2}\frac{j(j+3)}2n_j
\end{equation}
Indeed, the $1$ before the sum accounts for the possibility 
that the initial edge is in $E$. In addition, 
if $j$ new vertices from $I_*$ are exposed in a step then at most 
$\binom j2+2j=j(j+3)/2$ new edges from $E$ are exposed in that step.

Finally, we claim that 
\begin{equation}\label{E_hatLB}
(r-2)n_{r-2}\le \sum_{k\ge r-1} kt_k.
\end{equation}
To see this, consider a step in which
all $r-2$ of the new vertices in $H_i$ are infected, 
but that not all vertices in $H_i$ are seeds. 
Then, when the last vertex $v$ in $H_i$ is
infected it has at least $r-1$ infected neighbors. 
If $v$ is amongst the new 
vertices in this step then we obtain $r-1$ 
new edges in $E$; and otherwise, 
only $r-2$ such edges. 
The result follows, noting that 
each step in the $(r-2)_*$-BP process
that such a vertex $v$ becomes infected is a $k$-step,
for some $k\ge r-1$. 
All of the above described edges, of which there are at least $(r-2)n_{r-2},$
contribute to the infection of these vertices. 

{\bf Part 3.} To conclude, we combine the above inequalities. 

First, we multiply \eqref{E_ILB} by $(r+1)/2$, 
add it to \eqref{E_EUB}, and rearrange, so as  
to obtain
\[
|E|
\le
1+\frac{r+1}2|I_*|-\sum_{j=1}^{r-3}\frac {j(r-2-j)}2n_j\,.
\]
Together with Claim \ref{clm:partIbnd} and \eqref{E_hatLB}, this implies that 
\[
\frac{r+1}2|I|+\frac{r-2}{8}n_{r-2}
\le 
1+\frac{r+1}2|I_*|-\sum_{j=1}^{r-3}\frac {j(r-2-j)}2n_j, 
\]
and so 
\[
\sum_{j=1}^{r-3}\frac {j(r-2-j)}2n_j
+\frac{r-2}{8} n_{r-2}
\le
1+\frac{r+1}2|S|.
\]
Multiplying this inequality by 8, and then adding 
$2+|S|$ to both sides, yields
\[
\sum_{j=1}^{r-3}4j(r-2-j)n_j
+(r-2)n_{r-2} +(2 +|S|)
\le
10+(4r+5)|S|.
\]
Finally, using \eqref{E_IUB} and the fact that $j\le 4j(r-2-j)$, 
for every $1\le j\le r-3$,
we find that 
\[
|I_*|
\le 
10+(4r+5)|S|
=O(|S|),
\]
as required. 
\end{proof}

\subsection{Spread lemma}
\label{S_SprLem}

Combining 
Corollary \ref{cor:comp_cor} and Lemma \ref{L_ExpLem}, 
we obtain the following result. 

\begin{lemma}[Spread lemma]
\label{L_SprLem}
Fix $r\ge5$. Let $T$ be a $K_r$-tree 
and  $G$ be a graph with $x=|V(T)\cap V(G)|$
many vertices in common with $T$. 
Let $E_*$ be the set of edges  in $T$
that are contained in a copy of $K_r$  completed by the $K_r$-dynamics
started from $G\cup T$. 
Then
we have that $|E_*|=O(x)$. 
\end{lemma}

In other words, if we start with $T$ and add $G$, 
then at most $O(x)$ edges 
in $T$
are ever used by the 
$K_r$-dynamics. 

\begin{proof}
Denote $S=V(T)\cap V(G)$. By Corollary \ref{cor:comp_cor} the endpoints 
of the edges in $E_*$ belong to $I_*:=\langle S;T\rangle_*$. 
In addition, by Lemma \ref{L_ExpLem}, we have that $|I_*|=O(x)$. 
To conclude the proof, we observe that there are at most $O(|I_*|)$ 
edges in the subgraph of $T$ induced by $I_*$. 
Indeed, when we expose the copies of $K_r$ used in the construction of $T$, 
in every step there are $O(1)$ 
vertices and edges from that induced subgraph which are introduced.
\end{proof}

\subsection{The structure of maximal tree parts}
\label{S_TWGdecomp}

Next, using Lemma \ref{L_SprLem}, 
we obtain an upper bound on the 
complexity of the maximal 
tree parts of a witness graph. 
This observation will play a key role
in our improvement on the 
combinatorial bound in Lemma
\ref{L_ubWs}, given in the next section. 

\begin{lemma}
\label{L_targets_in_maxT}
Let $W$ be a WG of cost $\kappa>0$, and $T$ be one of its maximal 
tree parts, as given by the tree decomposition of $W$
(see Defintion \ref{D_maxTPs}). 
Then, $T$ has at most $O(\kappa)$ target edges. 
\end{lemma}
\begin{proof}
Denote the size of $W$ by $\sigma$. Recall that all the target edges of $T$, 
except perhaps one, must be used in the REA.
Let $G$ be the (typically non-disjoint) union of: 
(i) the other maximal tree parts $T_1,...,T_{\omega-1}$ of $W$, and 
(ii) all the copies of $K_r$ that are added in bad tree steps or costly steps. 

Here we view $G$ and $T$ as edge-colored graphs 
that contain both black and red edges. 
To be precise, 
a tree part $T_i$ contains all the edges in the tree steps that were used to create it, 
and they are assigned the color that was given to them by the REA. 
Similarly, each copy of $K_r$ corresponding to a bad tree steps 
or costly steps contains all $\binom r2$ edges 
(some red and some black) according to the way the edges were 
colored in the corresponding step. 
Importantly, note that while all the black edges of $W$ must belong to 
$G\cup T$, some of the red edges of $W$ may be absent from $G\cup T$. 
This is the case for the red edges that are added in IntR steps, 
but are not used in 
subsequent 
tree steps or costly steps. However, all the red edges of $W$ 
clearly belong to $\langle G\cup T\rangle_{K_r}$. Hence, 
every (red) target edge of $T$ that gets used in the REA is either an 
edge of $G$ or it belongs to a copy of $K_r$ that is completed by the 
$K_r$-dynamics started from $G \cup T$. 

Therefore,  using Lemma \ref{L_SprLem}, we complete the proof by 
showing that the intersection of $T$ and $G$ has $O(\kappa)$ vertices 
(and can therefore span only $O(\kappa)$ edges in the $K_r$-tree $T$). 
First, by Lemma \ref{lem:few_bad_TS} and the definition of the cost of a step, 
there are only $O(\kappa)$ vertices in the copies of $K_r$ that are 
added in bad tree steps or costly steps. This accounts for the vertices in 
$T\cap G$ that participate in such steps. In addition, we derive
\[
|V(T)|+\sum_{i=1}^{\omega-1}|V(T_i)\setminus V(T) | \ge \sigma  - O(\kappa)\,,
\]
since every vertex that appears in some maximal tree part is accounted 
for at least once in the left side, and there are at most $O(\kappa)$ vertices 
in $W$ that do not appear in any maximal tree part. 
By combining this inequality with Claim \ref{C_TP_sizes}, we derive
\[
\sum_{i=1}^{\omega-1} |V(T_i)\cap V(T)| \le O(\kappa)\,,
\]
and the proof is concluded.
\end{proof}

\begin{claim}\label{clm:num_TC_with_few_targes}
The number of unlabeled tree components with 
$x$ vertices and $\Delta$ target edges is at most $\gamma^x\cdot x^{O(\Delta)}$.
\end{claim}
\begin{proof}
Consider the target edges $e_1,...,e_\Delta$
of a tree component $T$, listed in some arbitrary order. 
For each target edge $e_i$, let $T_i\subseteq T$ be a TWG for $e_i$,
which can be found
using the WGA. After $T_i$ is found, we color all its edges in 
$T$ black, so that all future TWGs $T_j$, 
with $j>i$, will share at most one edge with it. 
Consider the auxiliary graph, 
whose vertices are
$T_1,...,T_\Delta$,
and 
such that two are adjacent if they share an edge. 
By induction, this auxiliary is a tree, 
and so 
\[
\sum_{i=1}^{\Delta} |V(T_i)| = x +2(\Delta-1)\,.
\]
Hence, to choose such a witness graphs there are 
$x^{O(\Delta)}$ choices for the sizes of the TWGs, 
at most $\gamma^{x+2(\Delta-1)}$ choices 
for the tree structure of the TWGs, by 
\eqref{E_tkasy} above, and $x^{O(\Delta)}$ ways to 
glue together the TWGs into  $T$. This concludes the proof. 
\end{proof}

\subsection{The WG bound, revisited}
\label{S_WGrev}

With Lemma \ref{L_targets_in_maxT}
in hand, we can significantly improve 
upon 
Lemma \ref{L_ubWs}.

\begin{lemma}
\label{L_ubWs2}
The number of (labelled) witness graphs
$W$ for a given edge $e$ with 
size $\sigma$ and cost $\kappa>0$
is at most 
\[
\gamma^{\sigma} \cdot \sigma! \cdot \sigma^{O(\kappa^2)}\,.
\]
\end{lemma}

\begin{proof}
Recall that, in the proof of Lemma \ref{L_ubWs}, 
we used the simple bound $A^x$
for the number of 
TCs $T$ with $x$ unlabelled vertices.
However, by Lemma \ref{L_targets_in_maxT}, we now know that 
each of the $O(\kappa)$ many maximal tree parts $T$ of 
$W$ have $O(\kappa)$ many targets, and by 
Claim \ref{clm:num_TC_with_few_targes} there are at most $\gamma^xx^{O(\kappa)}$ such TCs. 

We adapt the proof of Lemma \ref{L_ubWs} such that in item 
(ii) we choose for each $T_i$ one of the (at most) 
$\gamma^{t_i}\cdot \sigma ^{O(\kappa)}$ tree components with 
$t_i$ vertices and $O(\kappa)$ target edges. 
Therefore, the total number of choices for item (ii) is 
$$
\prod_{i=1}^{\omega}\gamma^{t_i}\sigma^{O(\kappa)}=
\gamma^{\sum_{i=1}^{\omega}t_i}\sigma^{O(\omega\kappa)}
=
\gamma^{\sigma+O(\kappa)}\sigma^{O(\kappa^2)}\,,$$
where the last equality is by Corollaries \ref{C_omega} and \ref{C_TP_sizes}.
The rest of the argument is not changed and the result follows directly.
\end{proof}

\begin{proof}[Proof of Proposition \ref{P_pcLBsharp}]

We adapt the proof of 
Proposition \ref{P_pcLBcoarse} 
above. 
Once again, we fix some $e\in E(K_n)$
and show that 
$e\notin E(\langle\Gnp\rangle_{K_r})$
with high probability. 

Since
\[
p=\left(\frac{1-\eps}{\gamma n}\right)^{1/\lambda}
\]
we have by Lemma \ref{L_lbTWG}
that with high probability 
$e$ is not added by a TWG. 

Hence, once again, by Lemma \ref{L_AL}, 
it remains to rule out the existence of 
WGs
\begin{itemize}
\item $W$ for $e$ with $\kappa(W)\ge1$
and $\sigma(W)/\log n\le \beta\log n$; or 
\item $W'$, for any edge $f$, with 
$\sigma(W')/\log n\in [\beta,{r\choose2}\beta]$,
\end{itemize}
where $\beta=\beta(\eps)>0$
will be determined below. 
As before, we can assume that 
any such $W$ and $W'$
has cost $O(\log\log n)$. 

For the first event, we apply Lemmas
\ref{L_chikappa} and 
\ref{L_ubWs2}. 
The expected number of such $W$
is at most 
\begin{align*}
\sum_{\sigma=0}^{\beta\log n}~
\sum_{\kappa=1}^{O(\log\log n)}
\binom n{\sigma} \gamma^\sigma \sigma!\sigma^{O(\kappa^2)}p^{\lambda \sigma+1+\xi\kappa} \\ \le 
(\log n)^{O((\log\log n)^2)}\cdot p\cdot 
\sum_{\sigma=0}^{\beta\log n} (1-\eps)^\sigma\,,
\end{align*}
which tends to $0$ as $n\to\infty$.

For the second event, the expected number of such $W'$ is at most 
\begin{align*}
\binom n2\sum_{\sigma=\beta\log n}^{\beta\binom r2\log n}~
\sum_{\kappa=0}^{O(\log\log n)}
\binom n{\sigma} \gamma^\sigma \sigma!\sigma^{O(\kappa^2)}p^{\lambda \sigma+1+\xi\kappa} \\ \le 
(\log n)^{O((\log\log n)^2)}\cdot n^2p \cdot 
\sum_{\sigma\ge \beta\log n} (1-\eps)^\sigma\,.
\end{align*}
which tends to $0$ as $n\to\infty$, provided $\beta>0$ is large enough.
\end{proof}

\section{Subcritical edge expansion}
\label{S_subexp}

In this section we prove Theorem \ref{T_sub_density}, 
which
shows that,
in the subcritical regime, the $K_r$-dynamics increases 
the edge density of $\Gnp$ by a factor that 
converges to a deterministic {\em constant} $\rho=\rho(p)$.
To accomplish this, we show that in the subcritical regime 
(i) almost every infected edge is infected 
by a unique ``small'' TWG, and 
(ii) the number of TWGs in $\Gnp$ is, 
with high probability, close to $\rho \cdot p \binom{n}{2}$.

To compute the constant $\rho$ we consider the
generating function of the $d$th Fuss--Catalan numbers,
given by 
\[
f(x)=\sum_{k=0}^\infty\mathrm{FC}_d(k)x^k\,.
\]
Recall that we denote  $\alpha_d = (d+1)^{d+1}/d^d$.
By \eqref{E_FCasy}, 
$f$ has radius of convergence $1/\alpha_d$, converging 
if and only if $|x|\le 1/\alpha_d$. 
For such $x$, the recursion \eqref{E_FCrec}  implies that
\begin{equation}\label{E_FC_gen}
xf(x)^{d+1}=f(x)-1\,.    
\end{equation}
It can be shown, for any fixed $0<x<1/\alpha_d$, 
that \eqref{E_FC_gen} has precisely two solutions, 
which are both greater than $1$. 
Using the continuity of $f$ and $f(0)=1$, 
one can deduce that $f(x)$ is the smaller of these two solutions. 

In addition, $f(1/\alpha_d)=(d+1)/d$, since this is the unique solution for \eqref{E_FC_gen} when $x=1/\alpha_d$. Hence, we have that $1< f(x) < (d+1)/d$ for every $0<x<1/\alpha_d$ by monotonicity of $f$.

\begin{proof}[Proof of Theorem \ref{T_sub_density}]
Let
\[
\rho 
= \sum_{k=0}^{\infty}
\mathrm{FC}_{\binom r2-2}(k)(1/\bar\alpha)^{k}.
\]
By the discussion above, we find that $\rho$ is the 
smallest root of the equation $\rho^{\binom r2-1}=\bar\alpha(\rho-1)$ 
and that it satisfies $1 < \rho < \frac{\binom r2-1}{\binom r2-2}$.

Let $\eps>0$, and choose a sufficiently 
large integer $K$ so that
\begin{equation}\label{E_rhoK}
\rho_K:=\sum_{k\le K}
\mathrm{FC}_{\binom r2-2}(k)(1/\bar\alpha)^{k}
\in (\rho -\varepsilon,\rho).
\end{equation}

We say that a TWG is {\it large} if its order is greater than $K$, 
and otherwise it is called {\it small}.

Denote by $X_k$ the number of TWGs of order $k$ 
that are contained in $\Gnp$. Then,
\begin{equation}\nonumber
\E(X_k) = \binom n2 \binom {n-2}{(r-2)k}t(k)p^{\lambda(r-2)k+1}\,.    
\end{equation}
If $k$ is fixed then 
\[
((r-2)k)!\binom{n-2}{(r-2)k}\sim n^{(r-2)k}. 
\]
Hence, using Lemma \ref{lem:count_TWG}, 
it follows that 
\begin{equation}\label{E_Exp_of_Xk}
\E(X_k)
\sim 
p\binom n2 \cdot
\mathrm{FC}_{\binom r2-2}(k)(1/\bar\alpha)^{k}. 
\end{equation}

Let us turn to the variance of $X_k$. 
We claim that, for some constant $D>0$, 
we have that 
\begin{align*}
\E(X_k^2)
&\le\E(X_k)+\E(X_k)^2+\sum_{v,e\ge 1}\binom n2^2t(k)^2n^{2(r-2)k-v}p^{2\lambda(r-2)k+2-e}\\
&\le \E(X_k) +\E(X_k)^2+D\left(p\binom n2\right)^2n^{-v+e/\lambda}\,.
\end{align*}
In the first line, the summation is over all the integers $v,e\ge1$ 
for which there exist $k$-TWGs $T_1,T_2$ that share $e$ edges and $v$ vertices. 
The first inequality holds since the first term accounts for pairs of identical TWGs, 
the second for pairs of edge-disjoint TWGs, and the third for pairs of TWGs 
that intersect non-trivailly. For the second inequality, we use the fact that 
$k,\bar\gamma$ are fixed and independent of $n$, hence $D>0$ is some constant.

Let $T_1,T_2$ be distinct $k$-TWGs for the edges $e_1,e_2$ respectively. 
Suppose that they share $e\ge 1$ edges and a set $V$ of $v$ vertices. 
We claim that $e < \lambda v$. Indeed, 
by Lemma \ref{L_strTWB}(1) we have
$e \le \lambda |V \setminus e_1| \le \lambda v$. 
Moreover, both inequalities cannot hold with equality, 
since if $V$ is disjoint from $e_1$, then the intersection of $T_1$ and $T_2$  
cannot be obtained from $T_1$ by removing one of its branches.  
Therefore, $\E(X_k^2)\sim \E(X_k)^2$ and so, by Chebyshev's inequality, 
$X_k\sim\E(X_k)$ in probability as $n\to\infty$. 
Since $K>0$ is a constant, we find that 
\begin{equation}\label{E_smallTWG}
\sum_{k\le K}X_k\sim \rho_K\cdot p\binom n2\,,
\end{equation}
in probability as $n\to\infty$.

Next, 
to bound the contribution from large TWGs, 
we observe that the left side of \eqref{E_Exp_of_Xk} is 
bounded from above by the right side for every $k$ (even if it grows with $n$). 
Therefore, the expected number of large TWGs in $\Gnp$ is 
\[
\E\left(\sum_{k>K}X_k\right) 
< p\binom n2\varepsilon\,.
\]
Therefore,
\begin{equation}\label{E_largeTWGs}
\P\left(\sum_{k>K}X_k>p\binom n2\sqrt\varepsilon\right) <\sqrt\varepsilon\,.
\end{equation}

Let $Y$ denote the number of pairs of small TWGs 
contained in $\Gnp$
with the same target edge $f\in (K_n)$. Then,
\[
\E(Y) \le \binom n2 \sum_{k_1,k_2\le K}\sum_{v,e}t(k_1)t(k_2) n^{(r-2)(k_1+k_2)-v}p^{\lambda(r-2)(k_1+k_2)+2-e}\,,
\]
where the second sum is over all the integers $v,e$ for which 
there is a $k_1$-TWG and a $k_2$-TWG for the same edge $f$ that have 
$e$ edges and $v$ vertices in common (besides the endpoints of $f$). 
Since $K,\bar\gamma$ are constants independent of $n$, we find that
\[
\E(Y) \le D\cdot \binom n2 p^2\cdot n^{-v+e/\lambda}\,.
\]
for some constant $D$. By Lemma \ref{L_strTWB}(1), $e \le \lambda v$, 
whence
$\E[Y] \le D \binom n2p^2$, and 
\begin{equation}\label{E_pairs_of_TWGs}
\P\left(Y>p\binom n2\varepsilon\right) \to 0\,,
\end{equation}
as $n\to\infty$.

Let $Z$ denote the number of edges $f\in K_n$ that can be added 
by the $K_r$-dynamics on $\Gnp$ by a non-tree WG 
(with cost $\kappa\ge1$). 
The proof of Proposition \ref{P_pcLBsharp} shows that for some constant
$\beta>0$ the following holds: 
\begin{enumerate}
\item The expected number of edges that are 
added by a non-tree WG of size at most $\beta \log n$ is at most
${n\choose2} p^{1+o(1)}$. 
\item The probability that there exists an edge $f\in K_n$ that is 
added by a WG of size at least $\beta \log n$ tends to $0$ as $n\to \infty$.
\end{enumerate}
We conclude that 
\begin{equation}\label{E_num_of_WGs}
\P\left(Z>p\binom n2\varepsilon\right) \to 0\,,
\end{equation}
as $n\to\infty$.

Finally, we observe that 
\[
\left| |E(\langle \Gnp\rangle_{K_r})| - \sum_{k\le K}X_k\right| 
\le \sum_{k>K}X_k+Y+Z.
\]
Indeed, on the one hand, an edge can belong to 
$\langle \Gnp\rangle_{K_r}$ but not be a target of a small TWG 
only if it is the target of a large TWG or a non-tree WG. 
On the other hand, counting small TWGs may 
overestimate $\langle \Gnp\rangle_{K_r}$, 
but only by at most the number $Y$ of 
pairs of small TWGs sharing a target edge. 
Combining this observation with 
\eqref{E_rhoK}
and \eqref{E_smallTWG}--\eqref{E_num_of_WGs},
we find that $|E(\langle \Gnp\rangle_{K_r})|$ 
deviates by less than $(3\varepsilon+\sqrt\varepsilon)p\binom n2$ 
from $\rho\cdot p\binom n2$ with probability 
at least $1-\sqrt\varepsilon-o(1)$, and the proof is complete.
\end{proof}

\section{Final remarks}

Many fascinating problems 
remain open.

\subsection{Near-critical structure}
\label{S_ncs}

Despite locating the first-order term of $p_c$,  
explaining ``how" exactly percolation occurs remains open.
Related to this, 
it would be interesting to investigate the 
critical window for $K_r$-percolation. 
For example, what is the second-order term in $p_c(n,K_r)$? 
In addition, simulations suggest that if the edges of $K_n$ are 
added one-by-one in the so-called evolution of random graphs 
then there is, with high probability, a step $m$ such that 
$\langle G(n,m)\rangle_{K_r}$ has density 
$O(n^{-1/\lambda})$ but $\langle G(n,m+1)\rangle_{K_r}=K_n$. 

Let us also recall the discussion 
surrounding \eqref{E_L} above
regarding a potential ``critical 
droplet'' of order $L$ for $K_r$-percolation. 
It would appear that such a graph should
have the peculiar property that it has no percolating proper subgraphs
of size $k\gg1$ (cf.\ the discussion in \cite{BBM12} at the top of p.\ 432 therein).
We note, quite intriguingly, that if the $K_r$-dynamics were to somehow create
a large enough clique of order
$L$ then, by the work of Janson, \L uczak, Turova and Vallier
\cite{JLTV12} (see Theorem 3.1 and compare (3.1) therein with  
\eqref{E_L} above),  
``nucleation'' (that is, the $(r-2)$-neighbor dynamics) would then 
take over 
and lead to full percolation.

Finally, let us mention that another natural direction for 
future study is the speed (number of rounds) of the $K_r$-dynamics on $\Gnp$. 
Such questions might be the most interesting when $p$ is barely supercritical 
(but can be asked more generally), as the speed will be intimately related to the 
(still largely mysterious) way in which $\Gnp$ percolates. 
We note that the speed of the $r$-neighbor dynamics 
on $\Gnp$ is studied in \cite{JLTV12}.

\subsection{Slow convergence}
\label{S_slow}
The 
simulations in Figure \ref{F_pc} above 
might give the impression   
that $K_5$-percolation 
spreads by nucleation, 
however, this is not be  
representative of the behavior as $n\to\infty$; let us explain. 
In these simulations, 
$n=2000$
and $r=5$, and so $L$
in \eqref{E_L} is only about 
$L\approx 1.6$. In this figure we are, in fact,  mostly  seeing
the nucleating $3$-neighbor dynamics. 

Bootstrap percolation is notoriously
difficult to simulate. For instance, as discussed in \cite{Hol03,HT24}, 
conjectures in the physics literature for the 
constant in $p_c\sim (\pi^2/18)/\log n$ for the $2$-neighbor dynamics 
on $[n]^2$ were far from the true value. 
That being said, simulations with only $n=300$ (see, e.g., 
\cite{Hol05}) already show the correct qualitative super-critical
behavior of the model: a rectangle grows by nucleation 
(its perimeter slowing expanding) 
until it reaches a critical size (a ``critical droplet''), 
and then the process explodes.

In this respect, $K_r$-percolation, with $r\ge5$, seems 
to be even more impervious to simulation: We expect 
nucleation (the $(r-2)$-neighbor dynamics) to take over 
once a clique of size order $L$ is formed. However, 
it appears that $n$ would need to be extremely large 
(perhaps well outside any computationally feasible range) 
in order for $L$ to be large enough to provide some  
clues about how 
$\Gnp$ manages to form this ``critical droplet'' of size $L$. 

This is perhaps reminiscent of the {\it slow aging} 
phenomenon in statistical
physics \cite{HP10}, where a system converges to its 
equilibrium very slowly, and at various 
intermediate time scales can exhibit different types of behavior. 
In our current context, 
$K_r$-percolation, with $r\ge5$, and $K_4$-percolation 
appear to behave similarly (the $(r-2)$-neighbor dynamics dominate) 
for reasonably large values 
of $n$, but nonetheless, their behaviors diverge quite drastically
in the large $n$ limit.

\subsection{Strictly balanced and general graph templates}

Our work focuses on the case $H=K_r$, but many questions 
about $H$-percolation remain open for general graphs $H$. 
Most prominently, it would 
be interesting to determine $p_c(n,H)$ 
and whether the $H$-percolation threshold is sharp.

Cliques $K_r$ are balanced, but only strictly balanced once $r\ge 5$, 
and this property plays a central role in our proof of Theorem \ref{T_pc}. 
Hence, it seems plausible that 
$p_c=\Theta(n^{-1/\lambda})$ for {\it all} 
strictly-balanced graphs $H$. The upper bound 
$p_c=O(n^{-1/\lambda})$ for such $H$ is proved in \cite{BK24}. 
Moreover, it is possible that our proof can be adapted to all 
strictly-balanced graphs
Indeed, our proof of the sharp upper bound in 
Proposition \ref{P_pcUB}
does not use the fact that 
$H$ is a clique in any  crucial way 
(only that $K_r$, $r\ge5$, is strictly balanced), 
apart from the fact that the 
exponential growth rate $\gamma$ of TWGs
has a convenient form in this case. Our proof 
of the sharp lower bound 
in Proposition \ref{S_sharpLB}, on the other hand, 
uses the fact that $H$ is a clique in a number of places. 
However, we note that, 
in our proof of Proposition \ref{P_pcLBsharp}, 
there is an extra factor $p^{\xi\kappa}$ that is not fully used. 
In fact, we only used here that $\xi\ge0$, although in fact 
$\xi>0$ by Lemma \ref{L_chikappa}. 
Indeed, the only place in our proof that $\xi>0$ 
plays a role is in ruling out very costly witness graphs. 
It could be that this extra factor  $p^{\xi\kappa}$ 
will be helpful in analyzing general strictly balanced graphs $H$.

\subsection{Rigidity of witness graphs}\label{S_rigid}
The connection between $K_r$-percolation and the 
algebraic notion of generic $(r-2)$-rigidity theory 
goes back to Kalai \cite {Kal84}. 
A similar connection between volume-rigidity 
and hypergraph bootstrap percolation was used in \cite {LP23} 
to obtain a lower bound for the corresponding $p_c$. 
Hence, it seemed plausible 
to us, at the outset of this work, that 
rigidity theory would be helpful with studying 
$K_r$-percolation, but, in the end, we were not 
able to make progress with this approach. 
For instance, we are unable to prove the following stronger version of the 
spread lemma (Lemma \ref{L_SprLem}): Fix $r\ge 5$. 
Let $T$ be a $K_r$-tree, $G$ be a graph with $|V(G)\cap V(T)|\le x$, 
and $C$ denote the closure in the generic $(r-2)$-rigidity matroid of $G\cup T$. 
Then, is it true that at most $O(x)$ vertices of $T$ have a neighbor 
in $C$ that was not already their neighbor in $T$? This is stronger than the 
spread lemma since $C\supseteq\langle G\cup T\rangle_{K_r}$, 
which follows from \cite{Kal84}.

\section*{Acknowledgements}

ZB was supported by 
National Research, Development and Innovation Office (NKFIH) 
grant KKP no.\ 138270. BK was partially supported by a 
Florence Nightingale Bicentennial Fellowship (Oxford Statistics)
and a Senior Demyship (Magdalen College).
GK was supported by the 
Royal Commission for the Exhibition of 1851. 
YP was partially supported by the 
Israel Science Foundation grant ISF-3464/24. 
Part of this project was completed during a visit  
to Magdalen College in July 2024. 
We thank the college for its hospitality
and wonderful working environment.

\makeatletter
\renewcommand\@biblabel[1]{#1.}
\makeatother

\providecommand{\bysame}{\leavevmode\hbox to3em{\hrulefill}\thinspace}
\providecommand{\MR}{\relax\ifhmode\unskip\space\fi MR }
\providecommand{\MRhref}[2]{%
  \href{http://www.ams.org/mathscinet-getitem?mr=#1}{#2}
}
\providecommand{\href}[2]{#2}


\begin{thebibliography}{10}

\bibitem{AL88}
M.~Aizenman and J.~L. Lebowitz, \emph{Metastability effects in bootstrap
  percolation}, J. Phys. A \textbf{21} (1988), no.~19, 3801--3813.

\bibitem{Alo85}
N.~Alon, \emph{An extremal problem for sets with applications to graph theory},
  J. Combin. Theory Ser. A \textbf{40} (1985), no.~1, 82--89.

\bibitem{AK18}
O.~Angel and B.~Kolesnik, \emph{Sharp thresholds for contagious sets in random
  graphs}, Ann. Appl. Probab. \textbf{28} (2018), no.~2, 1052--1098.

\bibitem{AK21}
\bysame, \emph{Large deviations for subcritical bootstrap percolation on the
  {E}rd{\H o}s--{R}\'{e}nyi graph}, J. Stat. Phys. \textbf{185} (2021), no.~2,
  Paper No. 8, 16.

\bibitem{antonir2025does}
A.~C. Antonir, Y.~Peled, A.~Shapira, M.~Tyomkyn, and M.~Zhukovskii, \emph{When
  does a tree activate the random graph?}, preprint available at
  \href{https://arxiv.org/abs/2507.05697}{arXiv:2507.05697}.

\bibitem{AVZ25}
M.~Akhmejanova, I.~Vorobyev, and M.~Zhukovskii. 
\emph{Weak saturation numbers of large complete bipartite graphs}, preprint available at
  \href{https://arxiv.org/abs/2508.19435}{arXiv:2508.19435}.

\bibitem{BBMS25}
P.~Balister, B.~Bollob\'{a}s, R.~Morris, and P.~Smith, \emph{Universality for
  monotone cellular automata}, J. Eur. Math. Soc. (2025), in press.

\bibitem{BBDCM12}
J.~Balogh, B.~Bollob\'{a}s, H.~Duminil-Copin, and R.~Morris, \emph{The sharp
  threshold for bootstrap percolation in all dimensions}, Trans. Amer. Math.
  Soc. \textbf{364} (2012), no.~5, 2667--2701.

\bibitem{BBM12}
J.~Balogh, B.~Bollob\'{a}s, and R.~Morris, \emph{Graph bootstrap percolation},
  Random Structures Algorithms \textbf{41} (2012), no.~4, 413--440.

\bibitem{BBMR12}
J.~Balogh, B.~Bollob\'{a}s, R.~Morris, and O.~Riordan, \emph{Linear algebra and
  bootstrap percolation}, J. Combin. Theory Ser. A \textbf{119} (2012), no.~6,
  1328--1335.

\bibitem{BKPS19}
J.~Balogh, G.~Kronenberg, A.~Pokrovskiy, and T.~Szab\'{o}, 
\emph{The maximum length of $K_r$-bootstrap percolation}, 
Proc. Amer. Math. Soc., to appear, 
preprint available at
  \href{https://arxiv.org/abs/2507.1907.04559}{arXiv:2507.1907.04559}.

\bibitem{BK24}
Z.~Bartha and B.~Kolesnik, \emph{Weakly saturated random graphs}, Random
  Structures Algorithms \textbf{65} (2024), no.~1, 131--148.

\bibitem{BKK24}
Z.~Bartha, B.~Kolesnik, and G.~Kronenberg, \emph{{$H$}-percolation with a
  random {$H$}}, Electron. Commun. Probab. \textbf{29} (2024), Paper No. 33, 5 pp. 

\bibitem{BC22}
E.~Bayraktar and S.~Chakraborty, \emph{{$K_{r,s}$} graph bootstrap
  percolation}, Electron. J. Combin. \textbf{29} (2022), no.~1, Paper No. 1.46,
  17.

\bibitem{BMTR21}
M.~Bidgoli, A.~Mohammadian, and B.~Tayfeh-Rezaie, \emph{On
  {$K_{2,t}$}-bootstrap percolation}, Graphs Combin. \textbf{37} (2021), no.~3,
  731--741.

\bibitem{BMTRZ24}
M.~Bidgoli, A.~Mohammadian, B.~Tayfeh-Rezaie, and M.~Zhukovskii,
  \emph{Threshold for stability of weak saturation}, J. Graph Theory
  \textbf{106} (2024), no.~3, 474--495.

\bibitem{Bol68}
B.~Bollob\'{a}s, \emph{Weakly {$k$}-saturated graphs}, Beitr\"{a}ge zur
  {G}raphentheorie ({K}olloquium, {M}anebach, 1967), B. G. Teubner
  Verlagsgesellschaft, Leipzig, 1968, pp.~25--31.

\bibitem{BPRS17}
B.~Bollob\'{a}s, M.~Przykucki, O.~Riordan, and J.~Sahasrabudhe, \emph{On the
  maximum running time in graph bootstrap percolation}, Electron. J. Combin.
  \textbf{24} (2017), no.~2, Paper No. 2.16, 20 pp. 

\bibitem{CC99}
R.~Cerf and E.~N.~M. Cirillo, \emph{Finite size scaling in three-dimensional
  bootstrap percolation}, Ann. Probab. \textbf{27} (1999), no.~4, 1837--1850.

\bibitem{CLR79}
J.~Chalupa, P.~L. Leath, and G.~R. Reich, \emph{Bootstrap percolation on a
  {B}ethe lattice}, J. Phys. C \textbf{21} (1979), L31--L35.

\bibitem{DHKSZ24}
S.~Diskin, I.~Hoshen, D.~Kor{\'a}ndi, B.~Sudakov, and M.~Zhukovskii,
  \emph{Saturation in random hypergraphs}, preprint available at
  \href{https://arxiv.org/abs/2405.03061}{arXiv:2405.03061}.

\bibitem{ER59}
P.~Erd{\H{o}}s and A.~R{\'e}nyi, \emph{On random graphs. {I}}, Publ. Math.
  Debrecen \textbf{6} (1959), 290--297.

\bibitem{EJKL24}
A.~Espuny~D\'{\i}az, B.~Janzer, G.~Kronenberg, and J.~Lada, \emph{Long running
  times for hypergraph bootstrap percolation}, European J. Combin. \textbf{115}
  (2024), Paper No. 103783, 18 pp. 

\bibitem{FGJ13}
R.~J.~Faudree, R.~J.~Gould, and M.~S.~Jacobson, 
\emph{Weak saturation numbers for sparse graphs}, 
Discuss. Math., Graph Theory \textbf{33} (2013), no.~4, 677--693.

\bibitem{FMS25}
D.~Fabian, P.~Morris, and T.~Szab\'{o}, \emph{Slow graph bootstrap percolation
  {II}: {A}ccelerating properties}, J. Combin. Theory Ser. B \textbf{172}
  (2025), 44--79.

\bibitem{FKR17}
U.~Feige, M.~Krivelevich, and D.~Reichman, 
\emph{Contagious sets in random graphs}, 
Ann. Appl. Probab. \textbf{27} (2017), no.~5, 2675--2697.

\bibitem{Fra82}
P.~Frankl, \emph{An extremal problem for two families of sets}, European J.
  Combin. \textbf{3} (1982), no.~2, 125--127.

\bibitem{Gar70}
M.~Gardner, \emph{Mathematical games: The fantastic combinations of {J}ohn
  {C}onway's new solitaire game ``life''}, Sci. Am. \textbf{223} (1970), no.~4,
  120--123.

\bibitem{GKP17}
K.~Gunderson, S.~Koch, and M.~Przykucki, \emph{The time of graph bootstrap
  percolation}, Random Structures Algorithms \textbf{51} (2017), no.~1,
  143--168.

\bibitem{HL24}
I.~Hartarsky and L.~Lichev, \emph{The maximal running time of hypergraph
  bootstrap percolation}, SIAM J. Discrete Math. \textbf{38} (2024), no.~2,
  1462--1471.

\bibitem{HT24}
I.~Hartarsky and A.~Teixeira, \emph{Bootstrap percolation is local}, preprint
  available at \href{https://arxiv.org/abs/2404.07903}{arXiv:2404.07903}.

\bibitem{HP10}
M.~Henkel and M.~Pleimling, \emph{Non-equilibrium phase transitions. {V}olume
  2: {A}geing and dynamical scaling far from equilibrium}, Theoretical and
  Mathematical Physics, Springer, Dordrecht, 2010.

\bibitem{Hol05}
A.~E. Holroyd, \emph{Bootstrap percolation},
  \href{https://personal.math.ubc.ca/~holroyd/boot/}{https://personal.math.ubc.ca/\textasciitilde
  holroyd/boot/}, webpage accessed on 2025-10-26.

\bibitem{Hol03}
\bysame, \emph{Sharp metastability threshold for two-dimensional bootstrap
  percolation}, Probab. Theory Related Fields \textbf{125} (2003), no.~2,
  195--224.

\bibitem{JLTV12}
S.~Janson, T.~{\L}uczak, T.~Turova, and T.~Vallier, \emph{Bootstrap percolation
  on the random graph {$G_{n,p}$}}, Ann. Appl. Probab. \textbf{22} (2012),
  no.~5, 1989--2047.

\bibitem{Kal84}
G.~Kalai, \emph{Weakly saturated graphs are rigid}, Convexity and graph theory
  ({J}erusalem, 1981), North-Holland Math. Stud., vol.~87, North-Holland,
  Amsterdam, 1984, pp.~189--190.

\bibitem{Kal85}
\bysame, \emph{Hyperconnectivity of graphs}, Graphs Combin. \textbf{1} (1985),
  no.~1, 65--79.

\bibitem{Kol22}
B.~Kolesnik, \emph{The sharp {$K_{4}$}-percolation threshold on the {E}rd{\H
  o}s--{R}\'{e}nyi random graph}, Electron. J. Probab. \textbf{27} (2022),
  Paper No. 13, 23 pp.

\bibitem{KS17}
D.~Kor\'{a}ndi and B.~Sudakov, \emph{Saturation in random graphs}, Random
  Structures Algorithms \textbf{51} (2017), no.~1, 169--181.

\bibitem{KMM21}
G.~Kronenberg, T.~Martins, and  N.~Morrison, 
\emph{Weak saturation numbers of complete bipartite graphs in the clique}, 
J. Combin. Theory Ser. A \textbf{178} (2021), Paper No. 105357, 15 pp.

\bibitem{Lov77}
L.~Lov\'{a}sz, \emph{Flats in matroids and geometric graphs}, Combinatorial
  surveys ({P}roc. {S}ixth {B}ritish {C}ombinatorial {C}onf., {R}oyal
  {H}olloway {C}oll., {E}gham, 1977), Academic Press, London-New York, 1977,
  pp.~45--86.

\bibitem{LP23}
E.~Lubetzky and Y.~Peled, \emph{The threshold for stacked triangulations}, Int.
  Math. Res. Not. IMRN (2023), no.~19, 16296--16335.

\bibitem{Mor17}
R.~Morris, \emph{Monotone cellular automata}, Surveys in combinatorics 2017,
  London Math. Soc. Lecture Note Ser., vol. 440, Cambridge Univ. Press,
  Cambridge, 2017, pp.~312--371.

\bibitem{NR23}
J.~A. Noel and A.~Ranganathan, \emph{On the running time of hypergraph
  bootstrap percolation}, Electron. J. Combin. \textbf{30} (2023), no.~2, Paper
  No. 2.46, 20.

\bibitem{Rie12}
E.~Riedl, \emph{Largest and smallest minimal percolating sets in trees},
  Electron. J. Combin. \textbf{19} (2012), no.~1, Paper 64, 18.

\bibitem{Sch92}
R.~H. Schonmann, \emph{On the behavior of some cellular automata related to
  bootstrap percolation}, Ann. Probab. \textbf{20} (1992), no.~1, 174--193.

\bibitem{ST23}
A.~Shapira and M.~Tyomkyn, \emph{Weakly saturated hypergraphs and a conjecture of Tuza}, 
Proc. Amer. Math. Soc. \textbf{151} (2023), no.~07, 2795--2805.

\bibitem{Ter2025}
N.~Terekhov, \emph{A short proof of Tuza's conjecture for weak saturation in hypergraphs}, 
preprint available at
  \href{https://arxiv.org/abs/2504.03816}{arXiv:2504.03816}.

\bibitem{TGL19}
G.~L.~Torrisi, M.~Garetto, and E.~Leonardi, 
\emph{A large deviation approach to super-critical bootstrap percolation on the random graph {$G_{n,p}$}},
  Stochastic Process. Appl. \textbf{129} (2019), no.~6, 1873--1902.

\bibitem{Tu1992}
T.~Tuza, \emph{Asymptotic growth of sparse saturated structures is locally determined}, 
Discrete Math. \textbf{108} (1992), no.~1-3, 397--402.

\bibitem{Ula52}
S.~Ulam, \emph{Random processes and transformations}, Proceedings of the
  {I}nternational {C}ongress of {M}athematicians, {C}ambridge, {M}ass., 1950,
  vol. 2, Amer. Math. Soc., Providence, RI, 1952, pp.~264--275.

\bibitem{vanEnt87}
A.~C.~D. van Enter, \emph{Proof of {S}traley's argument for bootstrap
  percolation}, J. Statist. Phys. \textbf{48} (1987), no.~3-4, 943--945.

\bibitem{vonNeu66}
J.~von Neumann, \emph{Theory of self-reproducing automata}, University of
  Illinois Press, Champaign, IL, 1966.

\bibitem{Yu93}
J.~T. Yu, \emph{An extremal problem for sets: a new approach via {B}ezoutians},
  J. Combin. Theory Ser. A \textbf{62} (1993), no.~1, 170--175.

\end{thebibliography}
\end{document}